\documentclass[11pt,reqno]{amsart}
\usepackage{amsfonts, amsmath, amssymb, amsthm,bm,bbm}

\usepackage[foot]{amsaddr}
\usepackage{color}
\usepackage{atbegshi, picture}
\usepackage{stmaryrd}
\usepackage{graphicx}
\usepackage{tcolorbox}
\tcbuselibrary{breakable}
\usepackage[pagebackref, colorlinks = true, linkcolor = blue, urlcolor  = brown, citecolor = red]{hyperref}
\usepackage[margin=1in, marginparwidth=2cm]{geometry}
\usepackage{appendix}

\AtBeginShipoutNext{\AtBeginShipoutUpperLeft{
  \put(\dimexpr\paperwidth-0.5cm\relax,-0.8cm){\makebox[0pt][r]{\Small YITP-18-91}}
}}

\DeclareMathOperator{\Tr}{Tr}

\newcommand{\nocontentsline}[3]{}
\newcommand{\tocless}[2]{\bgroup\let\addcontentsline=\nocontentsline#1{#2}\egroup}

\newcommand{\be}{\begin{equation}}
\newcommand{\ee}{\end{equation}}
\newcommand{\bea}{\begin{eqnarray}}
\newcommand{\eea}{\end{eqnarray}}

\newcommand{\cB}{\mathcal{B}}

\newcommand{\cI}{\mathcal{I}}
\newcommand{\cJ}{\mathcal{J}}
\newcommand{\cK}{\mathcal{K}}

\newcommand{\cM}{\mathcal{M}}

\newcommand{\cT}{\mathcal{T}}

\newcommand{\cW}{\mathcal{W}}

\newcommand{\bff}{\mathbf{f}}

 \newcommand{\alt}{\mathsf{alt}}
 
  \newcommand{\pure}{\mathsf{pure}}
\newcommand{\NJ}{N_J}

\newcommand{\bT}{\mathbb{T}}

\newcommand{\bM}{\mathbb{M}}

\newcommand{\bE}{\mathbb{E}}
\newcommand{\bC}{\mathbb{C}}
\newcommand{\bN}{\mathbb{N}}

\newcommand{\cS}{\mathcal{S}}

\newcommand{\un}{\mathbbm{1}}

\newcommand{\fsig}{\bm{\sigma}}

\newcommand{\Kc}{K}

\newtheorem{theorem}{Theorem}[section]

\newtheorem{lemma}[theorem]{Lemma}
\newtheorem{definition}[theorem]{Definition}
\newtheorem{notation}[theorem]{Notation}
\newtheorem*{theorem*}{Theorem}
\newtheorem*{theo31*}{Theorem \ref{thm:moments-fixed-N} in the maps formulation}
\newtheorem{corollary}[theorem]{Corollary}
\newtheorem{proposition}[theorem]{Proposition}
\newtheorem{remark}[theorem]{Remark}

\begin{document}

\title[The joint distribution of marginals of multipartite random quantum states]{
	On the joint distribution of the marginals\\of multipartite random quantum states}
\date{\today}
\author{{Stephane~Dartois$^1$}}
\address{$^1$School of Mathematics and Statistics, University of Melbourne, Parkville, Australia}
\email{stephane.dartois@unimelb.edu.au}

\author{{Luca~Lionni$^{2,3}$}}
\address{$^{2}$Yukawa Institute for Theoretical Physics, Kyoto University, Japan}
\email{luca.lionni@yukawa.kyoto-u.ac.jp}
\address{$^{3}$Departement of Physics, Chulalongkorn University, Bangkok, Thailand}

\author{{Ion~Nechita$^4$}}
\address{$^4$Laboratoire de Physique Th\'eorique, Universit\'e de Toulouse, CNRS, UPS, France}
\email{nechita@irsamc.ups-tlse.fr}

\begin{abstract}
We study the joint distribution of the set of all marginals of a random Wishart matrix acting on a tensor product Hilbert space. We compute the limiting free mixed cumulants of the marginals, and we show that in the balanced asymptotical regime, the marginals are asymptotically free. We connect the matrix integrals relevant to the study of operators on tensor product spaces with the corresponding classes of combinatorial maps, for which we develop the combinatorial machinery necessary for the asymptotic study. Finally, we present some applications to the theory of random quantum states in quantum information theory.
\end{abstract}


\maketitle

\tableofcontents

\newpage

\section{Introduction}

The Wishart ensemble was historically the first probability distribution on matrices which was studied \cite{wishart1928generalised}. The main motivation for Wishart was statistics; later, Wigner \cite{wigner1955characteristic} modelled complex, analytically intractable Hamiltonians in nuclear physics with Hermitian random matrices, and the field of Random Matrix Theory \cite{mehta2004random,anderson2010introduction} was born. Nowadays, random matrices play a significant role in many sub-fields of mathematics, such as  operator algebras, combinatorics and algebraic geometry, integrable systems and partial differential equations, as well as in other disciplines such as theoretical physics or telecommunication. In this paper, we are motivated by a recent application of random matrices to Quantum Information Theory \cite{nielsen2010quantum}. 

The mathematical formalism of Quantum Information Theory is constructed upon the central notion of \emph{quantum states}, also known in the physical literature as density matrices. These are positive semidefinite $d \times d$ complex matrices, normalized to have unit trace; here, $d$ is the number of degrees of freedom of the quantum system under consideration. Density matrices model ``open quantum systems'', that is quantum systems which interact with an environment (which, most of the times, is too complicated to be taken under consideration). Isolated systems, (which are called ``closed'') are modeled traditionally by unit vectors in $\mathbb C^d$, which we choose to identify with the rank-one projections on the corresponding vector space; these projections are the extremal points of the convex set of density matrices. 

One might want to study \emph{random quantum states} for several different reasons. Foremost, we would like to understand what are the typical mathematical (or physical, or even information-theoretical) properties of a \emph{typical} state, where typical should be understood as randomly distributed with respect to some natural (or physically relevant) probability distribution. Another reason one would like to understand random density matrices comes from the empirical observation that, in situations where explicit examples satisfying some desired properties are hard to come by, one should simply pick the sought-for object at random; in many cases, with large probability, the random sample will have the desired properties. Random quantum states (and random quantum channels) have been a valuable source of (counter-)examples in Quantum Information Theory (see, e.g.~the recent review paper \cite{collins2016random}).

There is a large literature on random density matrices and their applications to quantum information theory, most of it focusing on spectral properties of one random matrix. In particular, the focus was on random states of single quantum systems and bipartite quantum systems (mostly related to the study of entanglement). In this work, we tackle a fundamentally different question:

\medskip

\emph{Given a multipartite random density operator, what is the \emph{joint probability distribution} of its different marginals?}

\medskip

Recall that, for a quantum state of a multi-partite system, the marginal of a subset of systems is the partial trace with respect to the complementary set of systems. We shall study the question above in different settings, of increasing generality, first with just 4-partite states, and then for quantum states with arbitrarily many subsystems. Depending on the relative sizes of the subsystems and on their rates of asymptotic growth, we shall exhibit two types of behavior. In one situation, where the growth rates of the system dimensions are the same, we prove that the whole set of (balanced) marginals are \emph{asymptotically free}, meaning, in broad terms, that they behave like independent random matrices (although they might share one or more subsystems). In a different regime, where the dimensions of some of the subsystems are being kept fixed, we do not have asymptotic freeness, but we provide exact formulas for the limiting joint free cumulants, in terms of the types of marginals involved. We state next, informally, two of the main results of this paper, corresponding to the asymptotic regimes described above; for the more general and precise results, we refer the reader to Theorem~\ref{thm:asymptotic-freeness-balanced-general} and, respectively, Theorem~\ref{thm:2-marginals-fixed-BC}. 

\begin{theorem*}
	Let $\rho^{(N)} \in \mathcal M_{N^{2r}}(\mathbb C)$ be a random pure state of $2r$-partite quantum system, where each subsystem is $N$-dimensional. Then, the $\binom{2r}{r}$ marginals $\rho_S^{(N)}$, with $S \subseteq \{1,2, \ldots, 2r\}$ a set of cardinality $r$, are \emph{asymptotically free}: their (rescaled) joint distribution converges in moments to that of $\binom{2r}{r}$ free Mar{\v{c}}enko-Pastur elements. Equivalently, their limiting joint distribution is the same as that of $\binom{2r}{r}$ \emph{independent copies} of, say, $\rho_{\{1,2,\ldots, r\}}^{(N)}$.
\end{theorem*}

In the unbalanced case, we have, informally, for 4-partite systems, the following result. 

\begin{theorem*}
	Let $\rho_N \in \mathcal M_{Nm^2M}(\mathbb C)$ be a random pure state of $4$-partite quantum system $\mathcal H_{ABCD}$, where $\dim \mathcal H_A= N$,  $\dim \mathcal H_B = \dim \mathcal H_C= m$ and $\dim \mathcal H_D= M \sim cN$. Assuming $m,c$ are fixed constants, the (rescaled) joint distribution of the marginals $(\rho_{AB}, \rho_{BC})$	converges, in moments, 
	as $N \to \infty$, to a pair of non-commutative random variables $(x_{AB},x_{AC})$ having the following free cumulants:
	$$\kappa(x_{f(1)}, x_{f(2)}, \ldots, x_{f(p)}) = cm^{-\alt(f)},$$
	where $f \in \{AB, AC\}^p$ is an arbitrary word in the letters $AB, AC$, and $\alt(f)$ is the number of different consecutive values of $f$, counted cyclically:
	$$\alt(f) := |\{a \, : \, f(a) \neq f(a+1)\}|,$$
	where $f(p+1):=f(1)$.
\end{theorem*}

The contribution of our paper is threefold. First, we introduce, in full detail, the notions of \emph{combinatorial maps} relevant for the random matrix computations we perform, and we extend them to the case of matrices having a {tensor product structure}. Our presentation starts at a basic level, gradually adding layers of complexity, and can be used by readers with a quantum information background as an introduction to the subject. We develop the necessary combinatorial techniques to deal with the types of maps appearing in our study (combinatorial maps with vertices of two colors and edges of $2r$ colors). Secondly, we contribute to the theory of random matrices and free probability by computing the limiting distribution of a family of random matrices obtained as marginals of a unique random object. Although, globally, the random matrix model is standard (Wishart matrices), taking (intersecting) marginals (i.e.~partial traces) and considering their joint distribution is new; for this reason, in order to emphasize the importance of the tensor product structure of the Hilbert space, we shall call the random matrices we study \emph{Wishart tensors}. We prove asymptotic freeness in the balanced case, in a very general setting, and obtain the limiting free cumulants in the unbalanced setting; the explicit form of the free cumulants (see the second informal theorem above) is very interesting, involving a parameter counting the number of different consecutive letters appearing in the respective word. Thirdly, from the point of view of quantum information theory, our study shows that the marginals of a random pure quantum state behave independently in the balanced case and in the large $N$ limit: the moment statistics of the whole set of marginals are the same as if the marginals were independent. The situation is different in the unbalanced case: there is a strong correlation between, say, the marginals $\rho_{AB}$ and $\rho_{AC}$ of a pure random 4-partite quantum state $\psi_{ABCD}$ when $\dim \mathcal H_A \gg \dim \mathcal H_B = \dim \mathcal H_C$. 

Since the main focus of our paper is on random quantum states over Hilbert spaces with a tensor product structure, let us give now a brief survey of the literature on the subject, emphasizing the point of contact with our work. Physicists started working on ensembles of quantum states in the early '90s, when Page computed the average entropy of entanglement of a bipartite random pure state \cite{page1993average}. The study of probability measures induced by metrics in the one-party case was initiated by Hall in \cite{hall1998random} and developed by Sommers and {\.Z}yczkowski in \cite{zyczkowski2001induced,zyczkowski2003hilbert,sommers2003bures,sommers2004statistical}; see also \cite{osipov2010random} for the limiting eigenvalue distribution of the Bures ensemble. In the multi-partite case, the study of random tensors in quantum information theory was initiated \cite{ambainis2012random}, where superpositions of random product states were investigated. Later, specific models of randomness were studied in \cite{perez-garcia2006matrix} in the case of random matrix product states and in \cite{collins2010randoma,collins2013area} for random graph states. In \cite{christandl2014eigenvalue}, Christandl, Doran, Kousidis, and Walter studied the joint distribution of all the single particle marginals of a multi-partite quantum state in a very general setting, allowing for different distributions of the global state, and making use of the \emph{Duistermaat-Heckman measures} from Lie theory. The current work is, to our knowledge, the first instance where the question of the \emph{joint distribution of the possibly overlapping marginals} of a random quantum state is considered; we do so in the simplest framework, that of the \emph{Wishart ensemble}. In the framework of quantum information theory, this corresponds to considering a random pure quantum state on a multipartite Hilbert space, and tracing out some of the subsystems to obtain the marginals. 

\medskip

The paper is organized as follows. In Section \ref{sec:FirstSec} we recall some well-known results about Wishart matrices and random density matrices, which can be seen as marginals of bipartite Wishart tensors. This simple situation is also the occasion to introduce the machinery of combinatorial maps. We provide two proofs of the classical Mar{\v{c}}enko-Pastur theorem, one using the language of permutations and their metric properties, and another one using combinatorial maps; the reader can see from this example how the two approaches mirror each other. In Section \ref{sec:ABCD}, we study in full detail the case of 4-partite Wishart tensors. We consider two different asymptotical regimes: a balanced regime, where the dimension of all the spaces are equal, and an unbalanced regime, where two of the four spaces have fixed dimension. We compute the limiting joint distribution of the two 2-marginals in both regimes: in the balanced case, we show that the marginals are asymptotically free, while in the unbalanced case we compute the (non-trivial) limiting mixed free cumulants. Finally, in Section \ref{sec:general-multipartite}, we study the general multipartite case. In the balanced case, we show again that the marginals are asymptotically free; in the other asymptotic regimes, we only have partial results: we list the different regimes that take place, but leave their detailed description for future work.

\medskip

\noindent {\it Acknowledgments.} L.L. is a JSPS International Research Fellow. The work of S.D. was partially supported by the Australian Research Council grant DP170102028. I.N.'s research has been supported by the ANR projects {StoQ} (grant number ANR-14-CE25-0003-01) and {NEXT} (grant number ANR-10-LABX-0037-NEXT), and by the PHC Sakura program (grant number 38615VA). I.N.~also acknowledges the hospitality of the Technische Universit\"at M\"unchen. The authors would like to thank the Institut Henri Poincar{\'e} in Paris for its hospitality and for hosting the trimester on ``\href{https://sites.google.com/site/analysisqit2017/}{Analysis in Quantum Information Theory}'', during which part of this work was undertaken. S.D.~and I.N.~would also like to thank the organizers of the ``\href{https://crei.skoltech.ru/cas/calendar/conf170619/}{QUATR-17}'' conference in Skoltech/Moscow, and especially Leonid Chekhov, for bringing together researchers in random tensor theory and quantum information theory. 

\section{The limiting eigenvalue distribution of random density matrices}
\label{sec:FirstSec}

In this section, we discuss the different ensembles of random density matrices from the literature, focusing on the \emph{induced ensemble}, which will be the one we shall study in the later sections. We also compute the limiting eigenvalue distribution of (rescaled) random density matrices, using two different formalisms: an algebraic one, emphasizing the role of permutations, and a combinatorial one, featuring the theory of combinatorial maps. Although the two proofs given will be equivalent, we present both in full detail in order to introduce the main objects and to prepare the reader for the more complicated situations discussed in the later sections.

\subsection{Random density matrices}
To start, let us fix some notation. Density matrices with $N$ degrees of freedom are represented by unit trace, positive semidefinite $N \times N$ matrices:
$$\mathcal M_N^{1,+} := \{\rho \in \mathcal M_N(\mathbb C) \, : \, \rho \geq 0 \text{ and } \operatorname{Tr} \rho = 1\}.$$
This is a convex body, whose extreme points are rank one projections which we identify (up to a phase) with vectors $x \in \mathbb C^N$, $\|x\|=1$, called \emph{pure states}.

We consider first the canonical distribution on pure states, that is the Lebesgue measure on the unit sphere of $\mathbb C^N$. Integrating polynomials in the state's coordinates with respect to this measure is quite straightforward, see \textit{e.g.}~\cite{folland2013real}. The idea is to relate the spherical integral to a Gaussian one with the help of a change of variable to the polar coordinates, and then use Wick's (or Isserlis' \cite{isserlis1918formula}) formula to evaluate the Gaussian integral; we recall this result next. 

\begin{proposition}\label{prop:Wick}
Let $X_1, \ldots, X_k$ be a $k$-tuple of random variables having a joint (complex) Gaussian distribution. If $k$ is odd, then $\mathbb E[X_1 \cdots X_k] = 0$. If $k=2l$ is even, then
$$\mathbb E[X_1 \cdots X_k] = \sum_{\substack{p=\{\{i_1,j_1\},\ldots , \{i_l,j_l\}\} \\ \text{pairing of }\{1,\ldots ,k\} }} \quad \prod_{s=1}^l \mathbb E[X_{i_s} X_{j_s}].$$
\end{proposition}

Let us move now to ensembles on the whole set of $N \times N$ density matrices, $\mathcal M_N^{1,+}$. Again, there is a natural candidate here, the normalization of the Lebesgue measure on the ambient space. It turns out however, that this measure is just a specialization of a 1-parameter family of probability distributions, called the \emph{induced measures}. Introduced by {\.Z}yczkowski and Sommers in \cite{zyczkowski2001induced}, these measures have the advantage of being interesting and natural both from the physical and the mathematical perspectives. Let us start with some motivating consideration from quantum physics. Assume the physical system we are interested in (which has $N$ degrees of freedom) is not isolated, but coupled to an \emph{environment}, having $M$ degrees of freedom. In most physical applications, the environment is big and inaccessible, so we choose not to model it; in other words, if $\Psi \in \mathbb C^N \otimes \mathbb C^M$ is the (pure) quantum state describing jointly the system and the environment, we only have access to the state of the system
$$\rho = [\operatorname{id}_N \otimes \operatorname{Tr}_M](\Psi \Psi^*) \in \mathcal M_N^{1,+}.$$
In the equation above, we assume that the vector $\Psi$ is normalized, $\|\Psi\|=1$. The main idea of \cite{zyczkowski2001induced} is to consider $\Psi$ uniformly distributed on the unit sphere of $\mathbb C^N \otimes \mathbb C^M$; this leads to the following definition. 

\begin{definition}\label{def:induced-measure}
The \emph{induced measure of parameters $(N,M)$} is the image measure of the uniform probability distribution on the unit sphere of $\mathbb C^N \otimes \mathbb C^M$ through the map
\begin{align*}
\mathbb C^N \otimes \mathbb C^M & \to \mathcal M_N^{1,+}\\
\Psi & \mapsto [\operatorname{id}_N \otimes \operatorname{Tr}_M](\Psi \Psi^*),
\end{align*}
where $\Psi \Psi^*$ denotes the self-adjoint rank one projection on $\mathbb C \Psi$.
\end{definition}
Importantly, the uniform (Lebesgue) probability measure on the convex body $\mathcal M_N^{1,+}$ is exactly the induced measure with parameters $(N,N)$ \cite[Section 3]{zyczkowski2001induced}.

A detailed mathematical analysis of the induced measures defined above was performed in \cite{nechita2007asymptotics}, where it was emphasized that quantum states distributed along the induced measures are just \emph{normalized Wishart random matrices}. To make this observation more precise, let us briefly remind the reader the definition of the Wishart ensemble (we refer the reader to \cite[Section 4]{hiai2000semicircle}, \cite[Chapter 3]{bai2010spectral}, \cite{haagerup2003random} for detailed treatments of the Wishart ensemble from a random matrix theory perspective, and to the excellent book \cite[Section 6.2.3]{aubrun2017alice} for a quantum information theory point of view). Let $X \in \mathcal M_{N \times M}(\mathbb C)$ be a \emph{Ginibre} random matrix, that is a matrix having i.i.d.~complex standard Gaussian entries (no symmetry is assumed here). For completeness, we rewrite formally what we mean by standard i.i.d.~complex Gaussian entries. The entries of $X$ form a set of $NM$ complex numbers $x_{i,j}$ each with density
\begin{equation}
\frac1{2i\pi}e^{-|x_{i,j}|^2} \mathrm{d}\bar{x}_{i,j}\mathrm{d}x_{i,j},
\end{equation}
which rewrites in terms of the real part $r_{i,j}$ and imaginary part $s_{i,j}$ of $x_{i,j}$ as
\begin{equation}
\frac1{\pi}e^{-(r_{i,j}^2+s_{i,j}^2)}\mathrm{d}r_{i,j}\mathrm{d}s_{i,j}.
\end{equation}
A Wishart matrix of parameters $(N,M)$ is defined as $W = XX^*$. The relation to random density matrices from the induced ensemble has been made mathematically rigorous in \cite[Lemma 1]{nechita2007asymptotics}.

\begin{proposition}
Let $W$ be a random Wishart matrix of parameters $(N,M)$. Then,
\begin{equation}\label{eq:rho-from-W}
\rho = \frac{W}{\operatorname{Tr}W}
\end{equation}
is a random density matrix distributed along the induced measure from Definition~\ref{def:induced-measure} of parameter $(N,M)$. 
\end{proposition}

Although the normalization by the trace is a highly non-trivial (and non-linear) operation, in practice, for large random matrices, it does not pose technical difficulties. There are two reasons for this: first, in equation \eqref{eq:rho-from-W}, the trace $\operatorname{Tr} W$ and the normalized density matrix $\rho$ are independent random variables; this fact is similar to the result in classical probability which states that the norm and the direction of a (standard) Gaussian vector are independent random variables. The second reason which allows us to deal in a simple manner with the trace normalization is that the trace of a Wishart random matrix is a chi-squared random variable and thus concentrates very well around its average $\mathbb E \operatorname{Tr} W = NM$ (see \cite[Exercise 6.40]{aubrun2017alice})
\begin{equation}\label{eq:convergence-Tr-W}
\forall t>0, \qquad \mathbb P\left[ |\operatorname{Tr} W - NM| > tNM \right] \leq 2 \exp \left(- \frac{t^2NM}{2+4t/3} \right).
\end{equation}

\subsection{The limiting eigenvalue distribution}
We compute in this section the eigenvalue distribution of Wishart matrices, and thus of random density matrices, in the large $N$ limit. We shall work with the simpler model of Wishart matrices, and then translate the results to quantum states in Corollary \ref{cor:MP-convergence-random-density-matrices}. We shall present two proofs of the well-known convergence to the Mar{\v{c}}enko-Pastur distribution, one using permutations and the other one using combinatorial maps; this will be the occasion to introduce these two proof techniques and to familiarize the reader with the main objects appearing in the respective theories.\\

In the limit of large matrix dimension ($N \to \infty$), the behavior of Wishart matrices depends on the asymptotic ratio $M/N$. The most common situation is when $M/N \to c \in (0,\infty)$, in which case the matrix converges to the well-known Mar{\v{c}}enko-Pastur distribution \cite{marcenko1967distribution}. Although this result is straightforward and very well-known, we provide a self-contained proof in order to compare the approach of this section with the one in the next section. 

\begin{proposition}\label{prop:MP}
Let $W_N$ be a sequence of random Wishart matrices of parameters $(N,M_N)$, where $M_N$ is an integer sequence with the property that $M_N \sim cN$ as $N \to \infty$, where $c \in (0,\infty)$ is a constant. The sequence $N^{-1}W_N$ converges, in moments, towards the Mar{\v{c}}enko-Pastur distribution
$$\forall p \geq 1, \qquad \lim_{N \to \infty} \mathbb E \frac{1}{N} \operatorname{Tr}\left(\frac{W_N}{N}\right)^p = \int x^p \mathrm{d}\mathrm{MP}_c(x),$$
where 
\begin{equation}\label{eq:Marchenko-Pastur}
\mathrm{d}\mathrm{MP}_c=\max (1-c,0)\delta_0+\frac{\sqrt{(b-x)(x-a)}}{2\pi x} \; \mathbf{1}_{[a,b]}(x) \, \mathrm{d}x,
\end{equation}
with $a = (1-\sqrt c)^2$ and $b=(1+\sqrt c)^2$.
\end{proposition}
We plot the density of the  Mar{\v{c}}enko-Pastur distribution, along with Monte-Carlo simulations in Figure~\ref{fig:MP}; for other regimes, see \cite[Theorem 6.27]{aubrun2017alice}. The mass term for $c<1$ is easily explained by the fact that in this case $M<N$ and the rank of the $N\times N$ matrix $W$ is $M<N$ thus $W$ shall have $N-M\sim_{N\rightarrow \infty}N(1-c)$ null eigenvalues.

\begin{figure}[!ht]
	\centering
	\includegraphics[width=0.45\textwidth]{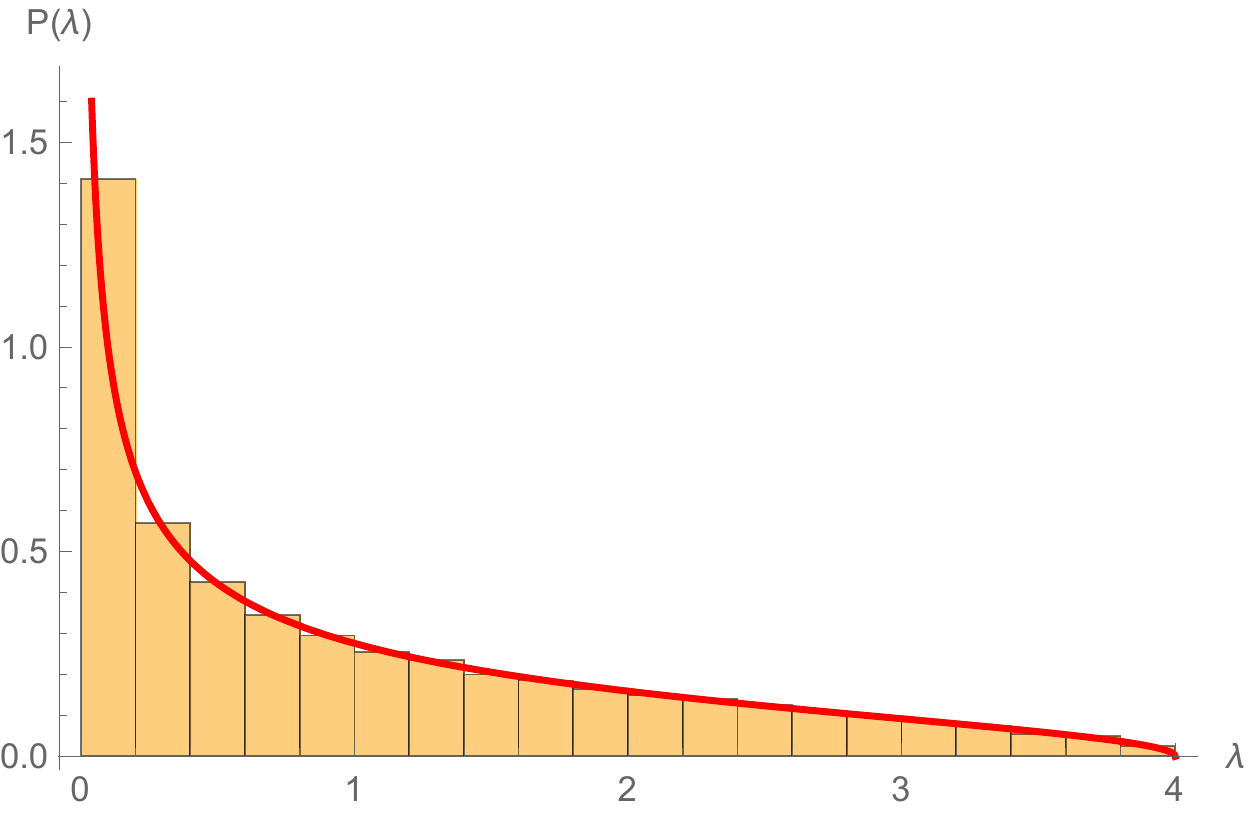} \qquad
	\includegraphics[width=0.45\textwidth]{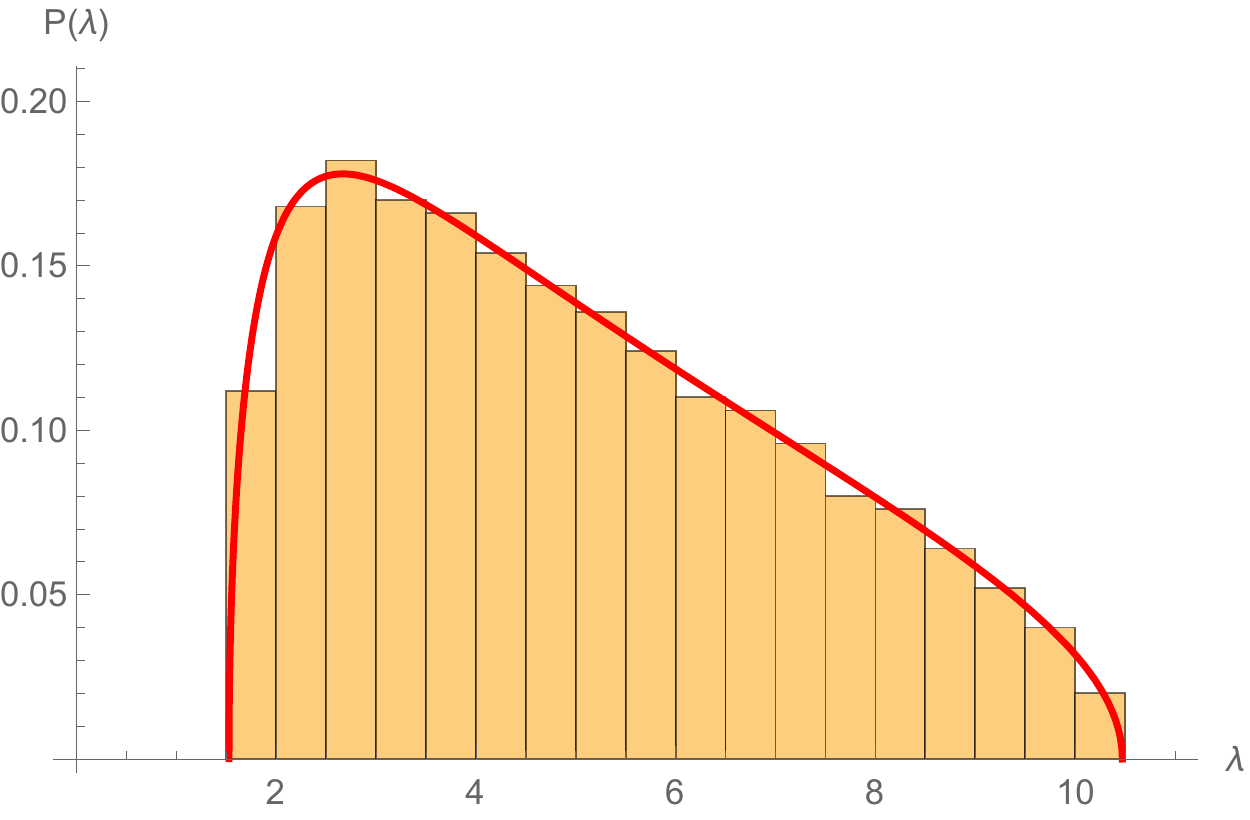}
	\caption{The density of the Mar{\v{c}}enko-Pastur distribution $\mathrm{MP}_c$ for $c=1$ (left) and $c=5$ (right).}
	\label{fig:MP}
\end{figure}

\begin{proof}
We shall use the method of moments, i.e.~compute  Gaussian expectations of the form 
	\be
	\mathbb E \Tr(W^{p})= \int_{(N,M)} \Tr\bigl((XX^*)^{p}\bigr) \mathrm{d}\mu(X),
	\ee
	with 
	$$\mathrm{d}\mu(X)=\frac1{(2i\pi)^{NM}}e^{-\Tr(XX^*)}\mathrm{d}X^*\mathrm{d}X, \ \textrm{ and } \ \mathrm{d}X^*\mathrm{d}X=\prod_{\substack{1\le i\le N \\ 1\le j\le M}}\mathrm{d}\bar x_{i,j}\mathrm{d}x_{i,j}.$$
The proof consists of three steps. First, we show the exact formula
\begin{equation}\label{eq:exact-formula-moment-bipartite}
\mathbb E N^{-1} \operatorname{Tr}\left[(N^{-1}W_N)^p\right] = \sum_{\alpha \in \mathcal S_p} N^{\#(\gamma\alpha)} M^{\#\alpha},
\end{equation}
where $\gamma \in \mathcal S_p$ is the full cycle permutation $\gamma = (1, 2, 3,  \ldots p)$  and $\#\alpha$ denotes the number of cycles of the permutation $\alpha$. Note that we dropped the dependence on $N$ of the parameter $M$, in order to keep the notation light; the reader should keep in mind that $M = M_N$ is a function of $N$ which grows as $M \sim cN$. The second step, Lemma~\ref{lemma:NonCrossingPart}, will consist in analyzing the dominating terms in \eqref{eq:exact-formula-moment-bipartite}. We will show that the surviving permutations are in bijection with non-crossing partitions, recovering the moments of the Mar{\v{c}}enko-Pastur distribution
$$\lim_{N \to \infty} \mathbb E N^{-1} \operatorname{Tr}\left[(N^{-1}W_N)^p\right] = \sum_{\alpha \in \mathrm{NC}(p)} c^{\#\alpha}.$$
Using Voiculescu's $R$-transform, we compute in a third step the Cauchy transform of the probability measure having the moments above, and then, by Stieltjes inversion, we recover the exact expression of the density \eqref{eq:Marchenko-Pastur}.

\

\noindent{\it First step. }To show \eqref{eq:exact-formula-moment-bipartite}, we need to perform the integration on the left-hand-side with the help of the Wick formula from Proposition~\ref{prop:Wick}. We are going to use a graphical reading of the Wick formula introduced in \cite{collins2011gaussianization}. In this framework, matrices (and more generally, tensors) are represented by boxes having decorations corresponding to the vector spaces the matrix is acting on. The decorations have two attributes: shape, distinguishing vector spaces of various dimensions, and shading, distinguishing primal (filled symbols) from dual (empty symbols) spaces. We depict in Figure~\ref{fig:moments-Wishart} (from left to right) the diagram of a Wishart matrix $W=XX^*$, then, in the center, the same diagram, with $X^*$ replaced by the transpose of $\bar X$, and then the diagram for the second moment $\operatorname{Tr}(W^2)$. 
\begin{figure}[!ht]
\raisebox{+0.5cm}{\includegraphics[scale=.8]{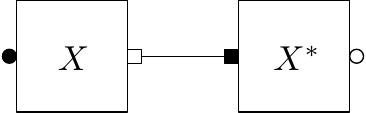} \qquad\qquad\includegraphics[scale=.8]{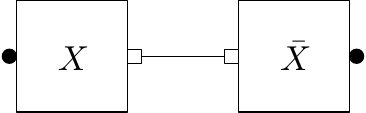} }\qquad\qquad\includegraphics[scale=.8]{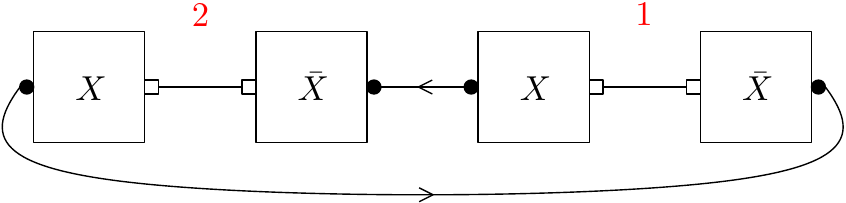} 
\caption{Diagrams for Wishart matrices. On the left, the diagram for $W=XX^*$, with $X \in \mathcal M_{N \times M}$. The Hilbert space $\mathbb C^N$ is depicted by round decorations, while $\mathbb C^M$ is depicted by square decorations. The tensor contraction between the two square decorations corresponds to the matrix product $X \cdot X^*$. In the center panel, we have the same diagram, after replacing $X^*$ by $(\bar X)^\top$; notice that taking the transposition amounts to inverting the shading of the decorations of the $\bar X$ box. In the last panel, we depict the scalar $\operatorname{Tr}(W^2) = \operatorname{Tr}(XX^*XX^*)$.}
\label{fig:moments-Wishart}
\end{figure}

In order to establish the formula \eqref{eq:exact-formula-moment-bipartite}, we need to apply the Wick formula to the quantity $\operatorname{Tr}(W^p)$, for an arbitrary $p \geq 1$. In \cite[Section 3.3]{collins2011gaussianization} (see also \cite[Section III.C]{collins2016random}), it has been shown that computing a Gaussian expectation can be done in a graphical way, as follows. Given a diagram containing boxes $X$ and $\bar X$ corresponding to random matrices (or tensors) with i.i.d.~standard complex Gaussian entries, if the number of $X$-boxes is different than the number of $\bar X$-boxes, the expectation (over the randomness in $X$) is zero. If one has, say, $p$ $X$-boxes and $p$ $\bar X$-boxes, the expectation of the diagram is a sum indexed by permutations $\alpha \in \mathcal S_p$, where the terms are obtained by deleting the $X$ and the $\bar X$ boxes, and connecting the corresponding attached decorations with the permutation $\alpha$: the decorations of the $i$-th $\bar X$-box are to be connected to the corresponding decorations of the $\alpha(i)$-th $X$ box. We would like to warn the reader at this point that the above convention is opposite from the one used in \cite[Section 3.3]{collins2011gaussianization}, where $\alpha$ was connecting $X$-boxes to $\bar X$-boxes. It turns out that the current convention makes the connection between the Wick graphical calculus and the theory of combinatorial maps more transparent, justifying our choice.

As an example, in our moment problem, the two diagrams appearing when computing $\mathbb E \operatorname{Tr}(XX^*XX^*)$ are depicted in Figure~\ref{fig:expectation-moment-2-Wishart}. Notice that the new diagrams are entirely made out of loops, so their (scalar) values are given by $N^aM^b$, where $a$ (resp.~$b$) is the number of loops corresponding to $\mathbb C^N$, that is to round decorations (resp.~$\mathbb C^M$, \textit{i.e.}~square decorations). 
\begin{figure}[!ht]
\includegraphics[scale=.8]{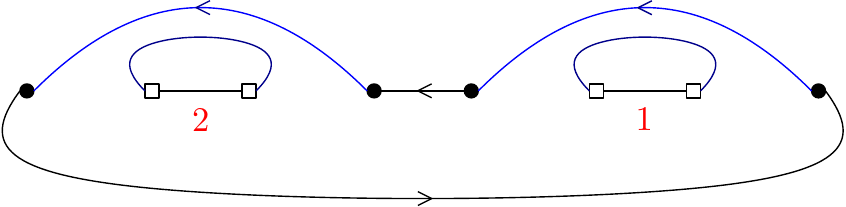} \qquad\qquad\includegraphics[scale=.8]{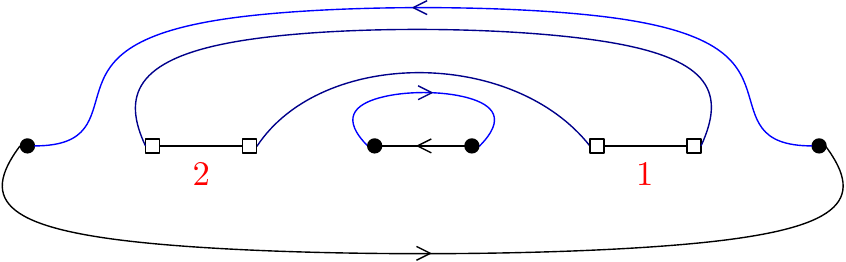} 
\caption{The two diagrams corresponding to the graphical application of the Wick formula for $\mathbb E \operatorname{Tr}(W^2)$, see Figure~\ref{fig:moments-Wishart}, rightmost panels. On the left, the diagram corresponding to the identity permutation (the blue wires connect the decorations of each $X$-box to  the corresponding ones of the $\bar X$-box directly following it). On the right is the diagram corresponding to the permutation $\alpha = (12)$. The two diagrams contain only loops, so their values are $NM^2$,respectively $N^2M$.}
\label{fig:expectation-moment-2-Wishart}
\end{figure}

Moving to the general case of an arbitrary $p$, we write $\mathbb E \operatorname{Tr}(W^p) = \sum_{\alpha \in \mathcal S_p} \mathcal D_\alpha$,
where $\mathcal D_\alpha$ is the diagram obtained by deleting the $p$ $X$- and $\bar X$-boxes and by connecting the corresponding decorations according to the permutation $\alpha$. It is clear that $\mathcal D_\alpha$ consists only of loops corresponding to the Hilbert spaces $\mathbb C^N$ and $\mathbb C^M$; it follows that in order to evaluate such a diagram, one has to count the number of loops of each type. Let us start by counting the loops attached to square decorations, each giving a contribution of $M$. Note that in the original diagram (before taking the expectation), the square decoration of the $i$-th $X$-box is connected to the square decoration of the $\bar X$-box belonging to the same, $i$-th, group. It is then easy to see that each distinct cycle of the permutation $\alpha$ gives rise to a loop, then the number of $M$-loops is $\#\alpha$, thus giving a total contribution of $M^{\#\alpha}$. The same reasoning can be applied when counting the contribution of loops attached to round decorations (which correspond to the Hilbert space $\mathbb C^N$), with one difference: in the initial wiring of the diagram (before taking the expectation), the $i$-th $X$-box is connected to the $(i-1)$-th $\bar X$-box (where the subtraction operation is understood cyclically, modulo $p$). In other words, the initial wiring is given by the full-cycle permutation $\gamma$, with $\gamma(i) = i+1$;  note that we are numbering the boxes $1, 2, \ldots, p$ from right to left. A similar combinatorial argument shows that the number of loops is, in this case, $\#(\gamma \alpha)$, for a final contribution of $N^{\#(\gamma \alpha)}$. We conclude that $\mathcal D_\alpha = N^{\#(\gamma \alpha)}M^{\# \alpha}$, proving \eqref{eq:exact-formula-moment-bipartite}.

\

\noindent{\it Second step. }We now move to the second step of the proof, which is computing the limit $N \to \infty$ of the moment formula \eqref{eq:exact-formula-moment-bipartite}.
\begin{lemma}
\label{lemma:NonCrossingPart}
The asymptotic moments of the (normalized) random matrices $W_N$ are given by a sum over non-crossing partitions
\begin{equation}\label{eq:MP-moments-free-cumulants}
\lim_{N \to \infty}\mathbb E N^{-1} \operatorname{Tr}\left[(N{^{-1}}W_N)^p\right] =\sum_{\tilde \alpha \in \mathrm{NC}(p)}c^{\#\tilde \alpha}.
\end{equation}
\end{lemma}
\noindent{\it Proof of Lemma~\ref{lemma:NonCrossingPart}. }Since we need to find the dominating terms in the sum, we have to maximize the function  $\mathcal S_p \ni \alpha \mapsto \#\alpha + \#(\gamma \alpha)$. The following lemma contains the key combinatorial insight which allows us to perform this task. The result below is contained in \cite{biane1997some} (see also \cite[Lecture 23]{nica2006lectures} for a textbook presentation). 

\begin{lemma}\label{lem:geodesic-permutations}
For a permutation $\alpha \in \mathcal S_p$, let $|\alpha|$ denote the minimum number of transpositions that multiply to $\alpha$; $|\alpha|$ is called the \emph{length} of the permutation $\alpha$ and satisfies the relations
$$ |\alpha| + \#\alpha = p \qquad\qquad |\alpha| = |\alpha^{-1}| \qquad\qquad |\alpha\beta| = |\beta\alpha|$$
for all permutations $\alpha,\beta\in \mathcal S_p$. The mapping
$d(\alpha, \beta) := |\alpha^{-1} \beta|$ defines a distance on $\mathcal S_p$. For two fixed permutations $\alpha, \beta \in \mathcal S_p$, the permutations $\chi \in \mathcal S_p$ saturating the triangle inequality
$$d(\alpha, \chi) + d(\chi, \beta) \ge d(\alpha, \beta)$$
are called \emph{geodesic}; we write $\alpha - \chi - \beta$. The set of geodesic permutations between the identity permutation $\mathrm{id}$ and the full cycle permutation $\gamma \in \mathcal S_p$ is in bijection with the set of non-crossing partitions $\mathrm{NC}(p)$: a partition $\tilde \alpha \in \mathrm{NC}(p)$ encodes the cycle structure of $\chi$, and the elements inside a given cycle of $\chi$ have the same cyclic ordering as in $\gamma$.
\end{lemma}

Using the lemma above and the asymptotic relation $M_N \sim cN$, the exponent of $N$ in the general term of \eqref{eq:exact-formula-moment-bipartite} can be bounded as follows
\begin{align}     
\label{eq:BoundCycles}
\nonumber\#\alpha + \#(\gamma \alpha) &= p-|\alpha| + p-|\gamma\alpha|\\
&= 2p-(|\alpha| + |\alpha^{-1}\gamma^{-1}|)\\
\nonumber&\leq 2p -  |\gamma^{-1}| = p+1,
\end{align}
where we have used the triangle inequality. The permutations $\alpha$ saturating this inequality are precisely the geodesic ones, \textit{i.e.}~the ones satisfying $\mathrm{id} - \alpha - \gamma$. These are in bijection with non-crossing partitions $\tilde \alpha \in \mathrm{NC}(p)$, and one has $\# \alpha = \# \tilde \alpha$, where the $\#$ notation denotes at the same time the number of cycles of a permutation $\alpha$ and the number of blocks of the corresponding non-crossing partition $\tilde \alpha$. We have shown
$$\mathbb E \operatorname{Tr}(W^p) \sim N^{p+1}\sum_{\tilde \alpha \in \mathrm{NC}(p)}c^{\#\tilde \alpha}.$$

\noindent{\it Third step. }The general term in the sum in the right hand side of \eqref{eq:MP-moments-free-cumulants} is a multiplicative function over the blocks of non-crossing partition $\tilde \alpha$: the contribution of each cycle is $c$, independently of the length of the block. We have thus identified the \emph{free cumulants} of the limiting distribution of the random matrices $W_N$: $\kappa_n = c$, for all $n \geq 1$ (we refer the reader to \cite[Lecture 11]{nica2006lectures} for the definition and the basic properties of free cumulants). In order to obtain the density of the probability distribution having the moments above, we use Voiculescu's $\mathcal R$-transform machinery. We have
$$\mathcal R(z) = \sum_{n=0}^\infty \kappa_{n+1} z^n = \sum_{n=0}^\infty c z^n = \frac{c}{1-z}.$$
The Cauchy transform and the $\mathcal R$-transform are related by the implicit equation $G(\mathcal R(z) + 1/z) = z$, which we solve for $G$, obtaining
$$G(w) = \frac{1-c+w-\sqrt{(w-a)(w-b)}}{2w}.$$
Above, we have chosen the solution of the second degree equation in $G$ such that $G(w) \sim w^{-1}$ as $w \to \infty$, and we have set $a = (1-\sqrt c)^2$ and $b = (1+\sqrt c )^2$.
Note that the function $G$ can have poles only at $w=0$, with residue $\max(1-c,0)$, explaining the atom at 0, when $0<c<1$. The expression for the density is obtained using the Stieltjes inversion formula
$$\frac{\mathrm{d} \mathrm{MP}_c}{\mathrm{d}x} = - \frac 1 \pi  \lim_{\varepsilon \to 0} \Im G(x+i \varepsilon). $$

\end{proof}

For random density matrices, one has to simply take into account the trace normalization: if $\rho_N$ is a sequence of random density matrices from the induced ensemble with parameters $(N, M_N)$, we can write $\rho_N = W_N / \operatorname{Tr} W_N$, for $W_N$ a sequence of random Wishart matrices of parameters $(N, M_N)$. We have then the following corollary.   
\begin{corollary}\label{cor:MP-convergence-random-density-matrices}
Let $\rho_N \in \mathcal M_N^{1,+}$ be a sequence of random density matrices from the \emph{induced ensemble} of parameters $(N,M_N)$, where $M_N$ is an integer sequence with the property that $M_N \sim cN$ as $N \to \infty$, with $c \in (0,\infty)$ a constant. The sequence $cN \rho_N$ converges, in moments, towards the Mar{\v{c}}enko-Pastur distribution $\mathrm{MP}_c$.
\end{corollary}
\begin{proof}
Write
$$cN \rho_N = cN \frac{W_N}{\operatorname{Tr}W_N} = \frac{W_N}{N} \cdot \frac{cN^2}{\operatorname{Tr}W_N}.$$
From equation \eqref{eq:convergence-Tr-W}, it follows that $cN^2/(\operatorname{Tr}W_N)$ converges, almost surely, to 1. Together with Proposition~\ref{prop:MP}, this proves the claim.
\end{proof}
Note that one can prove much stronger statements of convergence than the ones we cited; importantly, one can show that the largest eigenvalue of (properly normalized) Wishart and random density matrices converges, almost surely, towards the right edge of the support of the limiting Mar{\v{c}}enko-Pastur distribution \cite{bai1988note}.

\begin{remark}\label{rem:partial-trace-Wishart}
Note that if $W \in \mathcal M_{N_1N_2}(\mathbb C)$ is a Wishart \emph{tensor} of parameters $(N_1N_2,M)$, then $W' = \operatorname{Tr}_2(W)$ is a Wishart matrix of parameters $(N_1, N_2M)$. Indeed, if $W = XX^*$, with $X \in \mathcal M_{N_1N_2 \times M}(\mathbb C)$ a Gaussian matrix, then $W' = YY^*$, where $Y \in \mathcal M_{N_1 \times N_2M}(\mathbb C)$ is the matrix obtained by ``reshaping'' $X$ into a matrix of appropriate dimensions. This equivalence comes from the fact that both the partial trace and the matrix multiplication correspond to tensor contractions. In the random density matrix picture, we have, for a random vector $\Psi \in \mathbb C^{N_1} \otimes \mathbb C^{N_2} \otimes \mathbb C^{M}$, 
$$\rho_1 = [\operatorname{id}_{N_1} \otimes \operatorname{Tr}_{N_2}](\rho_{23}) = [\operatorname{id}_{N_1} \otimes \operatorname{Tr}_{N_2}\otimes \operatorname{Tr}_{M}](\Psi\Psi^*).$$
\end{remark}

\subsection{A combinatorial map version of the proof}

In this subsection we re-prove the formulas \eqref{eq:MP-moments-free-cumulants} and \eqref{eq:BoundCycles} using combinatorial map methods instead of the results on distances between permutations (Lemma~\ref{lem:geodesic-permutations}). Let us first provide a few definitions.

\begin{definition}\label{def:comb-maps} A \emph{connected labeled bicolored combinatorial map}, or simply a \emph{bicolored map}, is a triplet $\cM=(E, \sigma_\circ, \sigma_\bullet)$ where  
\begin{itemize}
\item $E$ is a set of edges labeled from 1 to $p$
\item $\sigma_\circ$ and $\sigma_\bullet$ are  permutations on $E$,
\item the group $\langle\sigma_\circ,\sigma_\bullet
\rangle$ generated by $ \sigma_\circ$ and $\sigma_\bullet$ 
acts transitively on $E$.
\end{itemize}
\end{definition}

We call white $vertices$ (resp.~black $vertices$, resp.~$faces$) the elements of the unique decomposition of $\sigma_\circ$ (resp.~$\sigma_\bullet$, resp.~$\sigma_\bullet\sigma_\circ$) in disjoint cycles.   A  bicolored map $\cM=(E, \sigma_\circ, \sigma_\bullet)$ satisfies the following classical result:
\be
\label{eq:Euler}
\#\sigma_\circ + \#\sigma_\bullet + \#(\sigma_\bullet\sigma_\circ)  - p  = 2 - 2g(\cM),
\ee
where $g(\cM)$ is a non-negative integer. 
A map can also be seen as a graph drawn on a two-dimensional surface, up to continuous deformations of the edges: each disjoint cycle of $\sigma_\circ$ (resp.~$\sigma_\bullet$) is a white (resp.~black) point (vertex) of that surface, each element of $E$ is an arc between a black and a white vertex, and the ordering of the cycles of $ \sigma_\circ$, and $\sigma_\bullet$ define an ordering of the half-edges incident to the vertices. The graph is said to be cellularly embedded (or simply embedded, in the context of this paper) on that surface if the connected components of the complement of the graph on that surface are homeomorphic to discs, in which case these components correspond to the faces, the disjoint cycles of $\sigma_\bullet\sigma_\circ$. The genus (number of holes) of a surface on which the underlying graph of a map can be embedded is precisely the integer $g(\cM)$ in \eqref{eq:Euler}. A genus 0 map can therefore be drawn on the plane without crossing. Finally, the transitivity condition forces the combinatorial map to be connected.

\

\begin{figure}[!ht]
\raisebox{+0.8cm}{\includegraphics[scale=1]{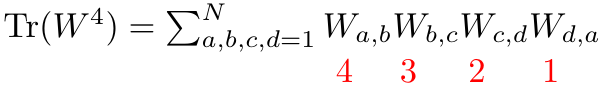}}\qquad \raisebox{+1.2cm}{$\leftrightarrow\qquad$}\includegraphics[scale=1]{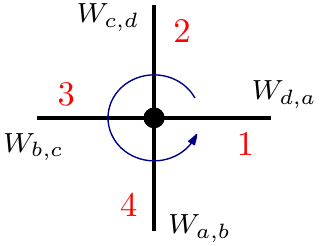}
\caption{\label{fig:BlackVertexBip} A trace of $p$ matrices can be represented as a vertex of valency $p$ with a cyclic ordering of incident half-edges. }
\end{figure}
With the notations of the proof of Prop.~\ref{prop:MP}, we define a bicolored map on  $E=\{1,\ldots,p\}$ by setting  $\sigma_\bullet := \gamma$, and $\sigma_\circ := \alpha$. As illustrated in Fig.~\ref{fig:BlackVertexBip}, in this representation,    a single  black vertex of valency $p$ is associated to $\Tr (W^p)$, each matrix $W$ corresponding to an edge incident to it. The trace induces a cyclic counter-clockwise ordering of these half-edges around the vertex\footnote{As will appear in the following, the moments are expressed in terms of labeled maps. This is because the trace corresponding to the only black vertex does not come with a factor $1/p!$.} (the matrices are labeled growingly from right to left, and the corresponding half-edges are labeled growingly when going counter-clockwise around the black vertex). The free index corresponding to the matrix $X$ (resp $X^\ast$) is associated to the left (resp. right) side of the half-edge $ XX^*$. Each Wick pairing induces  a permutation $\sigma_\circ$ which defines the white vertices. The maps corresponding to the examples of Figure~\ref{fig:expectation-moment-2-Wishart} are shown in Figure~\ref{fig:MapExTrW2}.

\begin{figure}[!ht]
\includegraphics[scale=1.2]{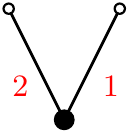}\hspace{2cm}\includegraphics[scale=1.2]{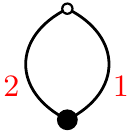}
\caption{\label{fig:MapExTrW2} Combinatorial maps for the example of Figure~\ref{fig:expectation-moment-2-Wishart}. }
\end{figure}
\begin{notation}
\label{not:SetsOfMaps}
 Throughout the text, we will denote $\bM_{p}$ the set of labeled connected bicolored maps with a single black vertex and $p$ edges, and $\bM^g_{p}$ its subset of maps having genus $g$. 
\end{notation}
The two elements of $\bM_{2} = \bM^0_{2}$ are shown in Fig.~\ref{fig:MapExTrW2}, and an example of map in $\bM^0_{7}$ is shown on the left of Fig.~\ref{fig:Bij1}. Denoting $F(\cM)$ the number of faces of a map $\cM$, and $V(\cM)$ its total number of vertices,
we can therefore restate the key relation \eqref{eq:exact-formula-moment-bipartite} as follows: 
\be
\label{eq:MomMapsFacesVert}
\mathbb E \Tr\bigl(W^{p}\bigr) = \sum_{\cM\in \bM_{p}} N^{F(\cM)} M^ {V(\cM) - 1}.
\ee

\

The genus \eqref{eq:Euler} of a connected map in $\bM_{p}$ writes
\be
\label{eq:EulerCharacSec1}
2-2g(\cM)
= V(\cM) + F(\cM) -p , 
\ee 
 so that, if we assume $M\sim cN$,
\be
\label{eq:MomentMapsc}
\mathbb E \Tr\bigl(W^{p}\bigr) = (1+o(1))\sum_{\cM\in \bM_{p}} N^{p + 1 -2g(\cM)} c^ {V(\cM) - 1}.
\ee
Since $g\ge 0$,  the bound \eqref{eq:BoundCycles} on the number of cycles $\#\alpha + \#(\gamma \alpha)$ is now obtained using the positivity of the genus instead of the distances between permutations, and the leading terms in $N$ correspond to planar maps. Therefore, with the Notation \ref{not:SetsOfMaps},
\be\label{eq:MP-limit-with-comb-maps}
\lim _{N\rightarrow + \infty}\frac 1 {N^{p+1}} \mathbb E \Tr\bigl(W^{p}\bigr) = \sum_{\cM\in\bM_{p}^0}c^ {V(\cM) - 1}.
\ee
It is easily seen by studying the permutations $\sigma_\bullet = \gamma$ and $\sigma_\circ = \alpha$, that elements of $\bM_{p}^0$ with $V(\cM)$ vertices correspond to non-crossing partitions with $V(\cM)-1$ disjoint cycles. Another simple way of seeing this is to notice that maps in  $\bM_{p}^0$ with $V(\cM)$ vertices and $F(\cM)$ faces are in bijection with bicolored labeled plane trees with $p$ edges, $V(\cM)-1$ white vertices, and $F(\cM)$ black vertices (plane trees are all bicolored), as explained below. Such trees are themselves in bijection with non-crossing partitions with $V(\cM)-1$ disjoint cycles throughout a dual mapping.

 Starting from a map in $\bM_{p}^0$, a vertex is added in each face and an edge is added between each corner of each white vertex and the newly added vertices (the ordering of edges around vertices is given by the clockwise ordering of appearance of corners around faces)\footnote{Note that by linking the newly added vertices to the corners around the black vertex instead of the white vertices, we have a duality between elements of $\bM_{p}^0$, see Def.~\ref{def:TutteDual}.}. This procedure is illustrated in Fig.~\ref{fig:Bij1}. It is a particular case of Tutte's bijection for bicolored maps \cite{TutteTrinity} and is known to be bijective.
\begin{figure}[!ht]
\raisebox{0.7cm}{\includegraphics[scale=0.7]{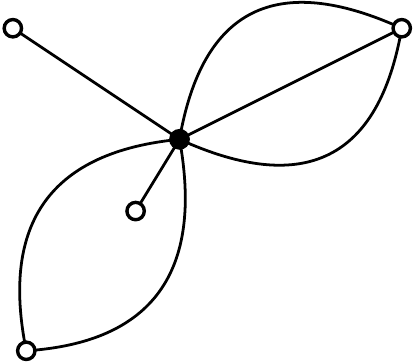}}\hspace{1.2cm}\includegraphics[scale=0.7]{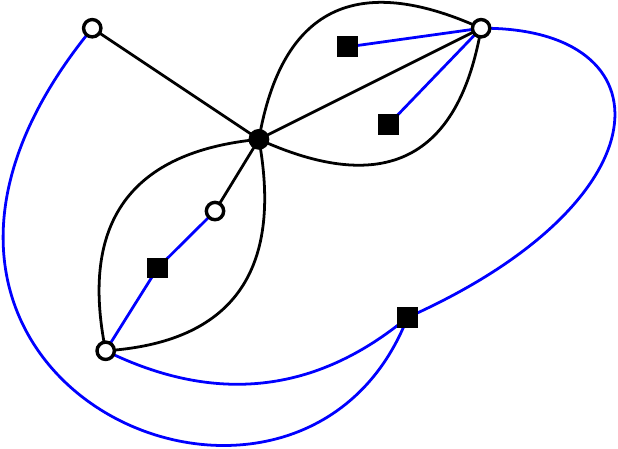}\hspace{1cm}\includegraphics[scale=0.7]{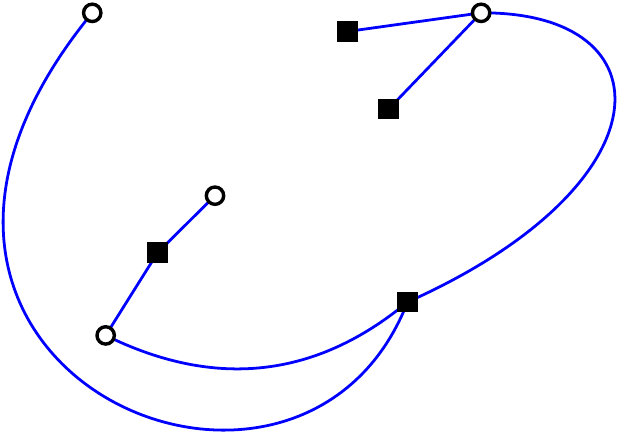}
\caption{\label{fig:Bij1} Bijection between elements of $\bM_p^0$ and labeled plane trees. }
\end{figure}
The map obtained from $M\in \bM_{p}^0$ is connected, has $V(M) - 1 +F(M) = p+1$ vertices, and its number of edges is the number of corners of $M$ around white vertices, which is $p$, and is therefore a tree.

In the case where $c=1$, $\lim _{N\rightarrow + \infty}\frac 1 {N^{p+1}} \mathbb E \Tr\bigl(W^{p}\bigr)$ is just the cardinal of $\bM_{p}^0$, namely the number of plane trees with $p$ edges, which is well-known \cite{harary1964number} to be the Catalan number $C_{p}= \frac1{p+1}\binom {2p} p $, 
\be
\lim _{N\rightarrow + \infty}\frac 1 {N^{p+1}} \mathbb E \Tr\bigl(W_{c=1}^{p}\bigr)= C_{p}.
\ee

At finite $N$ and for $c=1$, the contribution of order $N^{p+1-2g}$ is given by the number of maps in $\bM_{p}$ whose genus is $g$. The very same transformation which mapped maps in $\bM_{p}^0$ to plane trees can be applied. It maps elements of $\bM_{p}^g$ to bicolored maps of genus $g$ with a single face. Maps with a single face are called unicellular maps \cite{Chapuy2010}. Conversely, the inverse transformation maps bicolored unicellular maps of genus $g$ to genus $g$ bicolored maps with a single black vertex. Unicellular maps of genus $g$ are the object of many studies, and they can be counted exactly using various recursive formulas (see for instance the Harer-Zagier formula \cite{Zagier1986}, Lehman-Walsh formula \cite{walsh1972counting}, Goupil-Schaeffer formula \cite{Goupil98},
and Chapuy's formula \cite{Chapuy11, Chapuy13, Carrell15}). Using such results, \eqref{eq:MomentMapsc} can be expressed exactly  for any $p$.

\subsection{Free cumulants and maps}

We elaborate here on the relation between moments and free cumulants in the case of a Wishart distribution. As was pointed out in the proof of Proposition~\ref{prop:MP} (third step), the free cumulants of the limiting distribution are  $\kappa_n=c$ for all $n\ge 1$. We consider again the expectation value of the traces of powers of the matrix $W$
\begin{align}
\label{eq:ClassMomCumulForm}
\bE \Tr(W^p)&=\sum_{i_1,\ldots,i_p=1}^N\bE(W_{i_1i_p}\ldots W_{i_3i_2}W_{i_2i_1}),\\
&=\sum_{i_1,\ldots,i_p=1}^N\sum_{\cJ\vdash \{1,\ldots,p\}}\prod_{J_r \in \cJ}\cK(W_{J_r}), \label{eq:DecCumClass}
\end{align}
where the second sum is taken over all partitions $\cJ$ of $\{1,\ldots,p\}$, and we define 
\begin{equation}
\cK(W_{J_r})=\cK(\{W_{i_{k+1}i_k}\}_{\substack{k\in J_r \\ k\in \mathbb{Z}_p}})
\end{equation}
as the classical cumulants of the family of random variables $\{W_{i_{k+1}i_k}\}_{\substack{k\in J_r \\ k\in \mathbb{Z}_p}}$.

 The claim \eqref{eq:DecCumClass} can be obtained from the definition of the generating function of classical cumulants. Depending on conventions, this cumulant generating function can be defined either as the logarithm of the characteristic function or of the moment generating function of the corresponding probability density. For reference on classical cumulants, see for instance \cite[Lecture 11, Appendix]{nica2006lectures}, \cite{novak2014three} or \cite{Lukacs_book}.
 As moments, these classical cumulants can be represented graphically. In the mathematical-physics literature, the graphical objects representing them are often described as connected combinatorial maps (or ribbon graphs) with half-edges or connected opened combinatorial maps, see Remark~\ref{rem:on_cumulants-mathphys} for more details. In our case, as is shown in the next paragraphs, they will be represented as usual combinatorial maps with one white vertex\footnote{The fact that there is only one white vertex is the direct translation of the ``connected'' conditions introduced in the mathematical-physics literature.} and one black vertex. The fact that combinatorial maps with \emph{one} white vertex and one black vertex represent classical cumulants is the translation of the decomposition of classical cumulants in terms of classical moments for some random variables $X_1,\ldots, X_s$, which is the cumulant-moment formula in classical probability \cite[Definition 11.28]{nica2006lectures}, making use of the M\"obius function on the lattice of partitions:
\begin{equation}
\label{eq:CumulInMoments}
\cK(X_{\lvert J_r \rvert},\ldots,X_1 )= \sum_{\cI\vdash \{1,\dots, {\lvert J_r \rvert}\}}(-1)^{|\cI|-1} (|\cI|-1)! \prod_{I_m\in \cI}\bE\left(\prod_{k\in I_m}X_k\right),
\end{equation}
for the random variables $X_k = W_{i_{k+1}i_k}$. The moments $\bE\left(\prod_{k\in I_m}X_k\right)$ are then computed using Wick's theorem (Prop.~\ref{prop:Wick}). As detailed previously, these Wick pairings can be represented graphically, either using $X-$ and $X^*-$boxes, or using the combinatorial map representation. In this last representation, identifying $W_{i_{k+1}i_k}$ with the integer $k\in I_m$, the Wick pairings induce a permutation $\alpha_m$ on the set $I_m$ (see the proof of Prop.~\ref{prop:MP}). Each cycle of  $\alpha_m$ contributes with $\lfloor N c \rfloor$. 
As before, we can define a bicolored map in $\bM_{\lvert I_m \rvert}$ by setting $\sigma_\circ = \alpha_m$, and $\sigma_\bullet = \gamma_m$, where $\gamma_m$ is the cycle $(x_1, \ldots, x_{\lvert I_m \rvert})$, in which $x_1<x_2< \ldots <x_{\lvert I_m \rvert}$ are all the elements of $I_m$.

The difference with what has been done before, where we were considering the expectation of a trace, is that here, the indices $\{i_k, i_{k+1}\}$ are not summed. This means that the faces of the map $(I_m,  \alpha_m,  \gamma_m)$ (the disjoint cycles of $\gamma_m\alpha_m$) do not factorize in the expression of $\bE\left(\prod_{k\in I_m}X_k\right)$, and instead we have a Kronecker delta for each corner of each white vertex:
\begin{equation}
\bE\left(\prod_{k\in I_m}W_{i_{k+1}i_k}\right)= \sum_{\alpha_m \in \cS_{\lvert I_m \rvert}}\lfloor Nc \rfloor^{(\#\alpha_m)} \prod_{k\in I_m} \delta_{i_k,i_{\alpha_m(k) + 1}}.
\end{equation}

Furthermore, because of the factors $(-1)^{|\cI|-1} (|\cI|-1)!$ in the expression of the cumulants \eqref{eq:CumulInMoments}, and because the products of Kronecker deltas factorize, the terms corresponding to maps which have more than one white vertex will be canceled by products of terms corresponding to maps with fewer white vertices. The only non-vanishing terms are those for which $\cI$ is reduced to $\{1,\ldots, \lvert J_r\rvert \}$, and the cumulants can therefore be expressed in terms of cycles of $J_r$ of length $\lvert J_r\rvert$, or equivalently using  a graphical expansion over bicolored maps with two vertices,
\begin{equation}
\label{eq:CumulInMaps}
\cK(W_{J_r})= \lfloor Nc \rfloor  \sum_{ \substack{{\cM\in\bM_{\lvert J_r \rvert}}\\{V(\cM) = 2}} }\     \prod_{k\in J_r} \delta_{i_k,i_{\alpha_\cM(k) + 1}},
\end{equation}
where we have denoted $\alpha_\cM$ the cycle $\sigma_\circ$ for the map $\cM$.

We do not prove these statements here, however we give what we hope is an illuminating example in the present case. Consider the cumulant of three matrix elements ${W_{i_6i_5}, W_{i_4i_3}, W_{i_2i_1}}$, 
\begin{multline}
\label{eq:cumulant_example} \cK(W_{i_6i_5}, W_{i_4i_3}, W_{i_2i_1})=\bE(W_{i_6i_5}W_{i_4i_3}W_{i_2i_1}) 
-\bE(W_{i_6i_5}W_{i_4i_3})\bE(W_{i_2i_1}) \\
-\bE(W_{i_6i_5})\bE(W_{i_4i_3}W_{i_2i_1}) - \bE(W_{i_6i_5}W_{i_2i_1})\bE(W_{i_4i_3}) 
+2\bE(W_{i_6i_5})\bE(W_{i_4i_3})\bE(W_{i_2i_1}).
\end{multline}

\begin{figure}
 \includegraphics[scale=0.65]{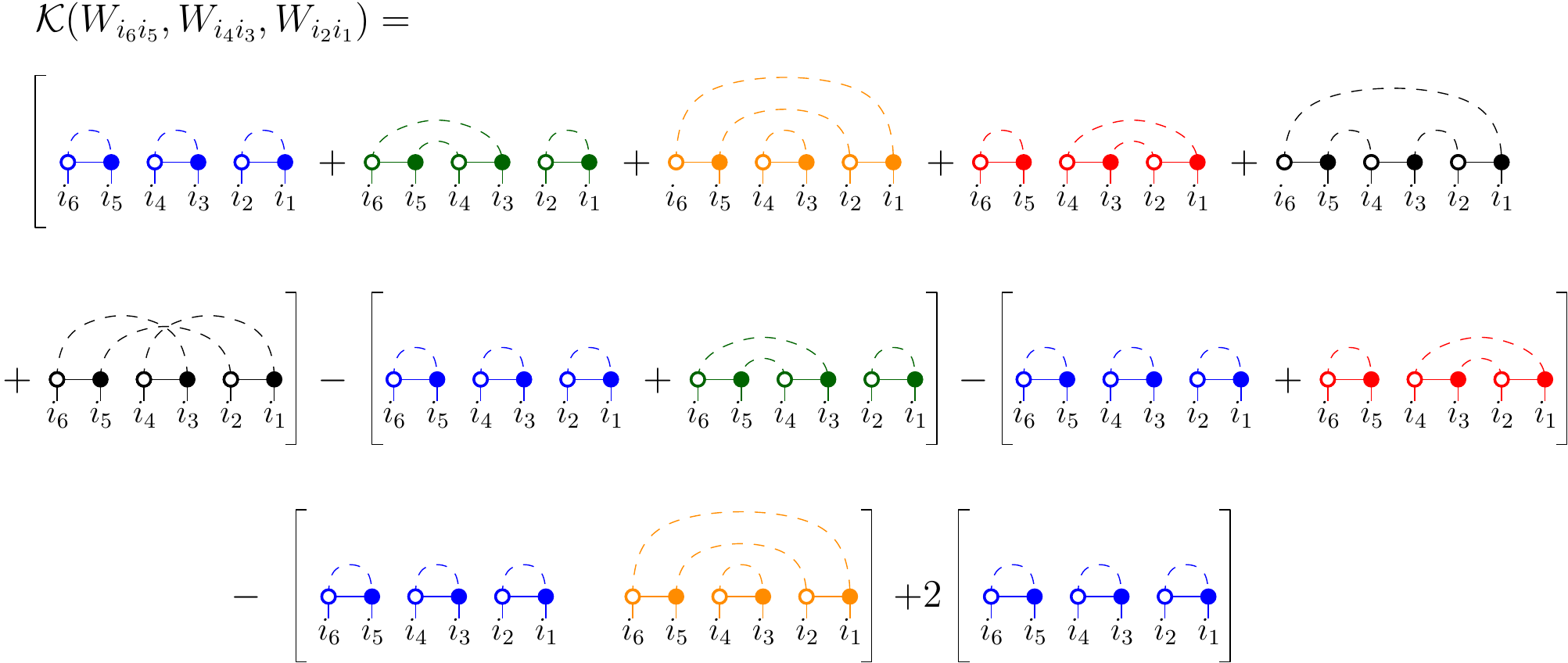}
 \caption{Graphical expansion of the cumulant $\cK(W_{i_6i_5}, W_{i_4i_3}, W_{i_2i_1})$ from its moments decomposition and the Wick theorem. Identical terms are shown in the same color.}\label{fig:cumulant_example_graphical}
\end{figure}
The graphical representation of the Wick pairings of equation \eqref{eq:cumulant_example} is shown in Fig.~\ref{fig:cumulant_example_graphical}. 
The dashed edges indicate which $X,X^*$ are paired together. In order to simplify (and compactify) the drawings, the vertices of empty types represent the $X$-boxes, while the filled vertices represent the $X^*$-boxes. In this representation, it is the disconnected parts of the graphical expansion of $\bE(W_{i_6i_5}W_{i_4i_3}W_{i_2i_1})$ which are suppressed by the remaining terms of the right hand side of \eqref{eq:cumulant_example}, that are all disconnected in nature. The weights associated to the only remaining terms are $\lfloor Nc \rfloor  \delta_{i_1i_6}\delta_{i_2i_3}\delta_{i_4i_5}$ and $\lfloor Nc \rfloor  \delta_{i_1i_4}\delta_{i_2i_5}\delta_{i_3i_6}$. 
In Fig.~\ref{fig:cumulant_example_maps}, we represent the same expansion using the representation in terms of combinatorial maps.

\begin{figure}
 \includegraphics[scale=0.65]{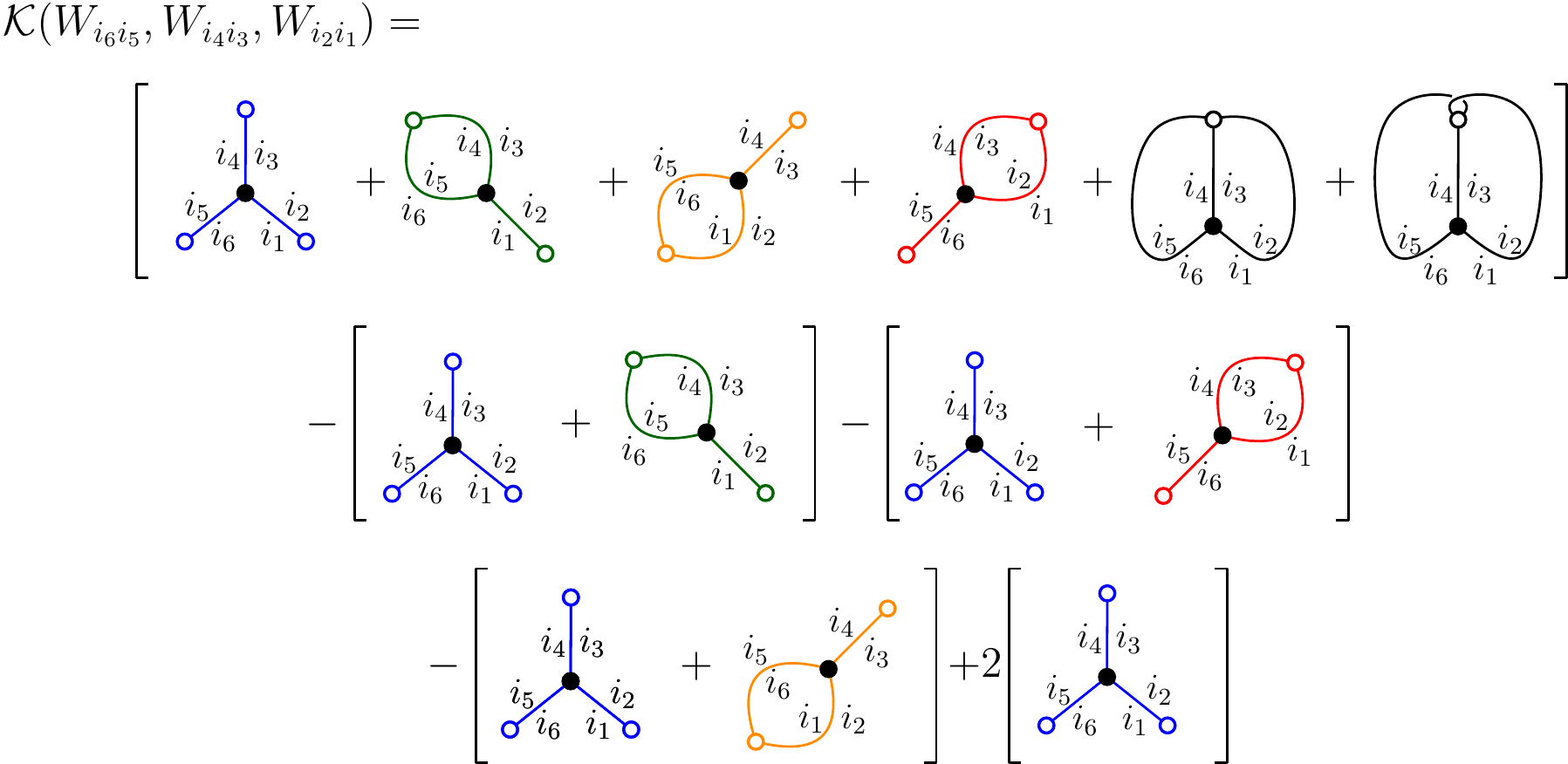}
 \caption{This is the translation of the previous Figure~\ref{fig:cumulant_example_graphical} in the language of combinatorial maps. Every corner on a white vertex carries a $\delta_{i_k, i_{k'}}$, where $i_k$ and $i_{k'}$ are the incident indices on the drawing. The only graphs which do not cancel are the ones corresponding to  $\lfloor Nc \rfloor  \delta_{i_1i_6}\delta_{i_2i_3}\delta_{i_4i_5}$ and $\lfloor Nc \rfloor  \delta_{i_1i_4}\delta_{i_2i_5}\delta_{i_3i_6}$ (shown in black). \label{fig:cumulant_example_maps}}
\end{figure}

Using the expression \eqref{eq:CumulInMaps} in the classical moment-cumulant formula \eqref{eq:ClassMomCumulForm} above, we have
\begin{align}
\bE \Tr(W^p)&=\sum_{i_1,\ldots,i_p=1}^N\sum_{\cJ\vdash \{1,\ldots,p\}}\prod_{J_r\in \cJ}    \lfloor Nc \rfloor  \sum_{ \substack{{\cM\in\bM_{\lvert J_r \rvert}}\\{V(\cM) = 2}} }\     \prod_{k\in J_r} \delta_{i_k,i_{\alpha_\cM(k) + 1}}\\
&= \sum_{\cJ\vdash \{1,\ldots,p\}}\lfloor N c\rfloor^{|\cJ|}\sum_{i_1,\ldots,i_p=1}^N \prod_{J_r\in \cJ}  \sum_{ \substack{{\cM\in\bM_{\lvert J_r \rvert}}\\{V(\cM) = 2}} }\     \prod_{k\in J_r} \delta_{i_k,i_{\alpha_\cM(k) + 1}}.
\end{align}

We can rewrite the terms on the second line in terms of bicolored maps in $\bM_p$. The partition $\cJ$ divides the edges incident to its black vertex in groups $J_r$ which are all linked to the same white vertex, one for each group. For each one of these groups, a sum is taken over all possible maps with two vertices, of the product of Kronecker deltas over the corners around the white vertex. Any choice of $|\cJ|$ two-vertices maps provides a different map in $\bM_p$. Denoting $\bM_p[\cJ]$ the subset of $\bM_p$ of maps for which for every $J_r\in \cJ$, $(\alpha_\cM)_{|J_r}$ is a unique cycle of length $|J_r|$, we can therefore rewrite  this sum as
\be
 \bE \Tr(W^p) =  \sum_{\cJ\vdash \{1,\ldots,p\}}\lfloor N c\rfloor^{|\cJ|}\sum_{i_1,\ldots,i_p=1}^N \,  \sum_{ \cM\in\bM_{p}[\cJ]} \, \prod_{k\in \{ 1,\ldots, p\}} \delta_{i_k,i_{\alpha_\cM(k) + 1}}.
\ee

Now, we have precisely one index per corner around the black vertex, and we still have one Kronecker delta per corner around a white vertex, so the sum over  $i_1,\ldots, i_p$ yields a factor $N^{F(\cM)}$, so that 
\be
 \bE \Tr(W^p) =  \sum_{\cJ\vdash \{1,\ldots,p\}}\lfloor N c\rfloor^{|\cJ|}  \sum_{ \cM\in\bM_{p}[\cJ]} N^{F(\cM)} = \sum_{ \cM\in\bM_p} \lfloor N c\rfloor^{V(\cM)-1} N^{F(\cM)},
\ee
and we indeed recover \eqref{eq:MomMapsFacesVert}, and therefore the results of Lemma~\ref{lemma:NonCrossingPart} which state that the leading terms correspond to the planar maps.

However, we can now  pull back the planarity result to the level of the cumulants. From the calculation above, we see that the cumulants are the two-vertices maps which are then re-arranged (their black vertices are crushed into a single one, and the edges are naturally ordered around this single vertex as they are labeled from 1 to $p$).
 In order for the resulting map to contribute to the large $N$ limit, it has to be planar, thus the collection of two-vertices maps we start from all have to be planar. Indeed it is a classical result that the genus of a map is greater or equal to that of a submap (see e.g.~\cite[Prop.~4.1.5, Sec.~4.1, p.~102]{GraphsOnSurfaces}). We reprove it in our case in Prop.~\ref{prop:GenusSubmap}. If all the two-vertices maps are planar, in addition, the partition $\cJ$ has to be non-crossing  for the map to be planar.  
\begin{proposition}
\label{prop:GenusSubmap}
Consider a map $\cM\in\bM_p$, and the submap $\cM'$ obtained by keeping the black vertex, a single white vertex, and all the edges between them. If $\cM'$ is non-planar, then $\cM$ is also non-planar.
\end{proposition}
\begin{proof}
Suppose that $g(\cM')>0$, and consider the set of edges of the combinatorial map that do not belong to $\cM'$. Among these edges, we first remove all the ones adjacent to univalent white vertices, and we remove the isolated white vertices created in this process. The obtained map has the same genus as the initial map. For each of the remaining edge we proceed as follows. If the edge is adjacent to only one face and to a white vertex which is not univalent, then the map has a non zero genus and the proposition is true. If the edge is adjacent to two different faces, withdrawing it does not change the genus. We do so, and then repeat this process. At each step, we may create univalent white vertices, which we remove, together with the isolated white vertices thus created. Either we find  during the process an edge which is not a bridge and is adjacent to only one face  (and the proposition is proved), or we have erased all the edges that did not belong to $\cM'$, without changing the genus of the map. We are thus left with a map with non vanishing genus $g(\cM)=g(\cM')>0$. 
\end{proof} 
We can now claim that \emph{the free cumulant of order $p$ is represented by the unique planar map with one white vertex, one black vertex and $p$ edges}. The interest of this remark lies in the fact that it gives a heuristic for reading freeness directly out of the combinatorial maps expression of moments (intuition which will be used in Sections \ref{subsec:Large-HBC} and \ref{sec:BalancedGeneral}). Indeed looking at moments of alternating products of different random matrices, one expects to express these moments using colored combinatorial maps\footnote{See later sections.} (where the colors label different types of edges indexing Wick pairing between different matrices). If one believes\footnote{This can actually be shown, however this lies out of the main scope of this paper.} that white vertices with adjacent edges of different colors correspond to mixed cumulants, it is easy to guess if in the large $N$ limit the different matrices converge to different free elements, just by confronting the weights of the white vertices with only one adjacent color to those of white vertices with strictly more than one adjacent color.

\begin{remark}
We notice here that the following expectation values of products of matrix elements such that $i_k\neq i_q, \ \forall k\neq q$ (indices are fixed and not summed)
\begin{equation}
\bE(W_{i_1i_p}\ldots W_{i_3i_2}W_{i_2i_1})=M=\lfloor Nc \rfloor,
\end{equation}
thus are equal up to a rescaling in $N$ to the free cumulants $\kappa_n=c$. A similar fact has been used in the study of Hermitian (formal) matrix integrals and simple combinatorial maps in \cite{Gaetan-Elba}. In this reference a combinatorial maps interpretation of the relation between moments and classical cumulants is given by bijective means valid at all orders in the large $N$ asymptotic expansion. In particular the bijective method is valid at the first order of the expansion and thus can be interpreted in the context of free probability and the corresponding free cumulants. 
\end{remark}

\begin{remark}\label{rem:on_cumulants-mathphys}
The cumulants can also be expressed in terms of the white vertices with incident-pending half-edges, whose both sides are labeled by pairs $(i_k,i_{k+1})$, while the half-edges themselves are labeled by the first index of the pair $k\in I_r$.  In this representation, the moments are obtained by gluing the cumulants around the black vertex while respecting the natural order of natural integers, and the free cumulants are the connected terms, for which the indices of the half-edges satisfy a non-crossing condition. 
\end{remark}

\bigskip

\section{The four-partite case: joint distribution of two marginals}
\label{sec:ABCD}

In this section, we describe the joint distribution of two marginals, $\rho_{AB}$ and $\rho_{AC}$, of a random pure quantum state $\rho_{ABCD} = |\psi\rangle\langle\psi|_{ABCD}$. We prove all results both using metric properties of the permutation groups and the combinatorial maps approach with the underlying discrete geometry, in parallel.  As before, we state our results for Wishart tensors; the corresponding results for random density matrices can be obtained by renormalizing the Wishart matrices (see Remark~\ref{rem:4-partite-balanced-QIT}, as well as the very last paragraph of this section). 

Let $X = X_{ABCD} \in \mathcal H_A \otimes \mathcal H_B \otimes \mathcal H_C \otimes \mathcal H_D$ be a random Gaussian tensor, i.e.~a random tensor with i.i.d.~standard complex Gaussian entries. We are going to study here the \emph{joint distribution} of the two marginals
\begin{equation}\label{eq:def-2-marginals}
W_{AB} := [\operatorname{id} \otimes \operatorname{id} \otimes \operatorname{Tr} \otimes \operatorname{Tr}](XX^*) \qquad \text{ and } \qquad  W_{AC}:=[\operatorname{id} \otimes \operatorname{Tr} \otimes \operatorname{id} \otimes \operatorname{Tr}](XX^*),
\end{equation}
in some asymptotic regime, where the dimension of some of the Hilbert spaces grow to infinity. In order to study mixed moments of the two random matrices $W_{AB}$ and $W_{AC}$, we are going to assume from now on that 
$$\dim \mathcal H_B = \dim \mathcal H_C.$$
We consider the following two asymptotic regimes:
\begin{itemize}

	\item the \emph{balanced regime}: $\dim \mathcal H_B = \dim \mathcal H_C = N \to \infty$, $\dim \mathcal H_D /  \dim \mathcal H_A \to c$, for some constant $c \in (0,\infty)$;

	\item the \emph{unbalanced regime}: $\dim \mathcal H_B = \dim \mathcal H_C = m$, $\dim H_A \sim N$, $\dim H_D \sim cN$, $N \to \infty$, for some constants $m \geq 1$ and $c \in (0, \infty)$.
\end{itemize}
As special cases of the balanced asymptotic regime, we shall emphasize the following two sub-cases (see Remark~\ref{rem:balanced-4-partite-subcases}):
\begin{itemize}
\item $\dim \mathcal H_B = \dim \mathcal H_C = N$, $\dim H_A \sim c_1N$, $\dim H_D \sim c_4N$, $N \to \infty$, for some constants $c_{1,4} \in (0, \infty)$;
\item  
$\dim \mathcal H_B = \dim \mathcal H_C = N$, $\dim H_A = k$, $\dim H_D = l$, $N \to \infty$, for some integer constants $k,l \geq 1$.
\end{itemize}

The balanced regime will be studied in Section \ref{sec:2-marginals-large-BC}, while the unbalanced regime will be studied in Section \ref{sec:2-marginals-fixed-BC}. We first compute the exact combinatorial expression for the moments.

\smallskip

\subsection{Exact expression for the moments}
\label{sec:ExactMomentsABCD}

\

\medskip

In this section, we prove a non-asymptotical result, describing the mixed moments of the marginals $W_{AB}$, $W_{AC}$ as a combinatorial sum. This result is the starting point for the asymptotic computations in Sections \ref{sec:2-marginals-large-BC} and \ref{sec:2-marginals-fixed-BC}. In order to describe the mixed moments of the two marginals, we write
$$W_f :=  \prod_{1 \leq a \leq p}^{\longleftarrow} W_{f(a)},$$
for a word $f \in \{AB,AC\}^p$ of length $p$ in the letters $(AB), (AC)$. Here again, the matrices are labeled from right to left. We also denote $N_X:=\dim \mathcal H_X$, for $X=A,B,C,D$; note that we have $N_B=N_C =:N_{B,C}$. 
\begin{theorem}\label{thm:moments-fixed-N}
The mixed moments of the two marginals $W_{AB},W_{AC}$ defined in \eqref{eq:def-2-marginals} are given by the following exact formula:
\begin{equation}\label{eq:exact-moments-2-marginals}
\mathbb E \operatorname{Tr} W_f = \sum_{\alpha \in \mathcal S_p} N_A^{\#(\gamma\alpha)} N_D^{\#\alpha} N_{B,C}^{L(f,\alpha)},
\end{equation}
where $\gamma = (1 2 3 \ldots p )$ is the full cycle permutation and $L(f,\alpha)$ is a combinatorial function described in equations \eqref{eq:L-f-alpha}-\eqref{eq:hat-alpha}.
\end{theorem}
\begin{proof}
The proof is a standard application of the graphical Wick formalism from \cite{collins2011gaussianization}, similar to the proof of the moment formula in Proposition~\ref{prop:MP}. First, consider the diagram associated to the trace of the operator $W_f$ (see Figure~\ref{fig:2-marginals-moment} for a simple example): each of the $p$ $X$ (or $\bar X$) boxes has 4 decorations, corresponding to the 4 Hilbert spaces $\mathcal H_{A,B,C,D}$. The sum over permutations $\alpha \in \mathcal S_p$ comes from the application of the Wick formula, and the different factors in the general term count the different types of loops, as follows:
\begin{itemize}
\item There are $\#(\gamma\alpha)$ loops associated to the Hilbert space $\mathbb C^{N_A}$, since the initial wiring of the boxes is given by the cyclic permutation $\gamma$, corresponding to matrix multiplication. The total number of loops is the number of cycles of the permutation obtained as the product of the permutation describing the initial wiring (here, $\gamma$, connecting $X$ boxes to $\bar X$ boxes) and the permutation coming from the graphical Wick formula ($\alpha$, connecting $\bar X$ boxes to $X$ boxes).
\item There are $\#\alpha$ loops associated to the Hilbert space $\mathbb C^{N_D}$, since the initial wiring of the boxes is given by the identity, corresponding to the partial trace over the fourth tensor factor (which appears in all the $W_{f(i)}$). 
\item We denote by $L(f,\alpha)$ the number of loops corresponding to the second ($B$) and the third ($C$) tensor factors. This quantity is less obvious to evaluate, since it depends in a non-trivial way on both the specific word $f$ and on the permutation $\alpha$. These loops contribute each a factor of $N_{B,C}$.  
\end{itemize}
\begin{figure}[!ht]
	\centering
	\includegraphics[width=0.45\textwidth]{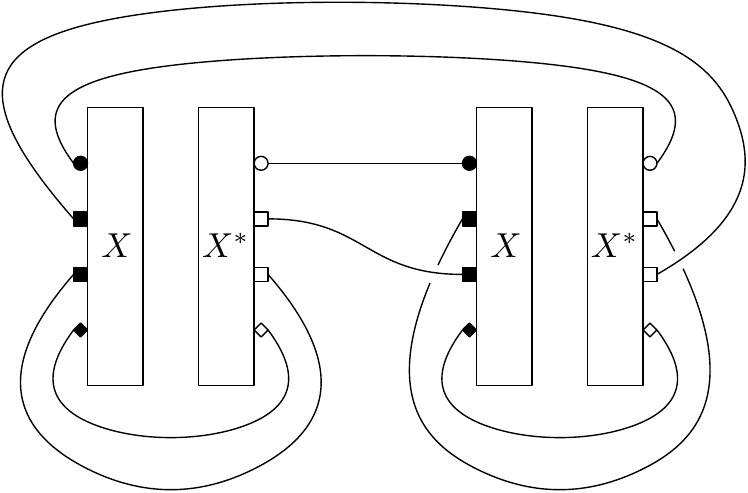}
	\caption{The diagram of the random variable $\operatorname{Tr}(W_{AB}W_{AC})$. Each tensor box has 4 decorations corresponding to $\mathcal H_A$ (round shape), $\mathcal H_B, \mathcal H_C$ (square shapes), and $\mathcal H_D$ (diamond shape), displayed in this order from top to bottom.}
	\label{fig:2-marginals-moment}
\end{figure}

We discuss next the exponent $L(f,\alpha)$, counting the number of loops associated to systems $B$ and $C$. As in the other cases, we have 
\begin{equation}\label{eq:L-f-alpha}
L(f,\alpha) = \#(\hat\gamma_f\, \hat \alpha),
\end{equation}
where $\hat\gamma_f, \hat \alpha \in \mathcal S_{2p}$ encode, respectively, the initial and the Wick wirings of the $B,C$ decorations. In order to define properly these permutations, let us relabel the $2p$ decorations of the $X$ boxes by $\bigl\{(1,B), (2,B), \ldots, (p, B), (1,C), (2,C), \ldots, (p, C)\bigr\}$ to keep track of the system they refer to (as usual, the numbering is increasing from right to left). Then, using cyclical notation (\textit{i.e.}~$p+1 \equiv 1$), we have 
\begin{align}
\label{eq:hat-f-B} \hat\gamma_f(a,B)&  = \begin{cases}
(a,B) & \qquad \text{ if } f(a) = AC \\
(a+1,B) & \qquad \text{ if } f(a) = AB \text{ and } f(a+1) = AB \\
(a+1, C) & \qquad \text{ if } f(a) = AB \text{ and } f(a+1) = AC
\end{cases} \\
\label{eq:hat-f-C} \hat\gamma_f(a,C)&  = \begin{cases}
(a,C) & \qquad \text{ if } f(a) = AB \\
(a+1, B) & \qquad \text{ if } f(a) = AC \text{ and } f(a+1) = AB \\
(a+1, C) & \qquad \text{ if } f(a) = AC \text{ and } f(a+1) = AC
\end{cases} \\
\label{eq:hat-alpha} \hat \alpha(i, Z)&  = (\alpha(i), Z), \quad \text{ for } Z=B,C.
\end{align}
\end{proof}

As a direct application of \eqref{eq:exact-moments-2-marginals}, we have, for the example in Fig.~\ref{fig:2-marginals-moment},
$$\mathbb E \operatorname{Tr}(W_{AB}W_{AC}) = N_A N_{B,C}^3 N_D^2 + N_A^2N_{B,C}N_D,$$
where the two terms in the right hand side correspond respectively to the identity permutation and to the transposition $(12)$.

\bigskip

We shall now present a different point of view on Theorem~\ref{thm:moments-fixed-N}, using combinatorial maps. First let us introduces some important (and somehow non-standard) notation for integer intervals: for integers $a,b$, we denote
\begin{equation}\label{eq:def-integer-interval}
\llbracket a, b\rrbracket := \{a, a+1, \ldots, b-1, b\}.
\end{equation}

\noindent{\bf Orbits.} 
Before going to the combinatorial maps version of this result, we provide a more formal definition of the objects we count, in terms of the permutation. If $I\subset \{A,B,C,D\}$, we denote 
\be 
\label{eq:LabellingEdgesABCD}
[p]^I := \bigsqcup_{i\in I}\bigl\{ (a,i)\bigr\}_{a\in\llbracket 1,p\rrbracket},
\ee 
so that for instance, $[p]^{B,C} = \bigl\{(a,B) \bigr\}_{a\in\llbracket 1,p\rrbracket} \sqcup\bigl\{(a,C) \bigr\}_{a\in\llbracket 1,p\rrbracket}$, where $a$ runs through the set of edges\footnote{The letter $a$ stands for \textit{ar\^etes}, which means edges in French.}. In some sense, $[p]^I$ is the set of all potential colored edges with color set $I$ that can be obtained from the $p$ edges of a map. 

\begin{definition}[Orbits]
\label{def:Orbits}
We call orbits for $f$ and $\alpha$, the cycles respectively induced by $\gamma\alpha$,  $\hat\gamma_f\, \hat \alpha$,  and $\alpha$  on the sets $[p]^A$, $[p]^{B,C}$, and $[p]^D$.\footnote{Where we canonically identified $[p]^A$ and $[p]^D$ with $\llbracket 1, p\rrbracket$.}
\end{definition}
Of course, these are in bijection with the loops in the box representation.
Note that the orbits involving the colors $A$ and $D$ are already understood from Section \ref{sec:FirstSec} (corresponding to the bipartite quantum system case), and only the orbits involving the colors $B$ and $C$ present a new behavior.

\

\noindent{\bf Combinatorial maps.} 
As just mentioned, as far as the colors $A$ and $D$ are concerned, the behavior is unchanged from Section~\ref{sec:FirstSec}. So, naturally, if we  define the map $\cM = (\gamma, \alpha)\in\bM_p$, the orbits corresponding to the color $A$ are simply the faces of $\cM$, while there is one orbit of color $D$ per white vertex of the map. Colors $A$ and $D$ therefore contribute with a factor $N_A^{F(\cM)}N_D^{V(\cM) - 1}$. 

We now focus on the orbits involving colors $B$ and $C$.
We can provide each edge  $a\in\llbracket 1,p\rrbracket$  of the map with a color $j(a)\in\{B,C\}$: the non traced color in $\{B,C\}$ for the matrix number $a$ of the word $f$ (\textit{i.e.} the color different from $A$ in $f(a)$). For instance, the moment $\bE(\Tr W_{AB}W_{AC})$ of the trace of Fig.~\ref{fig:2-marginals-moment} involves the two maps in $\bM_2$ (Fig.~\ref{fig:MapExTrW2}), but in addition, one edge carries the color $B$, the other the color $C$. We consider the set of the colored edges,
\be 
[p]^{B,C}_f = \bigl\{\bigl(a,j(a)\bigl) \bigr\}_{a\in\llbracket 1,p\rrbracket}.
\ee 
It is the subset of  $[p]^{B,C}$ corresponding to the colors which are not traced.

Some orbits contain at least one element of $[p]^{B,C}_f$, or equivalently, follow at least one edge of the map $\cM$. In the map language, when such an orbit reaches a corner on the black vertex after having followed a colored edge $\bigl(a, j(a)\bigr)$, $a\in\llbracket 1,p\rrbracket$ and $j(a)\in \{B, C\}$, it will go to the colored edge $\bigl(a+1, j(a+1)\bigr)$. Apart from the information on the color, the orbit locally behaves, around the black vertex,
as the face of the map $\cM$, which goes from $a$ to $a+1$. 

However, the orbits and the faces of the map do not behave alike around white vertices. Indeed, arriving from a colored edge $(a, j(a))$ at a white vertex $v_\circ$, the orbit goes to the first edge around $v_\circ$ counterclockwise that also carries the color $j(a)$, namely
\be 
\label{eq:FirstEdgeSameColorWhiteVert}
 \alpha^q(a),\quad \text{where} \quad q=\min \bigl\{s\in\bN^\ast \mid j(a) =j( \alpha^s(a))\bigr\}.
\ee 
If the integer $q$ is greater than one, the behavior differs from that of the face, which goes to $\alpha(a)$. A simple way of solving this issue is to locally split the white vertex into two vertices, one for each color $B$ and $C$, as illustrated in Fig.~\ref{fig:M23}.
\begin{figure}[!ht]
\includegraphics[scale=0.9]{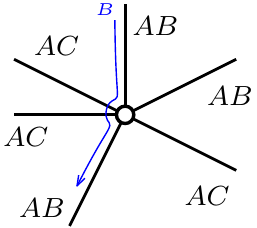}\hspace{1cm}\raisebox{
1cm}{$\rightarrow$}\hspace{1cm}\includegraphics[scale=0.9]{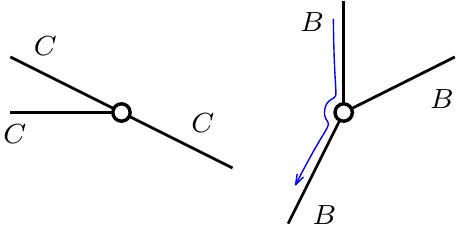}
\caption{\label{fig:M23} Local duplication of white vertices to obtain the map $\cM_f$. }
\end{figure}
Note that this operation has no effect if all the edges incident to the white vertex have the same color.
Otherwise, duplicating as in Fig.~\ref{fig:M23} a white vertex $v_\circ$ with both incident colors $B$ and $C$, we obtain a new map in $\bM_p$ with one additional white vertex, and the number of orbits remains unchanged. However, now, around the two vertices of color $B$ and $C$ resulting from splitting $v_\circ$, the orbits and the faces behave alike. 

Performing this operation on all the white vertices, we obtain a map $\cM_f \in \bM_p$, with the same black vertex  $\gamma$ as $\cM$, but with two permutations $\alpha_B$ and $\alpha_C$, which are simply the restrictions of $\alpha$ to the two sets of edges that respectively carry the colors $B$ and $C$. This defines a tricolored map $\cM_{f} = (\sigma_\bullet, \alpha_B, \alpha_C)$, where $\sigma_\bullet=\gamma$, and such that edges only link white vertices and edges of color $B$ or $C$ (\textit{i.e.} it is also a bicolored map $(\sigma_\bullet, \alpha_B \alpha_C)$ in the sense of Def.~\ref{def:comb-maps}). Because if $j(a)$ is \textit{e.g.}~$B$, $\alpha_B(a)= \alpha^q(a)$ defined in \eqref{eq:FirstEdgeSameColorWhiteVert}, the faces of $\cM_f$ are precisely the orbits, and therefore the colors $B$ and $C$ contribute with a factor $N_{B,C}^{F(\cM_f)}$. 
\smallskip

And lastly, some orbits do not follow any edge of the map. They correspond to the restriction of $\hat \gamma_f \hat \alpha$ to the set 
\be
[p]^{B,C} \setminus [p]^{B,C}_f,
\ee
and in the map language, to white vertices of $\cM$ whose incident edges all share the same color. Such white vertices are those left invariant by the operation of  Fig.~\ref{fig:M23}, and therefore their number is just 
\be 
\label{eq:PureMf}
2V(\cM)-V(\cM_f)-1.
\ee 
Indeed, $2(V(\cM)-1)-(V(\cM_f)-1)$ counts one for each white vertex with only adjacent edges of a single color, while it counts zero for each white vertex with both adjacent colors. 

We have therefore the following formulation of Thm.~\ref{thm:moments-fixed-N}.

\begin{theo31*}
 For $f$ a word of length $p$ in the alphabet $\{AB,AC\}$, the mixed moments of the two marginals $W_{AB}, W_{AC}$ are expressed exactly using a sum over combinatorial maps
 \begin{equation}
 \label{eq:general-4partite-combinatorics}
 \bE\Tr(W_f)= \sum_{\cM\in \bM_p} N_A^{F(\cM)}N_D^{V(\cM)-1}N_{B,C}^{L(f,\, \cM)},
 \end{equation}
where $L(f, \cM)$ coincides with $L(f, \alpha)$ and can be expressed as
\be 
\label{eq:LinMap}
L(f,\cM) = F(\cM_f)+ 2V(\cM)-V(\cM_f)-1.
\ee
\end{theo31*}

\

\smallskip

\noindent{\bf A bound on the exponent $L$.} 
We establish now some key results that will be used in the following sections, regarding the range of the functional $L(f, \cdot)$. Note that the other exponents in \eqref{eq:exact-moments-2-marginals} or \eqref{eq:general-4partite-combinatorics} have been dealt with in Lemma~\ref{lem:geodesic-permutations} in terms of permutations, where it has been shown that
$$|\alpha| + |\alpha^{-1}\gamma^{-1}| \geq |\gamma^{-1}| = p-1,$$
with equality if the permutation $\alpha$ is \emph{geodesic with respect to $\gamma^{-1}$}, \textit{i.e.}~it is associated to a non-crossing partition (the cycle structure of $\alpha$ is non-crossing and inside each cycle the elements have the same cyclic ordering as $\gamma^{-1}$, or in terms of maps, using the genus \eqref{eq:EulerCharacSec1}.   The next results are a generalization of the above fact, which corresponds to the situation where $f$ is constant. 

\begin{proposition}
\label{prop:BoundOn_L}
 For $f$ a word of length $p$ in the alphabet $\{AB,AC\}$ and $\alpha\in \cS_p$, we have  the bound
\be 
 L(f,\cM) \le p+1,
\ee
with equality iff the map $\cM$ is planar, and all the edges of $\cM$ incident to a common white vertex have the same color.

In terms of the permutation $\alpha$, $L(f, \alpha) = p+1$ with equality iff $\alpha$ is geodesic \emph{and} $\alpha \leq \ker f$ (for the usual partial order on partitions), where $\ker f$ is the partition having (at most) two blocks, $f^{-1}(\{AB\})$ and $f^{-1}(\{AC\})$.
\end{proposition}

\noindent{\it \underline{Proof in the map language.}\ } 
The relation of the Euler characteristic for the map $\cM_f$ writes 
\be 
2-2g(\cM_f)=V(\cM_f) - p + F(\cM_f).
\ee
we can therefore rewrite \eqref{eq:LinMap} as
\begin{equation}
 L(f,\cM) = p + 1 -2g(\cM_f)-2\Delta_f(\cM),
\end{equation}
where we have denoted $\Delta_f(\cM)$ the number of white vertices reached by both colors $B$ and $C$,
\be
\Delta_f(\cM)=V(\cM_f)-V(\cM).
\ee 
In particular, this vanishes iff  in $\cM$, all the edges incident to a common white vertex have the same color, and is positive otherwise. Therefore,  $L(f,\cM) = p + 1$ iff $\Delta_f(\cM)= g(\cM_f)=0$. Moreover, if $\Delta_f(\cM)=0$, the maps $\cM_f$ and $\cM$ coincide, so that  $L(f,\cM) = p + 1$ iff $\Delta_f(\cM)= g(\cM)=0$. 
\qed

\

Note that, using the again the genus formula for $\cM$, we obtain the following exact expression for the moments
 \begin{equation}
 \label{eq:ExactExpressionMomentsUsing-Delta-g}
\bE\Tr(W_f)= \sum_{\cM\in \bM_p} N_A^{p+2-2g(\cM)-V(\cM)}N_D^{V(\cM)-1}N_{B,C}^{p+1-2g(\cM_f)-2\Delta_f(\cM)}.
 \end{equation}

\

\noindent{\it \underline{Proof in terms of permutations.}\ } Let us first restate the proposition in the form of a lemma. To recover Prop.~\ref{prop:BoundOn_L}, it suffices to use the fact that 
\be 
\label{eq:CyclesvsDistancePerm2p}
\#(\hat \gamma_f \hat\alpha) + |\hat \gamma_f \hat\alpha|= 2p.
\ee
\begin{lemma}
\label{lem:L-f-alpha}
For any $f \in \{AB, AC\}^p$ and $\alpha \in \mathcal S_p$, we have
$$|\hat \gamma_f \hat\alpha | \geq p-1,$$
with equality iff $\alpha$ is geodesic \emph{and} $\alpha \leq \ker f$ (for the usual partial order on partitions), where $\ker f$ is the partition having (at most) two blocks, $f^{-1}(\{AB\})$ and $f^{-1}(\{AC\})$.
\end{lemma}
\begin{proof}
Let us start by re-writing
the permutation $\hat \gamma_f \in \mathcal S_{2p}$ (we are using the labeling of \eqref{eq:LabellingEdgesABCD}, $[p]^{B,C} = \{(1,B), \ldots, (p, B),(1,C), \ldots, (p, C)\}$)
\begin{equation}\label{eq:hat-f-delta}
\hat \gamma_f = \delta_C \left( \gamma^B \oplus \mathrm{id}^C \right) \delta_C,
\end{equation}
where $\gamma^B$ (resp.~$\mathrm{id}^C$) denotes the permutation $\gamma$ (resp.~$\mathrm{id}$) acting on $[p]^B=\{(1,B), \ldots, (p, B)\}$ (resp.~on $[p]^C=\{(1,C), \ldots, (p, C)\}$), and $\delta_C$ is a product of transpositions,
$$\mathcal S_{2p} \ni \delta_C := \prod_{a \in f^{-1}(\{AC\})} \bigl((a,B) (a,C)\bigr).$$
Indeed, one can check formula \eqref{eq:hat-f-delta} by comparing it with equations \eqref{eq:hat-f-B} and \eqref{eq:hat-f-C}.

We now apply \cite[Lemma 5.5]{collins2007second}	to the permutations $\hat \gamma_f,\hat \alpha \in \mathcal S_{2p}$:
\begin{equation}\label{eq:inequality-alpha-beta-vee}
|\hat \gamma_f\hat \alpha| + |\hat \gamma_f| + |\hat \alpha| \geq 2|\hat \gamma_f \vee \hat \alpha|,
\end{equation}
where the join operation $\vee$ on the right hand side should be understood as acting on the partitions induced by the cycles of the two permutations (the join of two partitions is their least upper bound). The notation $| \cdot |$ is extended to partitions as $|\pi| := 2p - \#\pi$, for a partition $\pi$ of $[2p]$, where $\#\pi$ denotes the number of blocks of $\pi$. Since $\hat \gamma_f$ has one cycle of length $p$ and $p$ fixed points, and $\hat \alpha = \alpha^B \oplus \alpha^C$ (using the same notation as above), the inequality becomes
\begin{align*}
|\hat \gamma_f\hat \alpha| &\geq 4p - 2\#(\hat \gamma_f \vee \hat \alpha) - (p-1) - (2p - 2\#\alpha)\\
&= p+1 + 2(\#\alpha - \#(\hat \gamma_f \vee \hat \alpha)).
\end{align*}
We claim that $\#(\hat \gamma_f \vee \hat \alpha) = 1 + \pure(\alpha, f)$, where $\pure(\alpha,f)$ denotes the number of cycles of $\alpha$ on which the restriction of $f$ is constant, i.e.~in the map language, the white vertices whose incident edges all carry the same colors, so that $\pure(\alpha,f)$ is given by \eqref{eq:PureMf}. Indeed, the blocks of the partition $\hat \gamma_f \vee \hat \alpha$ correspond to the blocks of $\alpha^B$ and $\alpha^C$ merged by the block of length $p$ of $\hat \gamma_f$, together with each block of $\alpha^{B,C}$ matching only fixed points of $\hat \gamma_f$. It is now clear that the latter are exactly blocks of $\alpha^B$ on which $f \equiv AC$ or blocks of $\alpha^C$ on which $f \equiv AB$; we conclude that there are exactly $\pure(\alpha,f)$ of those, proving the claim.  Hence, 
$$|\hat \gamma_f\hat \alpha| \geq p-1 + 2(\#\alpha - \pure(\alpha,f)) \geq p-1,$$
proving the inequality. 

Let us now characterize the permutations $\alpha$ which saturate the inequality (assuming $f$ is fixed). First, the last inequality above should be an equality, so $\alpha = \pure(\alpha,f)$; in other words, $f$ should be constant on the blocks of $\alpha$, which is the condition $\alpha \leq \ker f$ appearing in the statement of the lemma. Moreover, we can easily see that this condition is equivalent to the fact that the permutations $\hat \alpha$ and $\delta_C$ commute. Indeed, by direct computation, we can see that, for any $a \in \llbracket 1, p\rrbracket$ and any $Z=B,C$, we have 
\begin{equation}\label{eq:delta-C-hat-alpha}
[\delta_C \hat \alpha \delta_C](a,Z) = \begin{cases}
(\alpha(a),Z) &\quad \text{ if } f(a) = f(\alpha(a))\\
(\alpha(a),{\bar Z}) &\quad \text{ if } f(a) \neq f(\alpha(a)),
\end{cases}
\end{equation}
where we denote by $\bar Z$ the complement of $Z$ in $\{B,C\}$. Hence, $\delta_C \hat \alpha \delta_C = \hat \alpha$ iff $f$ is constant on the cycles of $\alpha$, which is the claimed statement. Using \eqref{eq:hat-f-delta}, we have
$$|\hat \gamma_f\hat \alpha| = |\delta_C \left( \gamma^B \oplus \mathrm{id}^C \right) \delta_C \hat \alpha | = |\left( \gamma^B \oplus \mathrm{id}^C \right) \hat \alpha | = |(\gamma\alpha)^B \oplus \alpha^C| = |\alpha| + |\alpha\gamma|\geq p-1,$$
with equality iff $\alpha$ is geodesic, proving the other equality condition in the statement and finishing the proof.
\end{proof}

\

\smallskip

\noindent{\bf Reformulation of the exponent $L$.}  In this paragraph, we give a more intuitive expression for the exponent $L$. We define $\alt(f,\alpha)$ as the total number of changes of colors around the cycles of $\alpha$
 \be 
 \label{eq: Definition-Alt}
 \alt(f,\alpha) = \bigl \lvert \{a \in \llbracket 1, p \rrbracket \, : \, f(a) \neq f(\alpha(a))\} \bigr \rvert. 
 \ee 
In the map language, $\alt(f,\cM)$ is the total number of corners around white vertices of $\cM$ whose incident edges $a$ and $\alpha(a)$ have different colors $j(a) \neq j(\alpha(a))$.

\begin{proposition}
\label{prop:L-and-alt}
For any $f \in \{AB, AC\}^p$ and $\alpha \in \mathcal S_p$, where $\alpha$ is geodesic (or equivalently, where $\cM=(\gamma, \alpha)$ is planar), we have
\be 
\label{eq:L-is-alt-Plan}
L(f,\alpha) = p+1 - \alt(f,\alpha).
\ee 

\end{proposition}

We also provide two proofs, first in terms of permutations, and then in terms of maps. In the following, for $\cM=(\gamma, \alpha)$, we identify $L(f,\alpha)$ and $L(f,\cM)$, as well as $\alt(f,\alpha)$ and $\alt(f,\cM)$. We stress that \eqref{eq:L-is-alt-Plan} is true only for a geodesic permutation $\alpha$, or equivalently a planar map $\cM$. In general, for non-planar maps, $2(g(\cM_f) + \Delta_f(\cM)) \neq \alt(f,\cM)$. For instance, if $\Delta_f(\cM) = 0$, $\cM_f = \cM$, so that $2(g(\cM_f) + \Delta_f(\cM))= 2g(\cM)$, but $\alt(f,\cM) = 0$. 

\

\noindent{\it \underline{Proof in terms of permutations.}\ }
Let us assume that $\alpha$ is geodesic (w.r.t.~$\gamma^{-1}$, i.e.~it satisfies $|\alpha|+|\alpha \gamma| = |\gamma^{-1}| = p-1$) and prove
\begin{equation}\label{eq:distance-hat-f-hat-alpha-geodesic}
|\hat \gamma_f \hat\alpha | = p-1 + \alt(f,\alpha), 
\end{equation}
which proves the proposition using \eqref{eq:CyclesvsDistancePerm2p}. We write, as before, 
\begin{align*}
|\hat \gamma_f\hat \alpha| &= |\delta_C \left( \gamma^B \oplus \mathrm{id}^C \right) \delta_C \hat \alpha |\\
&= |\delta_C \hat \alpha \delta_C ( \gamma^B \oplus \mathrm{id}^C )| \\
&\leq |\delta_C \hat \alpha \delta_C \hat \alpha^{-1}| + | \hat \alpha ( \gamma^B \oplus \mathrm{id}^C )| \\
&= |\alpha| + |\alpha\gamma| + |\delta_C \hat \alpha \delta_C \hat \alpha^{-1}| \\
&= p-1 + \alt(f,\alpha),
\end{align*}
where we have used the triangle inequality and the fact that 
$$\delta_C \hat \alpha \delta_C \hat \alpha^{-1} = \prod_{a \, : \, f(a) \neq f(\alpha^{-1}(a))} \bigl((a,B) (a,C)\bigr),$$
which follows easily from \eqref{eq:delta-C-hat-alpha}. To  conclude, we need to show that the triangle inequality used above is saturated, which is equivalent to the three permutations
$$\delta_C \hat \alpha \delta_C - \hat \alpha - (\gamma^B)^{-1} \oplus \mathrm{id}^C$$
lying on a geodesic. Since $\delta_C \hat \alpha \delta_C \hat \alpha^{-1}$ is a product of disjoint transpositions, this is in turn equivalent to the fact that, for all $a$ such that $f(a) \neq f(\alpha^{-1}(a))$, the elements $(a,B)$ and $(a,C)$ are contained in the same cycle of the permutation $\varepsilon :=\delta_C \hat \alpha \delta_C(\gamma^B \oplus \mathrm{id}^C)$. This permutation acts in the following way:
\begin{align*}
(a,B) &\mapsto \begin{cases}
(\alpha(a+1), B) &\quad \text{ if } \qquad f(\alpha(a+1)) = f(a+1) \\
(\alpha(a+1),C) &\quad \text{ if } \qquad f(\alpha(a+1)) \neq f(a+1) 
\end{cases}\\
(a,C) &\mapsto \begin{cases}
(\alpha(a),C) &\quad \text{ if } \qquad f(\alpha(a)) = f(a) \\
(\alpha(a), B) &\quad \text{ if } \qquad f(\alpha(a)) \neq f(a).
\end{cases}
\end{align*}
In other words, on the $B$-level, $[p]^B$, $\varepsilon$ acts as $\alpha\gamma$, which is the permutation associated to the non-crossing partition $\alpha^{\mathrm{Kr}}$ (see \cite[Exercise 18.25 and Remark 23.24]{nica2006lectures}), while on the $C$-level, $[p]^C$, $\varepsilon$ acts as $\alpha$. 
Using the fact that $\alpha$ is non-crossing and the definition of $\alpha^{Kr}$ (see \cite[Definition 9.21]{nica2006lectures}, and note that in our notation, we also have $\bar a > a$), we can easily see that $(a,B)$ and $(a,C)$ belong to the same cycle of $\varepsilon$, whenever $f(a) \neq f(\alpha^{-1}(a))$; we refer the reader to Figure~\ref{fig:alpha-alpha-Kr} for a graphical illustration of this fact. \qed
\begin{figure}[!ht]
	\centering
	\includegraphics[width=0.45\textwidth]{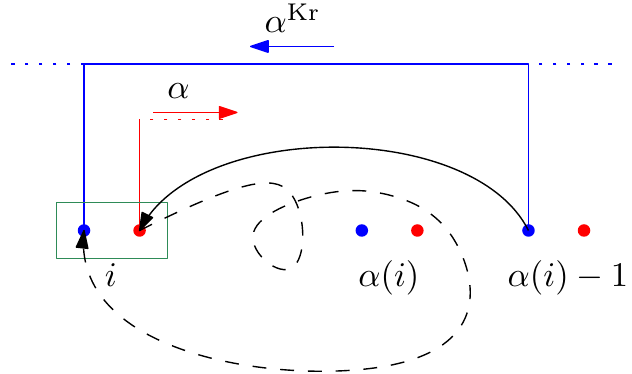}
	\caption{Diagram showing that whenever $f(a) \neq f(\alpha^{-1}(a))$, $(a,B)$ and $(a,C)$ belong to the same cycle of $\varepsilon =\delta_C \hat \alpha \delta_C(\gamma^B \oplus \mathrm{id}^C)$. Since the partition $\alpha^{\operatorname{Kr}}$ is non-crossing (blue partition), the element $(a,C)$ (in red) must ``escape'' the interval $[a,\alpha(a)]$ through $(a,B)$ (in blue).}
	\label{fig:alpha-alpha-Kr}
\end{figure}

\

\noindent{\it \underline{Proof in the map language.}\ } 
We assume that the map $\cM\in\bM_p$ is planar, and prove by induction that 
\be
\alt(f, \cM) = 2g(\cM_f)+2\Delta_f(\cM).
\ee
If $v_\circ$ is a white vertex and $v_\bullet$ is the only black vertex, we denote $\cM^{v_\circ }$  the submap obtained from $\cM$ by keeping only $v_\circ$ and  $v_\bullet$ and the edges linking them. We can apply the operation of Fig.~\ref{fig:M23}  to the white vertex of $\cM^{v_\circ }$, thus obtaining a map $\cM_f^{v_\circ }$.
Because the map $\cM$ is planar, the map $\cM_f$ can be constructed by recursively inserting the maps $\cM_f^{v_\circ }$ in the corners in the appropriate way (inverse operation of that shown in Fig.~\ref{fig:VertexSplitting}), where $v_\circ$ spans the white vertices of $\cM$. Therefore, any operation on the edges of a given $\cM_f^{v_\circ }$ does not affect the faces incident to other $\cM_f^{v'_\circ }$, and the genus of $\cM_f$ is the sum of the genera of all the $\cM_f^{v_\circ }$
\be
\label{eq:GenusSumWhite}
g(\cM_f) = \sum_{\substack{{v_\circ \in \cM}\\{\text{white vertex}}}} g(\cM_f^{v_\circ }).
\ee
Note that this is not true for a non-planar map $\cM$.
We can therefore study what happens locally, for a single $\cM^{v_\circ }$. The second step is to notice that if two consecutive edges around a $v_\circ$ have the same color,  then removing one or the other will not affect the genus $g(\cM_f^{v_\circ })$. We can therefore consider that there are no two consecutive edges of the same color. The case with 3 edges of each color is shown in Fig.~\ref{fig:MWhiteBlack} {\bf a)} below. 
\begin{figure}[!ht]
\includegraphics[scale=0.8]{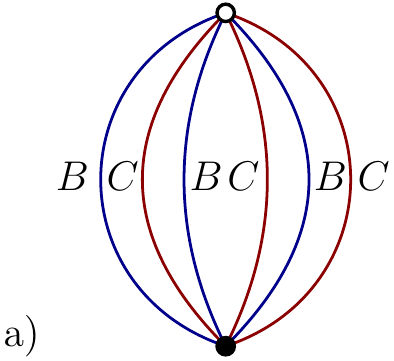}\hspace{1cm}
\includegraphics[scale=0.8]{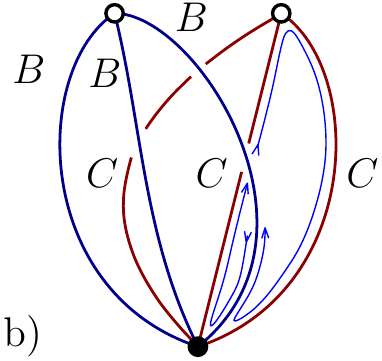}\hspace{1cm}
\includegraphics[scale=0.8]{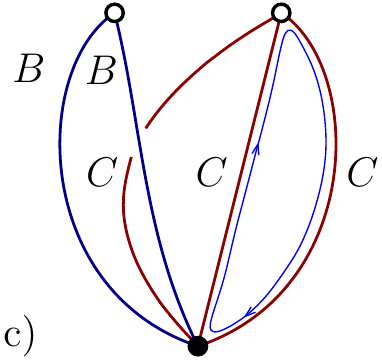}\hspace{1cm}
\includegraphics[scale=0.8]{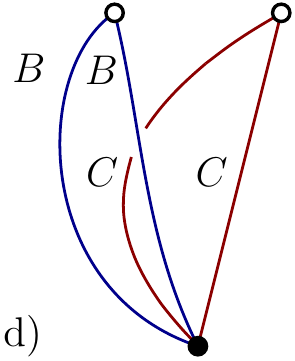}
\caption{\label{fig:MWhiteBlack} a) A six edges $\cM^{v_\circ }$. b-d) Recursive operations on $\cM_f^{v_\circ }$.  }
\end{figure}

Firstly, in the case where $\cM^{v_\circ }$ is made of only two edges of two different colors, we indeed have $2=\alt(f, \cM^{v_\circ}) = 2g(\cM_f^{v_\circ})+2\Delta_f(\cM^{v_\circ})=2\cdot 0+2\cdot(3-2)=2$.
If now there are more than two edges, it is easy to see that a single face visits all the edges twice in $\cM_f^{v_\circ }$ (see the dotted face in  Fig.~\ref{fig:MWhiteBlack} {\bf b)}). Furthermore, these edges are not bridges, and therefore, deleting any edge in $\cM_f^{v_\circ }$, the genus decreases by one (there is one more face and one less edge), and the number of corners incident to edges of different colors around $v_\circ $ in $\cM^{v_\circ }$ decreases by two. In the resulting map (Fig.~\ref{fig:MWhiteBlack} {\bf c)}), two edges have two incident faces. Deleting one of them, the genus does not vary, nor the number of corners incident to edges of different colors, and we recover a map with less edges and with the same property (Fig.~\ref{fig:MWhiteBlack} {\bf d)}). By induction we deduce that the number of corners incident to edges of different colors around $v_\circ $ in $\cM^{v_\circ }$ is twice the genus of $\cM_f^{v_\circ }$ plus two. Summing over white vertices \eqref{eq:GenusSumWhite}, we obtain the sought relation $\alt(f, \cM) = 2g(\cM_f)+2\Delta_f(\cM)$. \qed

\smallskip

\subsection{The balanced asymptotical regime}\label{sec:2-marginals-large-BC}
\label{subsec:Large-HBC}

\

\medskip

We focus in this section on the asymptotic regime where $N:=\dim \mathcal H_B = \dim \mathcal H_C \to \infty$, which we call balanced; for the case when the size of the $B,C$ subsystems stays bounded, see Section \ref{sec:2-marginals-fixed-BC}. In this regime, where $\mathcal H_{B,C}$ grow, it turns out that the asymptotic behavior of $\mathcal H_{A,D}$ is not so important, as long as the ratio $\dim \mathcal H_D /  \dim \mathcal H_A$ converges to a positive constant as $N \to \infty$. We present next the main result of this section, and discuss several particular asymptotic scenarios later. 

\begin{theorem}\label{thm:2-marginals-large-BC}
Let $X_N \in \mathbb C^{N_A} \otimes \mathbb C^N \otimes \mathbb C^{N} \otimes \mathbb C^{N_D}$ be a sequence of random Gaussian tensors, where $N_{A,D}$ are arbitrary functions of $N$ satisfying $N_D \sim cN_A$ as $N \to \infty$, for some constant $c \in (0,\infty)$. Then, the normalized marginals $(N_A^{-1}N^{-1}W_{AB}^{(N)}, N_A^{-1}N^{-1}W_{AC}^{(N)})$ defined in \eqref{eq:def-2-marginals} converge in distribution, as $N \to \infty$, to a pair of identically distributed and free elements $(x_{AB},x_{AC})$, where $x_{AB}$ and $x_{AC}$  have a  $\mathrm{MP}_{c}$ distribution. Equivalently, for any word in the two marginals $f \in \{AB, AC\}^p$, we have
\begin{equation}\label{eq:moments-marginals-AB-AC-balanced}
\lim_{N \to \infty} \mathbb E (N_A N)^{-p-1} \operatorname{Tr} \prod_{1 \leq i \leq p}^{\longleftarrow} W_{f(i)}^{(N)} = \sum_{\alpha \in \mathrm{NC}(p), \, \alpha \leq \ker f} c^{\#\alpha},
\end{equation}
where $\ker f$ is the partition having two blocks corresponding to the occurrences of $AB$ (resp.~$AC$) in the word $f$. 

This rewrites in terms of planar bicolored maps with one white vertex as  
\begin{equation}\label{eq:moments-marginals-AB-AC-balanced-MAPS}
\lim_{N \to \infty} \mathbb E (N_A N)^{-p-1} \operatorname{Tr} \prod_{1 \leq i \leq p}^{\longleftarrow} W_{f(i)}^{(N)} = \sum_{\cM \in \bM_p^{0}, \, \Delta_f(\cM) =0} c^{V(\cM) - 1},
\end{equation}
where the sum is restricted to maps whose white vertices only have incident edges of the same color, and we recall that $\bM_p^{0}$ is the subset of the elements of $\bM_p$ of vanishing genus.
\end{theorem}
\begin{proof}
We first prove formula \eqref{eq:moments-marginals-AB-AC-balanced}. Starting from the exact moment formula of Theorem~\ref{thm:moments-fixed-N}, let us analyze the contribution of $N$, through its exponent $L(f,\alpha)$. From Proposition~\ref{prop:BoundOn_L}, 
$L(f,\alpha) \le  p+1,$
with equality iff $\alpha$ is geodesic and $\alpha \leq \ker f$. Hence, we have
\begin{align*}
\mathbb E  \operatorname{Tr} W^{(N)}_f &= (1+o(1))N^{p+1} \sum_{\alpha \in \operatorname{NC}(p),\, \alpha \leq \ker f} N_A^{\#(\gamma\alpha)}N_D^{\#\alpha} \\
&=  (1+o(1))(N_AN)^{p+1} \sum_{\alpha \in \operatorname{NC}(p),\, \alpha \leq \ker f} \left( \frac{N_D}{N_A}\right)^{\#\alpha} \\
&= (1+o(1))(N_AN)^{p+1} \sum_{\alpha \in \operatorname{NC}(p),\, \alpha \leq \ker f} c^{\#\alpha},
\end{align*}
proving the claimed formula. Above, we have used the key fact that, for the surviving $\alpha$ terms (\textit{i.e.}~the permutations which are geodesic w.r.t.~$\gamma^{-1}$), $\#(\gamma\alpha) = p+1 - \#\alpha$.

We now show how the moment formula \eqref{eq:moments-marginals-AB-AC-balanced} implies the main claim. Using the moment-cumulant formula \cite[Proposition 11.4]{nica2006lectures}, one can read the asymptotic free cumulants directly off the moment formula:
$$\lim_{N \to \infty} \mathbb E (N_A N)^{-p-1} \operatorname{Tr} \prod_{1 \leq i \leq p}^{\longleftarrow} W_{f(i)}^{(N)} = \sum_{\alpha \in \mathrm{NC}(p)} \quad  \prod_{b \text{ block of }\alpha} c\mathbf{1}_{f \text{ is constant on } b} .$$
Hence, mixed cumulants vanish (implying freeness, see \cite[Theorem 11.16]{nica2006lectures}), and the distribution of the limiting variables $x_{AB}$, $x_{AC}$ is Mar{\v{c}}enko-Pastur of parameter $c$, ending the proof.
\end{proof}

\begin{remark}\label{rem:balanced-4-partite-subcases}
As a special case of the result above, one can consider the case
where all the Hilbert spaces have, up to constants, the same dimension $\dim \mathcal H_A = \lfloor c_1N\rfloor$, $\dim \mathcal H_D = \lfloor c_4N \rfloor$ and $\dim \mathcal H_{B} = \dim \mathcal H_C = N$. Then, the normalized marginals $(c_1N^2)^{-1} (W_{AB},W_{AC})$ converge in moments, as $N \to \infty$, towards two free elements having $\mathrm{MP}_{c_4/c_1}$ distribution.  The multi-partite equivalent of this result will be considered in Section \ref{sec:BalancedGeneral}.

Similarly, when the subsystems $A$ and $D$ have fixed dimension $\dim \mathcal H_A = k$, $\dim \mathcal H_D = l$ and $\dim \mathcal H_{B} = \dim \mathcal H_C = N$ the normalized marginals $(kN)^{-1} (W_{AB},W_{AC})$ converge in moments, as $N \to \infty$, towards two free elements having $\mathrm{MP}_{l/k}$ distribution. 
\end{remark}

\begin{remark}
            One can interpret the asymptotical freeness of the two marginals $W_{AB}$ and $W_{AC}$ in the following way: the two marginals behave as if they come from \emph{independent} random tensors $X$ and $Y$: 
$$W_{AB} = [\operatorname{id} \otimes \operatorname{id} \otimes \operatorname{Tr} \otimes \operatorname{Tr}](XX^*) \quad \text{ and } \qquad  \tilde W_{AC}=[\operatorname{id} \otimes \operatorname{Tr} \otimes \operatorname{id} \otimes \operatorname{Tr}](YY^*).$$
Indeed, for the random matrices above, the conclusion of the theorem above follows from the very general asymptotic freeness results of Voiculescu: the random matrices $W_{AB}$ and $\tilde W_{AC}$ are independent and unitarily invariant, and they converge to Mar{\v{c}}enko-Pastur elements. 

One can understand this parallel using the fact that, in the asymptotical regime under consideration here, the amount of fresh randomness (the Hilbert spaces $\mathcal H_{B,C}$) is growing. This situation is to be contrasted with the behavior of the marginals in the fixed $B,C$ regime discussed in Section~\ref{sec:2-marginals-fixed-BC}.
\end{remark}

\begin{remark}\label{rem:4-partite-balanced-QIT}
The remark above has interesting applications to quantum information theory. As the quantum marginals $\rho_{AB}$ and $\rho_{AC}$ are rescaled versions of the Wishart matrices $W_{AB,AC}$, the same asymptotic freeness result holds, with a different scaling (precisely, it is the random matrices $N_D N \rho_{AB}, N_D N\rho_{AC}$ which are asymptotically free). Thus, the previous remark implies that, in the large $N$ limit, the quantum marginals $\rho_{AB}$ and $\rho_{AC}$ ``forget'' that they are marginals of the same quantum state $|\psi\rangle_{ABCD}$ and behave like independent random density matrices from the induced ensemble with parameter $c = \lim N_D/N_A$. In particular, these marginals become \emph{uncorrelated} asymptotically, the intuition for this fact being that the amount of ``fresh randomness'' from the systems $\mathcal H_B$ and $\mathcal H_C$, which have dimension growing to infinity, is enough to erase the correlations from system $\mathcal H_A$, and this independently on the ratio $\dim \mathcal H_A$ vs.~$\dim \mathcal H_{B,C}$. 
\end{remark}

As an application of Theorem~\ref{thm:2-marginals-large-BC}, let us consider the product of the two marginals $W_{AB}$ and $W_{AC}$, or, to be exact, its self-adjoint version $P:=W_{AB}^{1/2} W_{AC} W_{AB}^{1/2}$. Applying Theorem~\ref{thm:2-marginals-large-BC}, the random matrix $P$ converges in moments to the element $x_{AB}^{1/2}x_{AC}x_{AB}^{1/2}$, where $x_{AB}$, $x_{AC}$ are two free elements having distribution $\mathrm{MP}_{c}$. In free probability theory, the (self-adjoint) multiplication operation of free elements is known as the \emph{free multiplicative convolution} (denoted by $\boxtimes$), see \cite[Lecture 14]{nica2006lectures}. In our case, we are interested in the probability measure $\mathrm{MP}_{c} \boxtimes \mathrm{MP}_{c} = \mathrm{MP}_{c}^{\boxtimes 2}$. Exact formulas for the densities of the above distributions have been computed in \cite{penson2011product} in the case $c=1$ and in \cite{dupic2014spectral} in the general case. We compare Monte Carlo simulations to the exact densities in Figure~\ref{fig:2-marginals-balanced-1-2}.

\begin{figure}[!ht]
	\centering
	\includegraphics[width=0.45\textwidth]{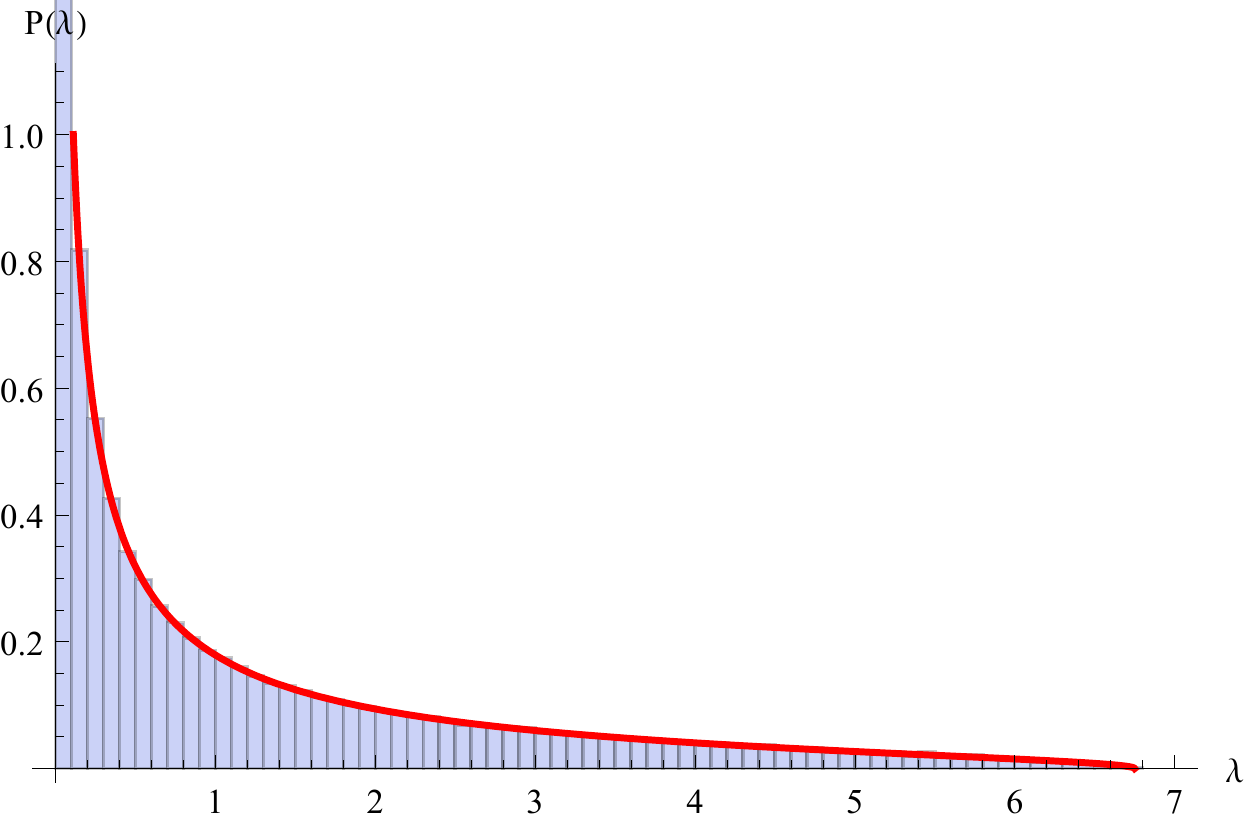} \qquad
	\includegraphics[width=0.45\textwidth]{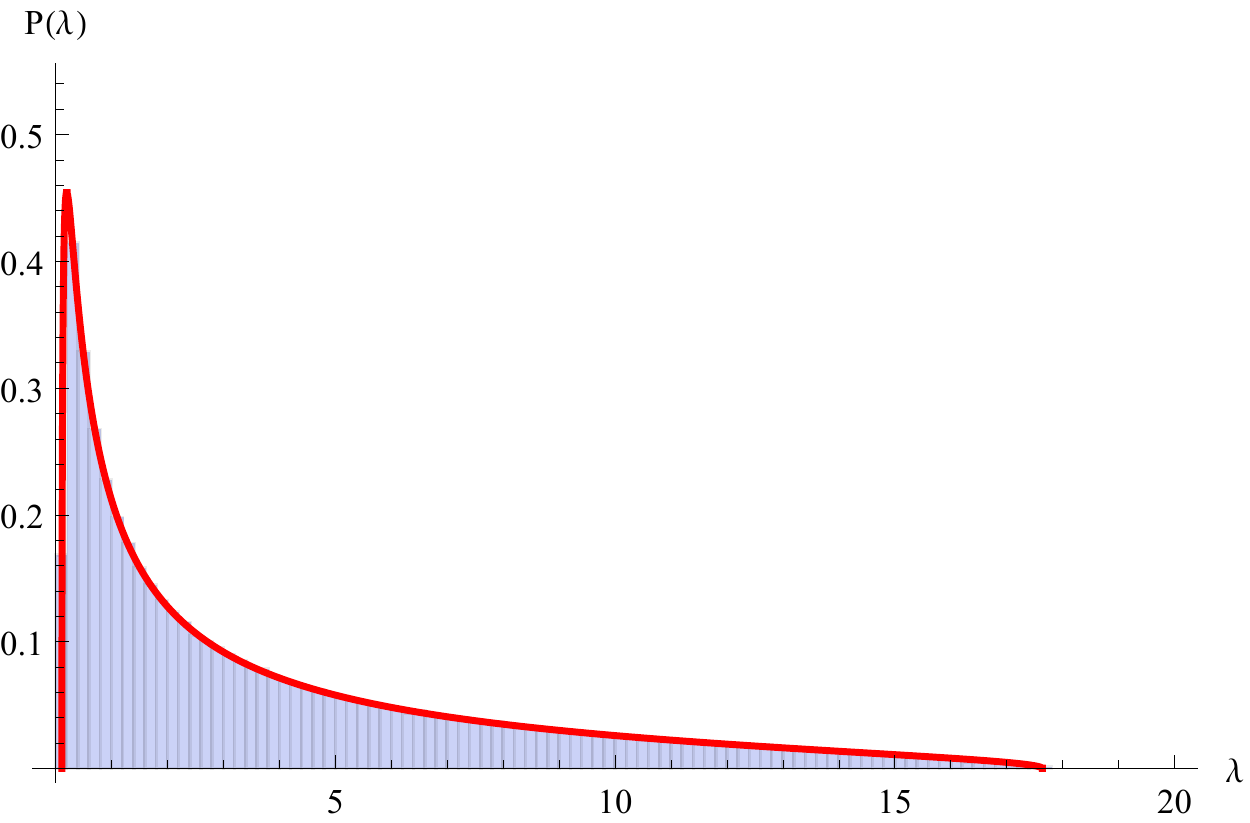}
	\caption{The density of the free multiplicative square of the Mar{\v{c}}enko-Pastur distribution $\mathrm{MP}_c$ for $c=1$ (left) and $c=2$ (right) versus Monte-Carlo simulations. The supports of the two probability measures are $[0, 27/4]$ (left) and $[(71-17 \sqrt{17})/8,(71+17 \sqrt{17})/8] \approx [0.11, 17.64]$ (right).}
	\label{fig:2-marginals-balanced-1-2}
\end{figure}
We discuss next another formulation of the asymptotic moment formula \eqref{eq:moments-marginals-AB-AC-balanced}, useful in practice when one has to evaluate specific mixed moments in the $AB$/$AC$ marginals. 

\begin{proposition}
\label{prop:ABCD-case-c1}
In the same setting as Theorem~\ref{thm:2-marginals-large-BC} with $c=1$, the asymptotical mixed moments of the marginals $W_{AB}$, $W_{AC}$ are indexed by arbitrary words $f \in \{AB, AC\}^p$, or, equivalently, by two integer vectors $r,s$:
$$\prod_{1 \leq i \leq p}^{\longleftarrow} W_{f(i)} = W_{AB}^{r_m} W_{AC}^{s_m} \cdots W_{AB}^{r_2} W_{AC}^{s_2} W_{AB}^{r_1} W_{AC}^{s_1}.$$
Asymptotically, we have
\begin{equation}\label{eq:moments-2-marginals-balanced-sigma-leq-pi}
\lim_{N \to \infty}(N_A N)^{-p-1} \mathbb E  \operatorname{Tr}W_{AB}^{r_m} W_{AC}^{s_m}\cdots W_{AB}^{r_1} W_{AC}^{s_1} = \sum_{\sigma, \pi \in \mathrm{NC}(m),\, \sigma \leq \pi}\mathrm{Cat}_\sigma(r) \mathrm{Cat}_{\pi^\mathrm{Kr}}(s) \operatorname{Mob}(\sigma, \pi),
\end{equation}
where $\operatorname{Cat}_\sigma$ is the multiplicative extension of Catalan numbers
$$\operatorname{Cat}_\sigma(r):=\prod_{b \text{ cycle of } \sigma} \mathrm{Cat}_{\sum_{i \in b}r_i},$$
$\mathrm{Kr}$ denotes the Kreweras complementation (see \cite{kreweras1972partitions} or \cite[Lecture 9]{nica2006lectures}) and $\operatorname{Mob}$ is the M\"obius function on the non-crossing partition lattice (see \cite[Lecture 10]{nica2006lectures}), defined for $\sigma \leq \pi$ by
$$\operatorname{Mob}(\sigma,\pi) := \operatorname{Mob}(\sigma^{-1}\pi) = \prod_{b \text{ cycle of }\sigma^{-1}\pi} (-1)^{|b|-1} \mathrm{Cat}_{|b|-1}.$$
\end{proposition}
\begin{proof}
The result follows from Theorem~\ref{thm:2-marginals-large-BC} and from the formula for the moments of a product of free random variables \cite[Theorem 14.4]{nica2006lectures}: denoting by $x_{AB}$ and $x_{AC}$ the limits in distribution of $W_{AB}$, resp.~$W_{AC}$ and by $\operatorname{tr}$ the expectation in the non-commutative probability space where $x_{AB}$ and $x_{AC}$ live, we have
\begin{align*}
\operatorname{tr}(x_{AB}^{r_{q}} x_{AC}^{s_q}\cdots x_{AB}^{r_1} x_{AC}^{s_1}) &= \sum_{\sigma \in \operatorname{NC}(q)} \operatorname{tr}_\sigma(x_{AB}^{r_q}, \ldots, x_{AB}^{r_1}) \kappa_{\sigma^{\mathrm{Kr}}}(x_{AC}^{s_q}, \ldots, x_{AC}^{s_1}) \\
&= \sum_{\sigma \in \operatorname{NC}(q)} \operatorname{Cat}_\sigma(r) \sum_{\pi' \leq \sigma^{\mathrm{Kr}}}\operatorname{tr}_{\pi'}(x_{AC}^{s_m}, \ldots, x_{AC}^{s_1}) \operatorname{Mob}(\pi',\sigma^{\mathrm{Kr}})\\
&= \sum_{\sigma \leq \pi} \operatorname{Cat}_\sigma(r) \operatorname{Cat}_{\pi^{\mathrm{Kr}}}(s)\operatorname{Mob}(\sigma,\pi),
\end{align*}
where we have used $\pi' = \pi^{\mathrm{Kr}}$.
\end{proof}

As an application of the proposition above, we give below the explicit moments in the two simplest cases. When $q=1$, writing $p = r_1 + s_1$, we have just one term in the sum (corresponding to $\sigma = \pi = \includegraphics[scale=.7, trim=0 5px 0 0]{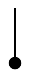}$)
$$\lim_{N \to \infty} (c_1N^2)^{-p-1} \mathbb E  \operatorname{Tr}W_{AB}^{r_1} W_{AC}^{s_1} = \mathrm{Cat}_{r_1}\mathrm{Cat}_{s_1}.$$
When $q=2$, write similarly $p=r_1+s_1+r_2+s_2$; this time, the sum contains three terms, corresponding respectively to $\sigma = \pi = \includegraphics[scale=.7, trim=0 5px 0 0]{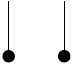}$, $\{\sigma=\includegraphics[scale=.7, trim=0 5px 0 0]{part-12}, \pi = \includegraphics[scale=.7, trim=0 5px 0 0]{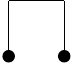}\}$, and $\sigma = \pi = \includegraphics[scale=.7, trim=0 5px 0 0]{part-11}$:
\begin{align*}
\lim_{N \to \infty} (c_1N^2)^{-p-1} \mathbb E \operatorname{Tr}W_{AB}^{r_2} W_{AC}^{s_2}W_{AB}^{r_1} W_{AC}^{s_1} &= \mathrm{Cat}_{r_1}\mathrm{Cat}_{r_2}\mathrm{Cat}_{s_1+s_2} \\
&- \mathrm{Cat}_{r_1}\mathrm{Cat}_{r_2}\mathrm{Cat}_{s_1}\mathrm{Cat}_{s_2}\\
&+\mathrm{Cat}_{r_1+r_2}\mathrm{Cat}_{s_1}\mathrm{Cat}_{s_2}.
\end{align*}

\begin{remark}
The number of terms in equation \eqref{eq:moments-2-marginals-balanced-sigma-leq-pi} is given by the number of pairs $(\sigma, \pi)\in \mathrm{NC}(m)^2$ such that $\sigma \leq \pi$. The set of all such pairs is known as the set of $2$-chains (or intervals) in the lattice of non-crossing partitions and has been enumerated by Kreweras in \cite{kreweras1972partitions}: their number is given by the Fuss-Catalan numbers of order 2 (sequence \href{https://oeis.org/A001764}{A001764} in \cite{oeis})
$$\mathrm{FC}_2(q):= \frac{1}{2q+1} \binom{3q}{q}.$$
Remarkably, these numbers are also the $q$-th moments of the free multiplicative square of the Mar{\v{c}}enko-Pastur distribution of parameter 1:
$$\forall q \geq 1, \qquad \int x^q \mathrm{d} \mathrm{MP}_1^{\boxtimes 2}(x) =   \frac{1}{2q+1} \binom{3q}{q}.$$
\end{remark}

\

\bigskip

\noindent{\bf Bijection with trees.}  
At the end of Section~\ref{sec:FirstSec} (Fig.~\ref{fig:Bij1}), we described a bijective mapping between planar bicolored maps in $\bM_p^{0}$. Naturally, this bijection still applies in the present context, as the sum in \eqref{eq:moments-marginals-AB-AC-balanced-MAPS} still involves elements of $\bM_p^{0}$. However we now have additional coloring information, which will translate into a coloring of the trees. The edges carry a color in $\{B,C\}$, but because the maps involved in \eqref{eq:moments-marginals-AB-AC-balanced-MAPS} satisfy   $\Delta_f(\cM) =0$, all the edges incident to the same white vertex share the same color. We color each white vertex with the color of the incident edges. 
We add a vertex in every face, and new edges linking it to the corners of the face, so that the result remains planar, and delete the initial edges and the black vertex. A tree is always bicolored, but here we see that the vertices of the corresponding tree are partitioned into two sets: white vertices, and vertices carrying a color $B$ or $C$, such that edges may only link white vertices to colored vertices. 
\begin{figure}[!ht]
\raisebox{0.7cm}{\includegraphics[scale=0.7]{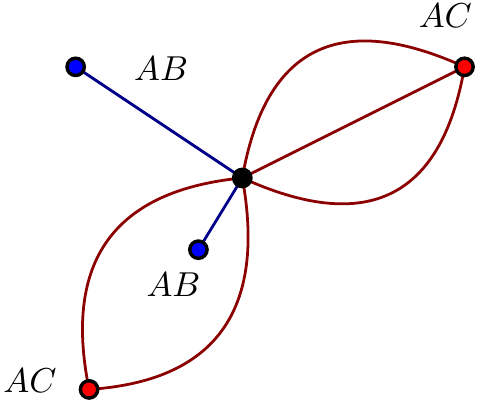}}\hspace{1.2cm}\includegraphics[scale=0.7]{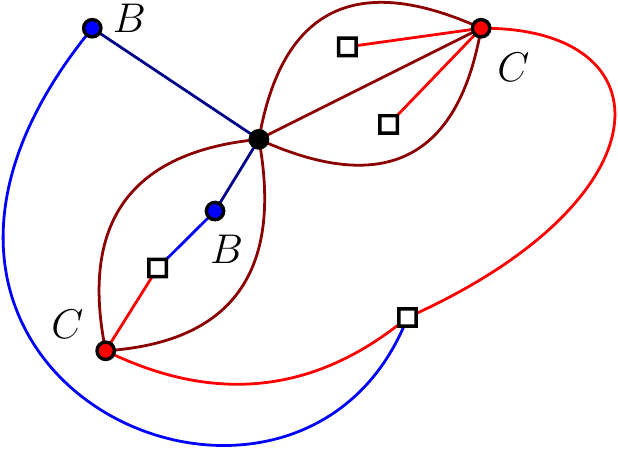}\hspace{1cm}\includegraphics[scale=0.7]{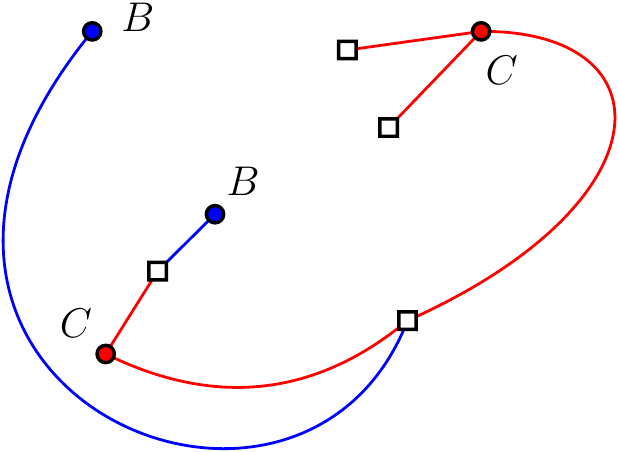}
\caption{\label{fig:Bij2} Bijection between tricolored maps with one black vertex, and tricolored trees in $\bT_{f}$. }
\end{figure}

Furthermore, the labeling of the edges now translates into a labeling of the {\it colored} corners of the tree: starting on corner 1 and following the face around the tree counter-clockwisely, we alternatively encounter corners incident to white vertices, and corners incident to colored vertices, labeled with $p$, then $p-1$, \ldots until returning to corner 1 (the ordering is reversed with respect to the initial bicolored map, because of the dual bijective mapping).

All the labeled colored trees are not images of the colored bicolored maps: the corner corresponding to the edge number $a$ should be incident to a vertex of color $j(a)$. Therefore, starting from corner 1 and going around the tree in the {\it clockwise} direction, we should encounter corners incident to vertices of color $j(1)$, then $j(2)$, then $j(3)$, and so on.   

We denote $\bT_{f}$ the set of tricolored trees with white vertices, and vertices of color $B$ or $C$, and with a such that the colors of the colored vertices encountered in the counterclockwise face are $(j(1), j(p), j(p-1),\ldots j(1))$, which is also the reverse word $f$ (adding color $A$ on every vertex).
In the case where $c=1$, the quantity $\lim_{N \to \infty} \mathbb E (N_A N)^{-p-1} \operatorname{Tr}W_{AB}^{r_q} W_{AC}^{s_q}\cdots W_{AB}^{r_1} W_{AC}^{s_1}$ of Prop.~\ref{prop:ABCD-case-c1} therefore counts the number of tricolored trees in $\bT_{p}$. With the notations of Prop.~\ref{prop:ABCD-case-c1}, we therefore have:
\be
\lvert \bT_{f} \rvert = \sum_{\sigma, \pi \in \mathrm{NC}(q),\, \sigma \leq \pi}\mathrm{Cat}_\sigma(r) \mathrm{Cat}_{\pi^\mathrm{Kr}}(s) \operatorname{Mob}(\sigma, \pi).
\ee
In particular, this is always a non-negative quantity.

Remark that because of the tree structure, we can easily find recursive relations for the quantity $\lvert \bT_{f} \rvert$.\footnote{Note that these relations can also be found from the Schwinger-Dyson equations applied to the matrix formulation of the moments.} In order to write the recursive relations, we rather denote\footnote{For practical reasons,  the labeling is reversed with respect to our usual convention.}
\be
\cW_{u_1,\ldots, u_q}^{d_1,\ldots, d_q} = \lvert \bT_{f} \rvert = \lim_{N \to \infty} (N_A N)^{-p-1} \mathbb E  \operatorname{Tr}W_{AB}^{u_1} W_{AC}^{d_1}\cdots W_{AB}^{u_q} W_{AC}^{d_q}.
\ee 
We find:  
\begin{align}
&\cW_{u_1,\ldots, u_q}^{d_1,\ldots, d_q + 1}=\sum_{k=1}^{q-1}\sum_{s=0}^{d_k-1}\cW_{u_1,\ldots, u_k}^{d_1, \ldots ,s}\cW_{u_{k+1},\ldots, u_q}^{d_{k+1}, \ldots ,d_q+d_k-s} + \sum_{s=0}^{d_q-1}\cW_{u_1,\ldots, u_q}^{d_1, \ldots ,s}  C_{d_q - s}
\\
&\cW_{u_1,\ldots, u_q + 1}^{d_1,\ldots, d_q }=\sum_{k=1}^{q}\sum_{s=0}^{u_k-1}\cW_{u_1,\ldots, u_{k-1},s}^{d_1,\ldots,  d_{k-1}, d_q}\cW_{u_q+u_k - s,u_{k+1},\ldots, u_{q-1}}^{d_k, d_{k+1},\ldots,d_{q-1}} + \sum_{s=0}^{u_q-1}\cW_{u_1,\ldots, s}^{d_1, \ldots ,d_q}  C_{u_q - s}.
\end{align}
However, solving these relations directly is a difficult task, and the solution is found considerably more easily using the techniques of the proof of  Prop.~\ref{prop:ABCD-case-c1}.

\smallskip

\subsection{Comparing with the Fuss-Catalan matrix model}

\

\medskip

In this subsection, we would like to compare the matrix model discussed above in the balanced regime where $c_1=c_4=1$ (\textit{i.e.}~$\mathcal H_A = \mathcal H_B = \mathcal H_C = \mathcal H_D = \mathbb C^N$) with another one having the same asymptotic moments (the Fuss-Catalan numbers), the so-called \emph{free Bessel laws} of parameter $2$ from \cite{banica2011free}. More precisely, the latter matrix model is given by $Q = X_1X_2X_2^*X_1^*$, where $X_{1,2}$ are i.i.d.~$N \times N$ complex Gaussian random matrices. The exact moments of the random matrix $Q$ are given by
\begin{equation}\label{eq:moments-FC2-permutations}
\mathbb E \operatorname{Tr} Q^p = \sum_{\alpha_1, \alpha_2 \in \mathcal S_p} N^{\#(\gamma\alpha_1\alpha_2
) + \#\alpha_1 + \#\alpha_2}.
\end{equation}
Indeed, in our usual representation in maps, the sum is taken over bicolored maps with one black vertex, such that white vertices only have incident edges of the same color. In that sense, white vertices inherit the color 1 or 2 of their incident edges. If $\alpha_1$ (resp.~$\alpha_2$) is the permutation whose cycles encode the vertices of color 1 (resp.~2), then we have one orbit for each white vertex (so $\#\alpha_1 + \#\alpha_2$), and one for each face of the map (so $\#(\gamma\alpha_1\alpha_2)$).
We wish to compare the expression \eqref{eq:moments-FC2-permutations}  with the exact moments \eqref{eq:ExactExpressionMomentsUsing-Delta-g} of the model studied in Section~\ref{sec:ExactMomentsABCD}, which we report here  for $N_A=N_B=N_C=N_D$, and for $P=W_{AB}^{1/2}W_{AC}W_{AB}^{1/2}$: denoting $f_0$ the length $2p$ word $AB, AC, AB, AC, \ldots$ (there are no two consecutive edges of the same color),
\begin{equation}
\label{eq:moments-ABCD-Compare}
\mathbb E \operatorname{Tr} P^p = N^{2(2p+1)}\sum_{\cM \in \bM_{2p}} N^{-2g(\cM)-2g(\cM_{f_0}) -2\Delta_{f_0}(\cM)}.
\end{equation} 
Using that the Euler characteristics of the map writes $F(\cM) + V(\cM) = 2p+2 - 2g(\cM)$, we notice that  \eqref{eq:moments-FC2-permutations} is very similar to \eqref{eq:moments-ABCD-Compare}, provided that we impose the condition 
\begin{equation}
\Delta_{f_0}(\cM) = 0.
\end{equation} 
In that case, $\cM=\cM_{f_0}$, so that $g(\cM_{f_0})=g(\cM)$.
We rewrite
\begin{equation}
\label{eq:moments-FC2-permutations-BIS}
\mathbb E \operatorname{Tr} Q^p = N^{2p+1}\sum_{ \substack{{\cM \in \bM_{2p}}\\{\Delta_{f_0}(\cM) = 0}}} N^{-2g(\cM)}.
\end{equation} 

As a consequence, the asymptotic moments of the two matrix models are identical, but the lower orders are different:
\begin{align*}
\mathbb E N^{-2}\operatorname{Tr}(N^{-4} P) &= 1 + N^{-2} \\
\mathbb E N^{-2}\operatorname{Tr}(N^{-4} P)^2 &= 3+ 8N^{-2} + 8N^{-4} + 5N^{-6}\\
\mathbb E N^{-2}\operatorname{Tr}(N^{-4} P)^3 &= 12+ 54N^{-2} + 135N^{-4} + 278N^{-6}+170 N^{-8} + 71 N^{-10},\\[+1ex]
\mathbb E N^{-1}\operatorname{Tr}(N^{-2}Q) &= 1 \\
\mathbb E N^{-1}\operatorname{Tr}(N^{-2}Q)^2 &= 3+N^{-2} \\
\mathbb E N^{-1}\operatorname{Tr}(N^{-2}Q)^3 &= 12 + 21N^{-2} + 3 N^{-4}.
\end{align*}

\begin{figure}[!ht]
	\centering
	\includegraphics[width=0.9\textwidth]{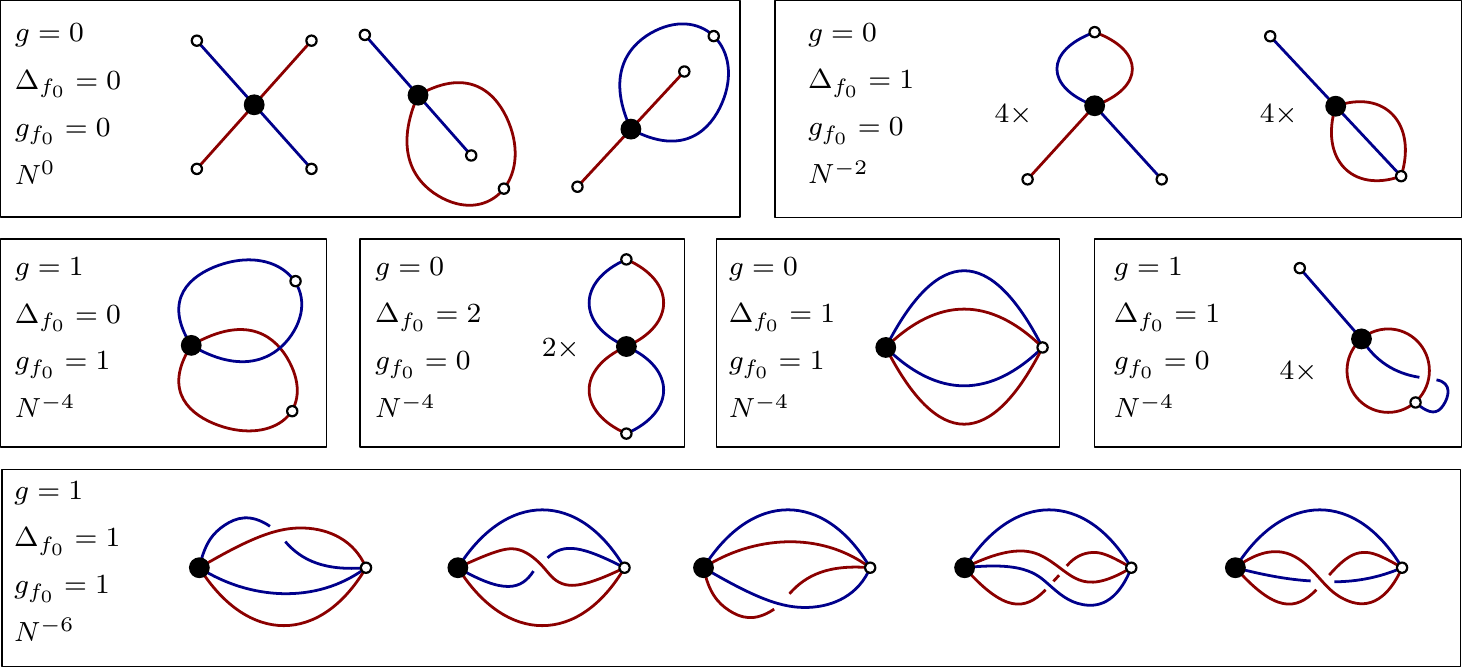}
	\caption{Maps contributing to $\mathbb E N^{-2}\Tr(N^{-4} P)^2$. We have denoted $g_{f_0}=g(\cM_{f_0})$. The edges labeled $1$ is always the upper left one (blue). The only maps contributing to $\mathbb E N^{-1}\Tr(N^{-2}Q)^2 $ are those with $\Delta_{f_0}=0$, however in that case, the contribution in $N$ is only given by the genus $g$ of the map.}
	\label{fig:DiagramsExample}
\end{figure}

The coefficients of the lower orders in $N$ of $P$ are bigger than the ones of $Q$. This originates from the fact that in the $Q$ case, the lower orders come from maps in $\bM_{2p}$ with $\Delta_{f_0}(\cM) =0$ with non-trivial genus, while in the case of $P$, lower orders are obtained either from higher genus combinatorial maps\footnote{It is also important to notice that in the $P$ case the genus appears in two ways in the exponent of $N$ - once for the combinatorial map $\cM$ and once for the combinatorial map $\cM_{f_0}$. This has to be taken into account when comparing the $Q$ case against the $P$ case. } or combinatorial maps containing both white vertices adjacent to both type $AB$ and type $AC$ edges (i.e.~with non-vanishing $\Delta_{f_0}$).
In Figure~\ref{fig:DiagramsExample}, we show all the maps contributing to $\mathbb E N^{-2}\operatorname{Tr}(N^{-4} P)^2$. The values of $g$, $\Delta_{f_0}$ and $g_{f_0} = g(\cM_{f_0})$ are shown. 
For instance  the maps in the upper right box of Fig.~\ref{fig:DiagramsExample} are planar but contain one white vertex adjacent to one edge of type $AB$ and one edge of type $AC$, and  therefore contribute to $\mathbb E N^{-2}\operatorname{Tr}(N^{-4} P)^2$ at order $N^{-2}$.  
The maps contributing to $\mathbb E N^{-1}\operatorname{Tr}(N^{-2}Q)^2$ are those for which $\Delta_{f_0} = 0$. Their contribution to $\mathbb E N^{-1}\operatorname{Tr}(N^{-2}Q)^2$ is $N^{-2g}$, but their contribution to $\mathbb E N^{-2}\operatorname{Tr}(N^{-4} P)^2$ is $N^{-4g}$, because $g(\cM_{f_0})=g(\cM)$. 

\

Keeping in mind the aim of comparing the two matrix models mixed moments with the moments of $P$ we start by defining the following operation on combinatorial maps
 \begin{definition}
 Let $\cM_1,\cM_2 \in \bM(p) \times \bM(p')$ be two combinatorial maps with respectively $p$ and $p'$ edges. Both maps have a labeling of the edges. We define the gluing convolution as 
\begin{align}
 \odot : \ &\bM(p)\times \bM(p')\rightarrow \bM(p+p')\\
 &(\cM_1, \cM_2)\longmapsto \cM 
\end{align} 
where $\cM$ is the empty map if $p\neq p'$ and is otherwise obtained from $\cM_1$ and $\cM_2$ by stacking their black vertices one onto the other, in such  way that the edges of $\cM_2$ slip into the corners of the black vertex of $\cM_1$ and edges of $\cM_1$ and $\cM_2$ alternate around the black vertex of the newly created map $\cM$. In order to select a unique way to perform this operation, we ask that the edge with label $i$ in $\cM_1$ is followed by the edge labeled $i$ in $\cM_2$ when following the ordering of the edges around the black vertex of $\cM$. Graphically one obtains the local construction shown on Fig.~\ref{fig:gluing-convolution}.
\begin{figure}[!ht]
 \begin{center}
{\begin{tabular}{@{}c@{}}  \includegraphics[scale=0.6]{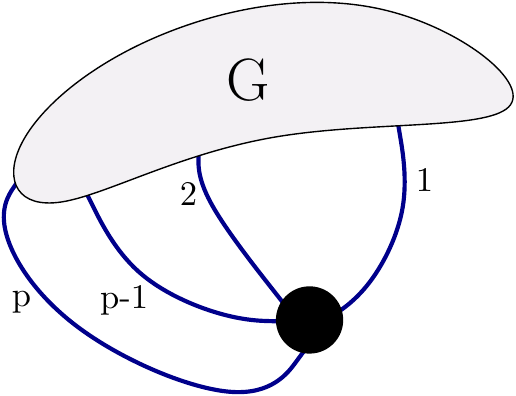} \\
$\cM_1$\end{tabular} }
\raisebox{+0.5ex}{$\odot$}\quad
{\begin{tabular}{@{}c@{}}  \includegraphics[scale=0.6]{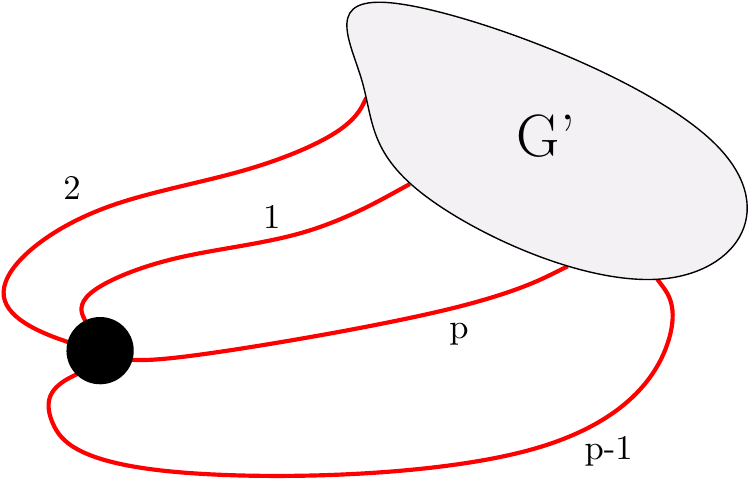} \\
$\cM_2$\end{tabular} }
\raisebox{+0.5ex}{$=$}\quad
{\begin{tabular}{@{}c@{}}  \includegraphics[scale=0.6]{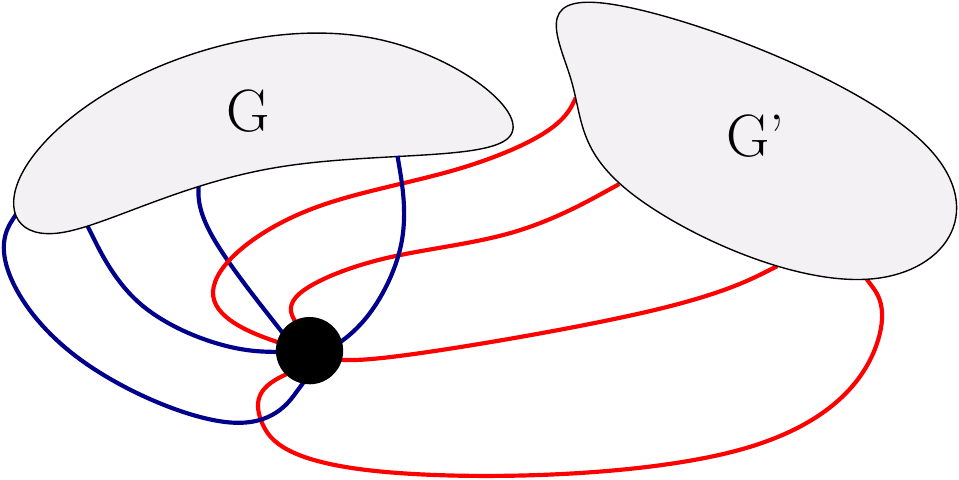} \\
$\cM_1\odot \cM_2$\end{tabular} }
  \caption{Gluing convolution of two maps. The gluing convolution is made locally around the black vertex. The rest of the maps stays untouched.}\label{fig:gluing-convolution}
 \end{center}
\end{figure}
\end{definition}
We also define the splitting of a vertex.
\begin{definition}
 A vertex-splitting is a local move on a vertex with at least two corners of a combinatorial map. It is performed by choosing two corners of the considered vertex and splitting the vertex along a straight line between these two corners.
\end{definition}
\begin{figure}[!ht]
 \includegraphics[scale=0.6]{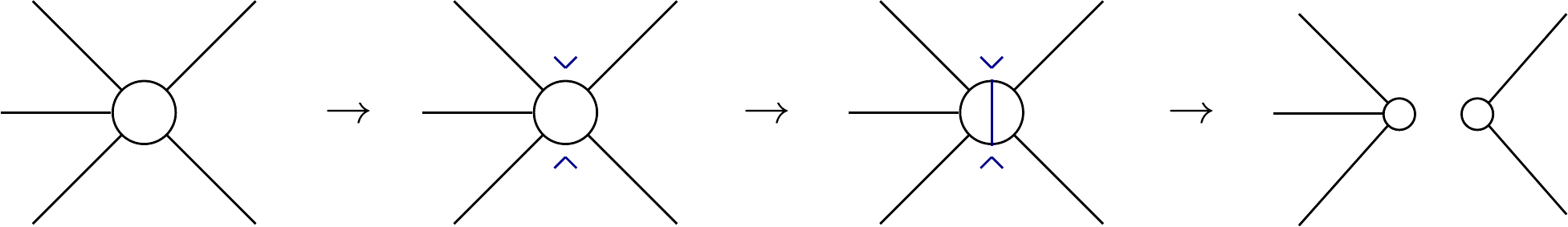}
 \caption{A vertex-splitting move applied to a white vertex.}
 \label{fig:VertexSplitting}
\end{figure}

Notice that the vertex splitting move can be understood in the language of permutations. In this language, a $p$-valent white vertex is a cycle of length $p$, $(\alpha_c)$ of $\alpha=\prod_{c\in \alpha} (\alpha_c)$. Such a cycle writes $(\alpha_c)=(a_1 a_2 \ldots a_p)$. Choosing two corners of a white vertex amounts to picking two elements $a_k, a_{k'}$ in $(\alpha_c)$ such that $a_{k'}\neq a_k$. This corresponds to the choice of the two corners of the white vertex located between edges $a_k,a_{k+1}$ and $a_{k'}, a_{k'+1}$. These elements appear in the cycle $(\alpha_c)=(a_1 a_2 \ldots a_k a_{k+1}\ldots a_{k'} a_{k'+1} a_p)$. The splitting of a white vertex is just the composition of $(\alpha_c)$ with the transposition $(a_k a_{k'})$, as $(\alpha_c)(a_k a_{k'})=(\alpha_{c'})(\alpha_{c''})$ where $(\alpha_{c'})=(a_1 a_2\ldots a_{k}a_{k'+1}\ldots a_p)$ and $(\alpha_{c''})=(a_{k+1}\ldots a_{k'})$ forms the two new white vertices.\\

We consider the partial order on $\bM^0_p$ defined as follows, for $\cM,\cM' \in \bM^0_p$ two planar combinatorial maps, we say that $\cM \le \cM'$ if and only if there exists a finite sequence of maps $\{\cM_i\}_{i=0}^F$ such that $\cM_0=\cM'$, $\cM_F=\cM$, and $\cM_{i+1}$ can be obtained from $\cM_i$ by applying a vertex-splitting move on one white vertex of $\cM_i$.\\

We now define the Tutte dual of a map. It is a particular case of one of the Tutte bijections for bicolored maps \cite{TutteTrinity}, as was the bijection presented in Section~\ref{sec:FirstSec}, Fig.~\ref{fig:Bij1}. Though it can be defined for more general sets of maps we assume $\cM \in \bM_p$,
\begin{definition}
\label{def:TutteDual}
The Tutte dual $\cT(\cM)$ of $\cM \in \bM_p$ is obtained as follows. In every face of $\cM$ we apply the following rules 
\begin{itemize}
 \item We draw a white vertex inside the face under consideration.
 \item We draw a new edge between each corner adjacent to the black vertex inside this face and the newly created white vertex inside this face.
\end{itemize}
Once we have followed this procedure for every face of $\cM$, we erase all initial edges and white vertices.
\end{definition}
We notice that the Tutte dual preserves the genus of the map, that is to say that $g(\cM)=g(\cT(\cM))$. Remark that the only difference with the bijection presented in Section~\ref{sec:FirstSec}, Fig.~\ref{fig:Bij1}, is that the new edges are added between the new vertices and the black vertex, instead of the white vertices. A consequence of this choice is that after the bijection of Fig.~\ref{fig:Bij1}, one no longer has elements of $\bM_p$, while we do in the present case, thus the name ``dual". We report the reader to \cite{TutteTrinity} for the general bijection encoding both cases. We then have the following proposition
\begin{proposition}\label{prop:genus-gluing-convolution}
Let $\cM_1, \cM_2$ be two planar maps in $\bM^0_p$ (\textit{i.e.} $g(\cM_1)=g(\cM_2)=0$). Then $g(\cM_1\odot \cM_2)=0$ if and only if $\cM_2\le \cT(\cM_1)$.
\end{proposition}
\begin{proof}
First notice that by construction of $\cT(\cM_1)$, $g(\cM_1\odot\cM_2)=0$ if $\cM_2=\cT(M_1)$.\\
Then when performing a vertex-splitting move on a white vertex of $\cM_2$ in $\cM_1\odot\cM_2$ the number of vertices is raised by one $V\rightarrow V'=V+1$, the number of edges stays the same, and the number of faces is decreased by one $F\rightarrow F'=F-1$ as a consequence the genus stays constant under such a move. Thus we have $\cM_2\le \cT(\cM_1)\Rightarrow g(\cM_1 \odot \cM_2)=0$.  \\
Assume now that $\cM_2\not \le \cT(M_1)$. Therefore there exists a vertex in $\cM_2$ adjacent to two edges of $\cM_2$ in $\cM_1\odot \cM_2$ whose starting points are at corners that belong to two different faces of $\cM_1$. Thus edges of $\cM_1$ and $\cM_2$ cross in $\cM_1 \odot \cM_2$, which is equivalent to saying that $g(\cM_1 \odot \cM_2)\neq 0$. Thus we have $\cM_2 \le \cT(\cM_1) \Leftarrow g(\cM_1\odot \cM_2)$. 
\end{proof}

If we come back to the definition of the Tutte dual we notice that it is is equivalent to the Kreweras complementation. This similarity is pictured on the Fig.~\ref{fig:TuttevsKreweras}, where the graphical representation of non-crossing partition has been borrowed from \cite{nica2006lectures}. Note however that the Tutte dual is defined for every maps not just the planar ones (the latter corresponding to non-crossing partitions).
\begin{figure}
 \begin{center}
  \includegraphics[scale=0.7]{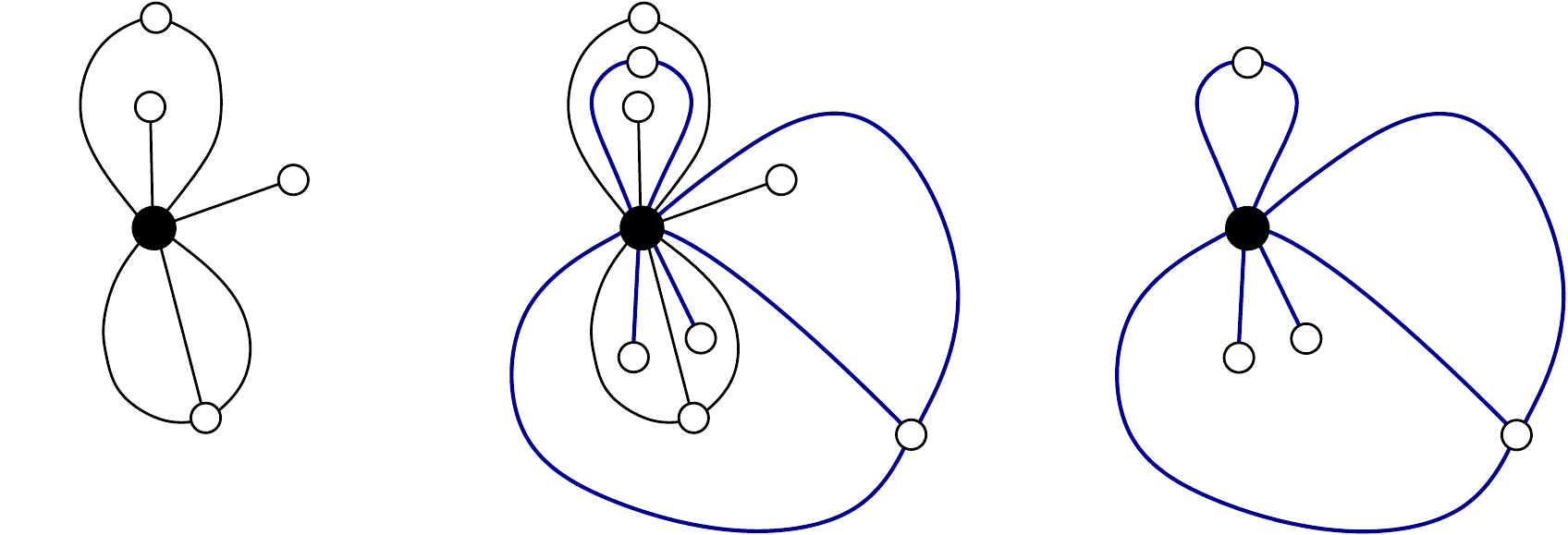}\\
  \vspace{5mm}
  \includegraphics[scale=1.0]{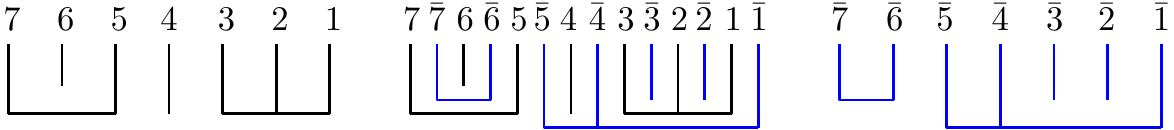}
  \vspace{3mm}
  \caption{
  {Above:} Tutte duality. The initial map $\cM$ on the left, its Tutte dual $\cT(\cM)$ on the right. The middle map represents the intermediate steps of the construction. The labeling of the edges of $\cT(\cM)$ is inherited from the map $\cM$. 
  {Below:} Kreweras complementation map. }\label{fig:TuttevsKreweras}
 \end{center}
\end{figure} 
We can write
\begin{align}
\label{eq:moments-P-maps}\mathbb E \operatorname{Tr} P^p &=N^{4p+2}\sum_{\mathcal M \in \mathbb M_{2p}} N^{-2g(\mathcal M)-2(g(\cM_{f_0}) + \Delta_{f_0}(\cM))}\\
\label{eq:moments-FC2-glue}\mathbb E \operatorname{Tr} Q^p &= N^{2p+1}\sum_{\mathcal M_1, \mathcal M_2 \in \mathbb M(p)} N^{-2g(\mathcal M_1 \odot \mathcal M_2)} \\
\label{eq:moments-FC2-alt}&= N^{2p+1}\sum_{\mathcal M \in \mathbb M_{2p} \, : \, \Delta_{f_0}(\cM) = 0} N^{-2g(\mathcal M)},
\end{align}
The proof of the equations \eqref{eq:moments-FC2-glue}, \eqref{eq:moments-FC2-alt} is straightforward. From \eqref{eq:moments-FC2-glue} we see the connection with the formula \eqref{eq:moments-FC2-permutations}, while from \eqref{eq:moments-FC2-alt} we see why the moments of $P$ are larger than those of $Q$: the terms in the sum for the moments of $Q$ \eqref{eq:moments-FC2-alt} are a subset of the terms in \eqref{eq:moments-P-maps}. Indeed, we only have $\left.\bM_{2p}\right|_{\Delta_{f_0}=0} =\bM_p\odot\bM_p$, and this equality does not hold in the planar case, $\left.\bM^0_{2p}\right|_{\Delta_{f_0}=0} \subsetneq \bM^0_p\odot\bM^0_p$
\footnote{The extension of the $\odot$ operation to sets of maps is understood in the straightforward way as the set of all maps with $2p$ edges that can be obtained by making the gluing convolution of two maps with $p$ edges.}.
In the limit $N\rightarrow \infty$, thanks to Proposition~\ref{prop:genus-gluing-convolution}, the equation \eqref{eq:moments-FC2-glue} rewrites
\begin{align}\label{eq:moments-FC2-FPstyle}
\lim_{N\rightarrow \infty}\frac1{N^{2p+1}}\mathbb E \operatorname{Tr} Q^p &= \sum_{\mathcal M_1\in \mathbb M_p} \Biggl( \sum_{\cM_2 \le \cT(\cM_1)} 1\Biggr),
\end{align}
which mimics, in the language of maps, the expression obtained in \cite[Lecture 14, Theorem 14.4]{nica2006lectures} with cumulants equal to one in the language of permutations. This translates the fact that in the large $N$ limit the moments of $Q$ are the moments of the multiplicative convolution of two Mar\v cenko-Pastur laws.

\smallskip

\

\subsection{The unbalanced asymptotical regime}\label{sec:2-marginals-fixed-BC}

\

\medskip

In this section, we study the asymptotical regime where the Hilbert spaces $\mathcal H_B$ and $\mathcal H_C$ have fixed dimension, while the dimensions of $\mathcal H_A$ and $\mathcal H_D$ grow to infinity. To be more precise, we assume in this section that
\begin{itemize}
\item $\dim \mathcal H_B = \dim \mathcal H_C = m$, for some positive integer constant $m$;
\item $\dim \mathcal H_A = N$, $\dim \mathcal H_D = \lfloor cN \rfloor$, for some constant $c \in (0,\infty)$, where $N \to \infty$.
\end{itemize}

Contrary to the results proven in Section~\ref{subsec:Large-HBC}, in this setting, the random matrices $W_{AB}$ and $W_{AC}$ are no longer asymptotically free. One can understand this fact, stated precisely in the theorem below, by noticing that the ``shared randomness'' between the two random matrices ($\dim \mathcal H_A = N \to \infty$) is much larger than the ``fresh randomness'' ($\dim \mathcal H_B = \dim \mathcal H_C = m$, fixed).

\begin{theorem}
\label{thm:2-marginals-fixed-BC}
In the asymptotical regime described above, the pairs of random matrices
$$\left((mN)^{-1}W_{AB}, (mN)^{-1}W_{AC}\right)$$
converge in distribution, 
as $N \to \infty$, to a pair of non-commutative random variables $(x_{AB},x_{AC})$ having the following free cumulants:
\begin{equation}\label{eq:4-partite-unbalanced-free-cumulants}
\kappa(x_{f(1)}, x_{f(2)}, \ldots, x_{f(p)}) = cm^{-\alt(f)},
\end{equation}
where $f \in \{AB, AC\}^p$ is an arbitrary word in the letters $AB, AC$, and $\alt(f)$ is the number of different consecutive values of $f$, counted cyclically:
$$\alt(f) := |\{a \, : \, f(a) \neq f(a+1)\}|,$$
where $f(p+1):=f(1)$. Equivalently, for any word in the two marginals $f \in \{AB, AC\}^p$, we have
\be 
\mathbb E \operatorname{Tr} W_f = (1+o(1)) (mN)^{p+1} \sum_{\alpha \in \mathrm{NC}(p)} c^{\#\alpha} m^{-\alt(f,\alpha)},
\ee 
where $\alt(f,\alpha)$ has been defined in  \eqref{eq: Definition-Alt} as  
$$ \alt(f,\alpha) = \bigl \lvert \{a \in \llbracket 1, p \rrbracket \, : \, f(a) \neq f(\alpha(a))\} \bigr \rvert. $$
\end{theorem}
\begin{proof}
We begin by analyzing the exact moment formula of Theorem~\ref{thm:moments-fixed-N}:
$$\mathbb E \operatorname{Tr} W_f = \sum_{\alpha \in \mathcal S_p} N^{\#(\gamma\alpha)} \lfloor cN \rfloor^{\#\alpha} m^{L(f,\alpha)}.$$
Note that since $m$ is fixed, the dominating terms correspond to permutations $\alpha$ maximizing the exponent $\#\alpha + \#(\alpha\gamma)$; these permutations have been shown before to be exactly the non-crossing ones, so we have 
$$\mathbb E \operatorname{Tr} W_f = (1+o(1)) N^{p+1} \sum_{\alpha \in \mathrm{NC}(p)} c^{\#\alpha} m^{L(f,\alpha)}.$$
For $\alpha$ non-crossing, we know from Proposition~\ref{prop:L-and-alt} that $L(f,\alpha) = p+1 - \alt(f,\alpha)$, and the conclusion concerning the cumulants follows from Speicher's moment-cumulant formula.
\end{proof}

\begin{remark}
In the degenerate case $m=1$ (i.e.~there are no $B$ and $C$ systems), one has that $W_{AB} = W_{AC}=W_A$ is a Wishart matrix of parameters $(N, \lfloor cN \rfloor)$, and one recovers the result from the classical Mar{\v{c}}enko-Pastur setting, Proposition~\ref{prop:MP}: the free cumulants of $x_{AB} = x_{AC}$ are all equal to $c$ (see equation \eqref{eq:MP-moments-free-cumulants} and the comments following it). 
\end{remark}

\begin{remark}
If, after taking the limit $N \to \infty$ in the theorem above, one takes the limit $m \to \infty$, we recover the result of Theorem~\ref{thm:2-marginals-large-BC}. Indeed, the only free cumulants surviving are the ones with $\alt(f)=0$ (i.e.~mixed cumulants vanish), proving that $x_{AB}$ and $x_{AC}$ are asymptotically free. 
\end{remark}

\begin{remark}
Elaborating on the preceding remarks, we notice that the eigenvalues distribution of the matrix $P=W_{AB}^{1/2}W_{AC}W_{AB}^{1/2}$ for $m=1$ is the distribution of the square of the eigenvalues of a Wishart matrix, the corresponding density, that we denote $\mathrm{d}\rho_{m=1,c}(x)$ writes $\mathrm{d}\rho_{m=1,c}(x)=\mathrm{d}\textrm{MP}_c(\sqrt{x})$, while if $m\rightarrow \infty$ we end up with $\mathrm{d}\rho_{m=\infty,c}(x)= \mathrm{d}\textrm{MP}_c^{\boxtimes 2}(x)$. This fact does not depend on the rate at which $m$ is sent to infinity with respect to the $N$ limit. Thus the parameter $m$ allows one to canonically interpolate between the two different distributions $\mathrm{d}\textrm{MP}_c(\sqrt{x})$ and $\mathrm{d}\textrm{MP}_c^{\boxtimes 2}(x)$, the former being the distribution of the square of eigenvalues of a Wishart random matrix, while the latter is the free multiplicative convolution of two Mar\v cenko-Pastur distribution. 
Notice also, that one can also consider the (non-selfadjoint) matrix $\tilde{P}=W_{AB}W_{AC}$ which has the same moments and  eigenvalues to reach the same conclusion. Finally, we plot Monte-Carlo simulations of the eigenvalues of $P$ versus the square (resp.~the free multiplicative square) of a Mar\v cenko-Pastur distribution in Figure~\ref{fig:P-vs-MPsquare}.
\end{remark}
\begin{figure}
 \begin{center}
  \includegraphics[scale=0.4]{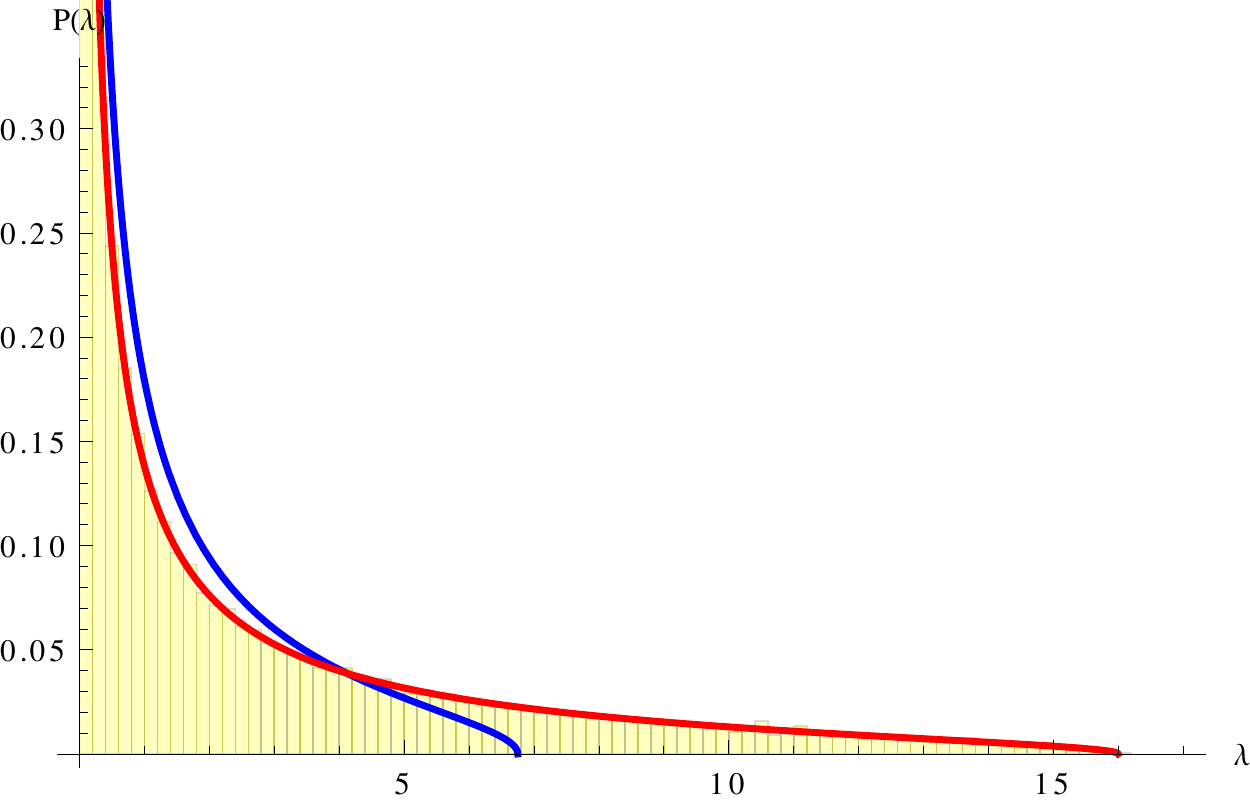}\quad 
    \includegraphics[scale=0.4]{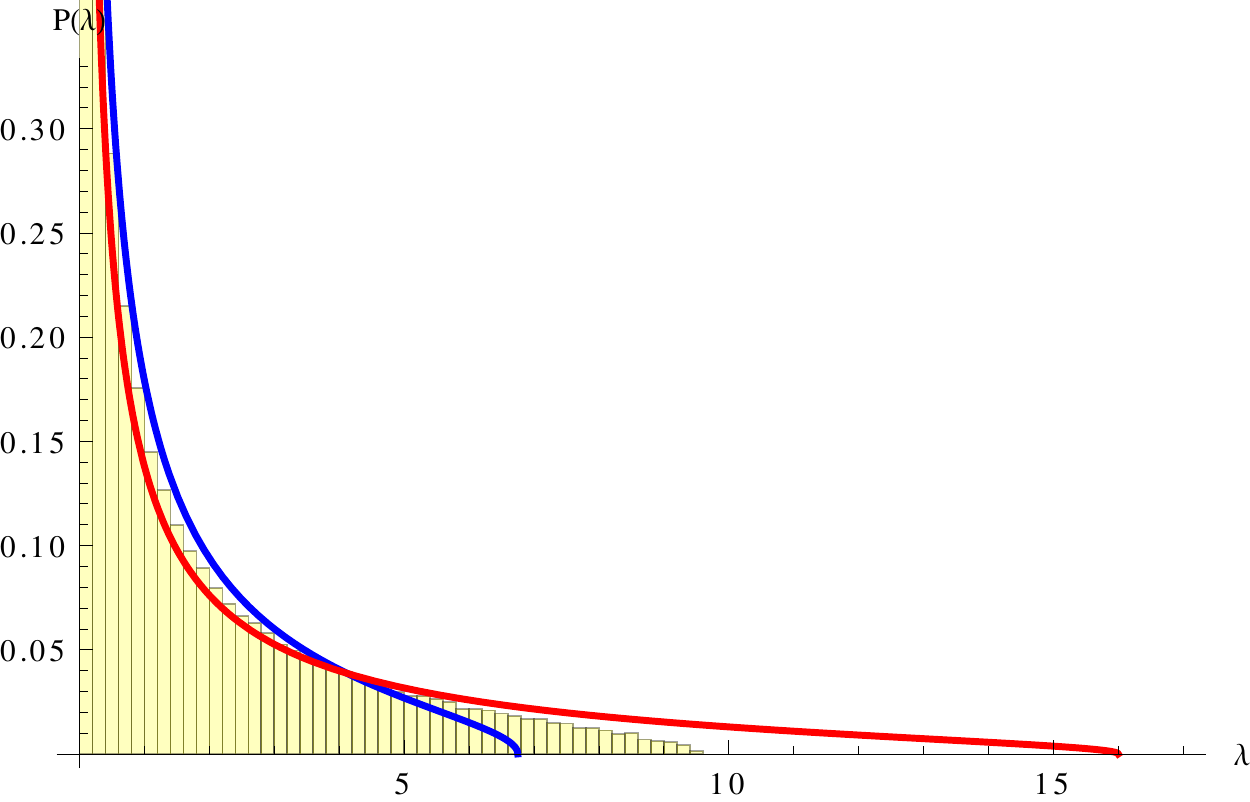}\quad 
      \includegraphics[scale=0.4]{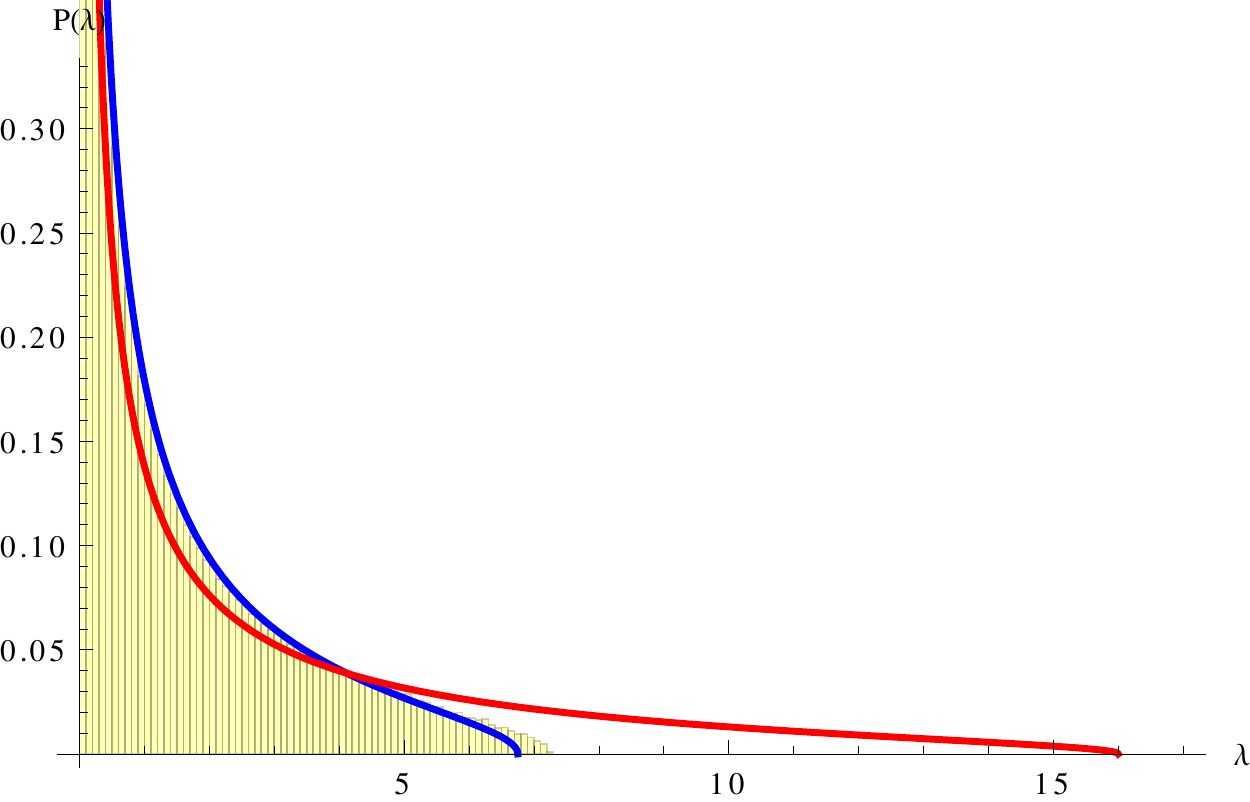}\\
   \includegraphics[scale=0.4]{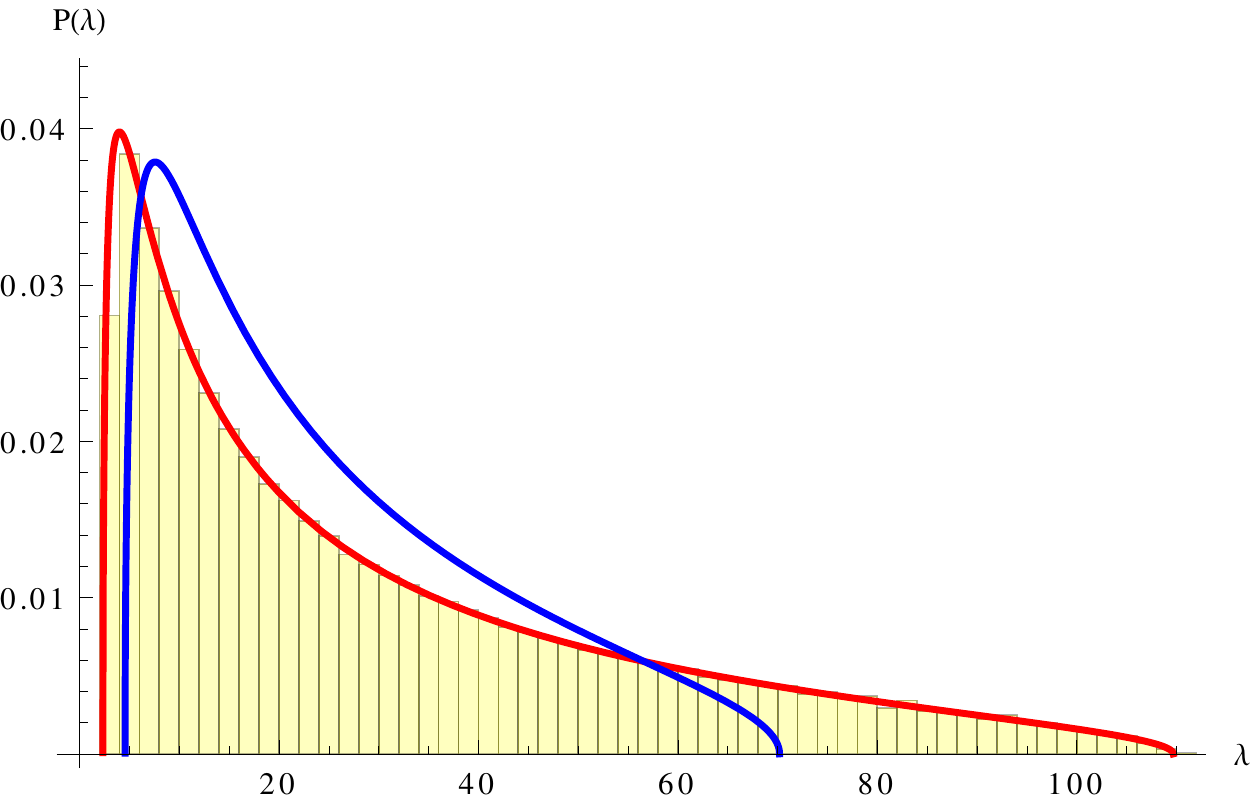}\quad 
    \includegraphics[scale=0.4]{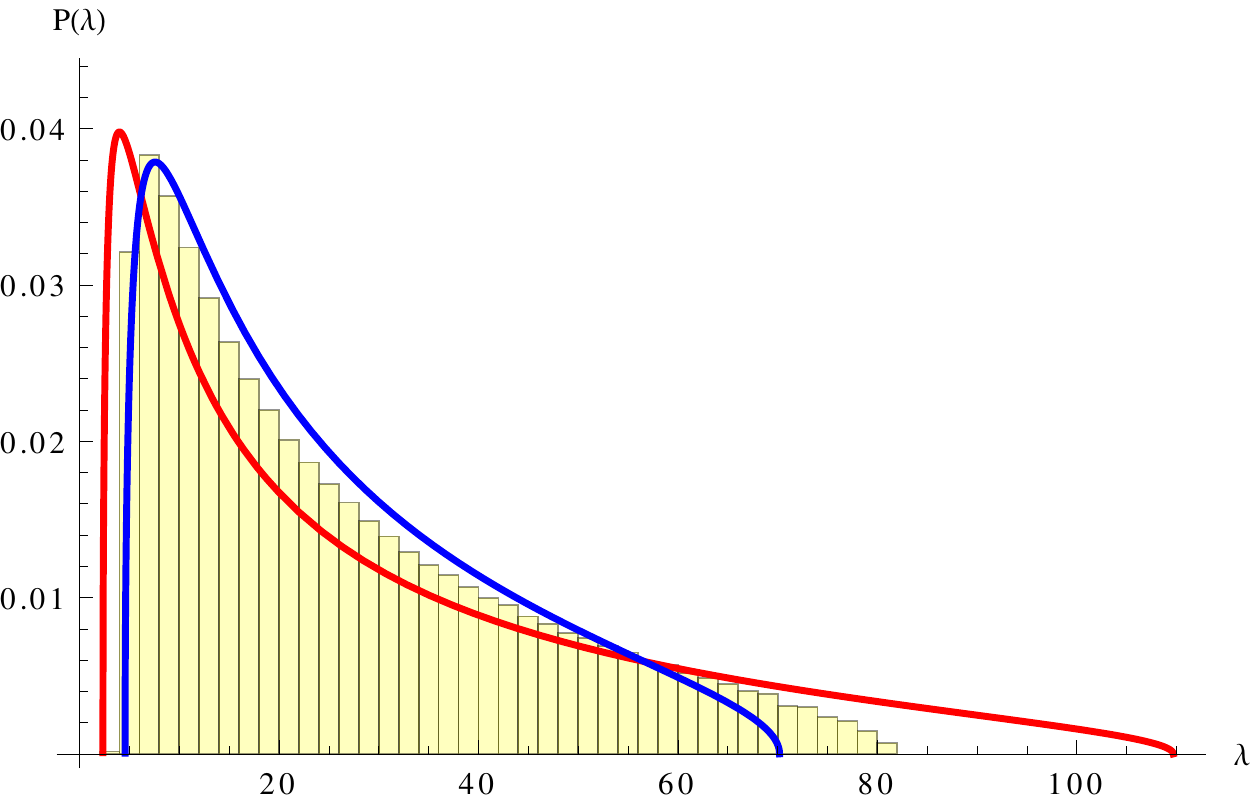}\quad 
      \includegraphics[scale=0.4]{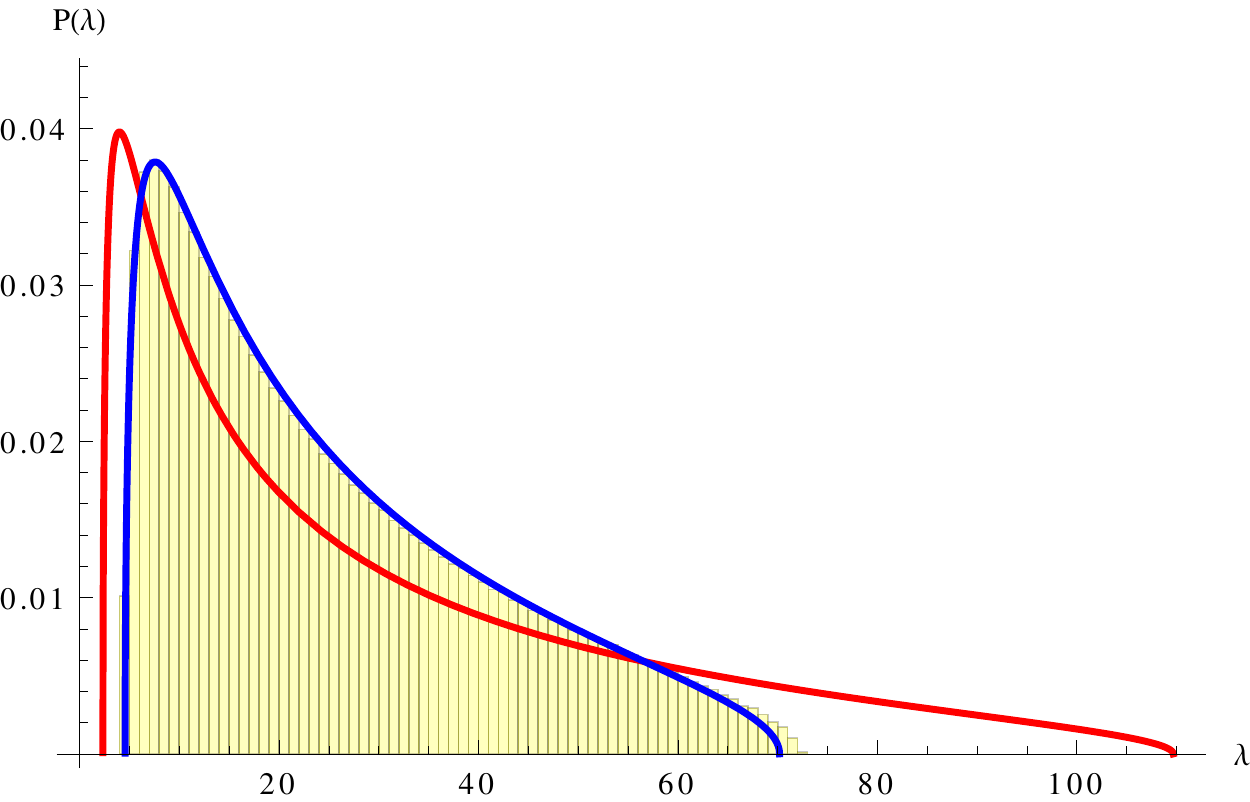}   
  \vspace{3mm}
  \caption{Plots of Monte-Carlo simulations of the eigenvalues of $P=W_{AB}^{1/2}W_{AC}W_{AB}^{1/2}$ (yellow histogram) versus the square of a Mar\v cenko-Pastur distribution (red curve) and the free multiplicative square of the same Mar\v cenko-Pastur distribution (blue curve). On the top row, we have $N=600$, $c=1$ and $m=1,2,5$, while on the bottom row we have $N=600$, $c=5$, $m=1,2,5$.}
  \label{fig:P-vs-MPsquare}
 \end{center}
\end{figure} 
We record below some low mixed moments in the variables $x_{AB}$, $x_{AC}$ (we denote by $\operatorname{tr}$ the expectation in the non-commutative probability space where these variables live):
\begin{align*}
\operatorname{tr}(x_{AB}x_{AC}) &= c^2+\frac{c}{m^2}\\
\operatorname{tr}(x_{AB}x_{AB}x_{AC}) &= c^3+c^2+\frac{2 c^2}{m^2}+\frac{c}{m^2}\\
\operatorname{tr}(x_{AB}x_{AB}x_{AC}x_{AC}) &= c^4+2 c^3+c^2+\frac{4 c^3}{m^2}+\frac{4 c^2}{m^2}+\frac{c}{m^2}+\frac{c^2}{m^4}\\
\operatorname{tr}(x_{AB}x_{AC}x_{AB}x_{AC}) &= c^4+2 c^3+\frac{4 c^3}{m^2}+\frac{4 c^2}{m^2}+\frac{2 c^2}{m^4}+\frac{c}{m^4}.
\end{align*}

\medskip

Let us now explore the consequences of the formulas for the mixed free cumulants from this section to quantum information theory. As in Section \ref{sec:2-marginals-large-BC}, the two quantum marginals $\rho_{AB}$ and $\rho_{AC}$, when properly rescaled, converge in moments, jointly, to a pair of non-commutative random variables having free cumulants as in \eqref{eq:4-partite-unbalanced-free-cumulants}: $N_DN(\rho_{AB},\rho_{AC}) \to (x_{AB},x_{AC})$. This allows one to compute the asymptotic value of any correlation function involving $\rho_{AB}$ and $\rho_{AC}$. For example, the rescaled overlap between the matrices converges to 
	$$\lim_{N \to \infty} N_A N \langle \rho_{AB} , \rho_{AC} \rangle = \frac{1}{c^2} \kappa(x_{AB}, x_{AC}) = 1+ \frac{1}{cm^2}.$$
The computation above should be compared with the similar limit from the balanced regime of Section \ref{sec:2-marginals-large-BC}, where the two marginals were uncorrelated: $\lim_{N \to \infty} N_A N \langle \rho_{AB} , \rho_{AC} \rangle =0$.

\smallskip

\section{The general multipartite case}
\label{sec:general-multipartite}

In this section we consider the more general situation of a random Wishart tensor defined on a Hilbert space which is factorized in an arbitrary number of factors.  The section consists of three parts: we first derive the general, non-asymptotic mixed moment formula, and then consider two asymptotic regimes: the balanced regime, where all tensor factors have the same dimension, and the unbalanced regime, where some of the tensor factors (the ones corresponding to the ``moving legs'') are being kept fixed. In the balanced case, we prove that the marginals are asymptotically free (Proposition~\ref{prop:LargeNinTheLarge1nRegime} in the Wishart setting and Theorem~\ref{thm:asymptotic-freeness-balanced-general} in quantum information language), while in the unbalanced case, we show in an example that it is not possible to factorize the expression of the mixed cumulant functions over the cycles of the non-crossing partitions. Indeed they depend more finely on the structure of the non-crossing partitions. This implies that we cannot give an expression for the mixed free cumulants. However, we expect that this situation can be dealt with in the framework of free probability with amalgamation. Such results will be presented in a following paper.

We consider complex tensor $X$ of size $N_1\times N_2\times \cdots \times N_n$, an un-normalized quantum state in the said Hilbert space $\mathbb C^{N_1} \otimes \cdots \otimes \mathbb C^{N_2}$. The density matrix of the corresponding pure state is $X\otimes X^*$, the (un-normalized) unit rank projection on the space $\mathbb C X$.
For a given set $I\subset\{1,\ldots, n\}$, we denote $\widehat I = \{1,\ldots,n\}\setminus I $, and define the \emph{reduced density matrix} as the tensor $X._{\widehat I}\bar X$ obtained by summing, for each $i\in \widehat I$, the index of position $i$ of $X$ with the index of position $i$ of $\bar X$, $X._{\widehat I}\bar X =[\mathrm{id}_I \otimes \mathrm{Tr}_{\widehat I}](XX^*)$. This partial contraction of two tensors can also be understood as the matrix $[X._{\widehat I}\bar X]$, whose first (resp.~second) sub-index of position $j\in I$ is the free-index of position $j$ of $X$ (resp.~$\bar X$). For instance, for $n=4$, choosing $\widehat I=\{3,4\}$,
\be
[X\cdot_{ 3,4}\bar X]_{i_1,i_2\, ;\, i'_1,i'_2} = \sum_{i_3=1}^{N_3}\sum_{i_4=1}^{N_4}
X_{i_1,i_2,i_3,i_4} \bar X_{i'_1,i'_2, i_3,i_4}.
\ee

There is a canonical one-to-one correspondence which maps $I$ to $\{1,\ldots,\lvert I \rvert\}$ while preserving the ordering of natural integers. We denote $\cS_{\lvert I \rvert}$ the set of permutations of $\lvert I \rvert$ elements. In the following, we implicitly make use of these canonical bijections when saying that a permutation $\sigma\in\cS_{\lvert I \rvert}$ acts on $I$ and has $I'$ of same cardinality as an image. For instance if $I=\{A,B,C\}$ and $I'=\{A,C,E\}$, the identity $\operatorname{id}:I\rightarrow I'$ is understood as the map $A\rightarrow A, B\rightarrow C, C\rightarrow E $.

Given $\cI\in\bN$, a permutation $\sigma\in\cS_\cI$, and two matrices $P$ and $Q$, whose two indices have $\cI$ sub-indices, we define the  product $P\cdot_\sigma Q$ as the twisted contraction
\be
\bigl(P\cdot_\sigma Q \bigr)_{i_1,\cdots,i_\cI\, ;\, i'_1,\cdots, i'_\cI} = \sum_{\substack{
{j_1, \cdots, j_\cI}\\ 
{j'_1, \cdots, j'_\cI}}} \prod_{b=1}^\cI \delta_{j'_b}^{j_{\sigma(b)}} P_{i_1,\cdots,i_\cI\, ;\, j_1,\cdots j_\cI}Q_{j'_1,\cdots, j'_{\cI}\, ;\, i'_1,\cdots,i'_\cI  }.
\ee
We define the associated trace $\Tr\cdot_{\sigma}$ accordingly. 
We are interested in computing expectations of the form 
\be
\bE \Tr\cdot_{\sigma_p} [X._{\widehat I_p}\bar X]\cdot _{\sigma_{p-1}}\,\ldots\,  \cdot _{\sigma_3} [X._{\widehat I_3}\bar X] \cdot _{\sigma_2} [X._{\widehat I_2}\bar X] \cdot_{\sigma_1}[X._{\widehat I_1}\bar X], 
\ee
for some integer $p$, some non-necessarily distinct sets $I_a$ which all have the same number of elements $\lvert I_a\rvert = \cI$, and some permutations $\sigma_a\in\cS_\cI$, (with our convention, $\sigma_a:I_a \rightarrow I_{a+1}$, and $I_{p+1}=I_1$). We denote $W_{I} = [X._{\widehat I}\bar X]$, so that the objects under focus are rewritten as 
\be
\label{eq:ExpectationGenSigma}
\bE \Tr_{\fsig} W_{\bff} (\{N_i\})  =
\bE \Tr\cdot_{\sigma_p} W_{I_p}\cdot _{\sigma_{p-1}}\,\ldots\,   \cdot _{\sigma_3} W_{I_3} \cdot _{\sigma_2} W_{I_2} \cdot_{\sigma_1}W_{I_1},
\ee
where we respectively denoted $\bff$ and $\fsig$ the ordered lists $\bff=[I_1,\ldots, I_p]$ of  $p$  subsets of $\llbracket 1,n \rrbracket$, and $\fsig=[\sigma_1,\ldots, \sigma_p]$ of  $p$ permutations in $\cS_\cI$.
Note that depending on the $\{I_a\}$ and the $\{\sigma_a\}$, all choices are not possible for $\{N_1, \ldots, N_n\}$. In general the above defined object is always meaningful when $N_1=\cdots=N_n$, but  interesting cases can be considered for specific $\{I_a\}$ and $\{\sigma_a\}$.

To highlight this, we separate for a given $\bff$, the colors which are traced for every $W_{I_a} = [X._{\widehat I_a}\bar X]$, which, without loss of generality, we can suppose to be the colors from $r$ to $n$,
\be
\{r,\ldots, n\} = \bigcap_{a=1}^p \widehat I_a.
\ee
We also suppose that the colors from 1 to $l<r$ are the common fixed points of all the $\{\sigma_a\}$,
\be 
\forall i \in \llbracket 1, l\rrbracket,\ \forall a \in \llbracket 1, p\rrbracket,\quad  \sigma_a(i) = i,
\ee 
where we use $\llbracket \cdot, \cdot \rrbracket$ to denote integer intervals, see \eqref{eq:def-integer-interval}. For each $a\in \llbracket 1,p\rrbracket$, among the colors $\llbracket l+1 , r-1 \rrbracket$, the colors $\widehat I_a \setminus \llbracket r, n\rrbracket  = \widehat I_a \cap \llbracket l+1,r-1\rrbracket$
are traced, and $\sigma_a$ only acts non-trivially on the colors   
\be 
J_a =  I_a \setminus \llbracket 1, l\rrbracket  = I_a \cap \llbracket l+1,r-1\rrbracket,
\ee
whose cardinal we denote 
\be
k = \cI - l.
\ee
Note that the sets $J_a$ generalize the color $j(a)$ of Section~\ref{sec:ExactMomentsABCD}. In the 4-partite case $ABCD$ we considered in Section~\ref{sec:ExactMomentsABCD}, we had $l=1$ (which corresponded to color $A$), $r=4$ (which corresponded to color $D$), and $k=1$. The only non-trivial action of the permutations $\sigma$ was either $J_a = \{2\} \equiv \{B\}$ or $J_a = \{3\} \equiv \{C\}$, color which was denoted by $j(a)$. 
	The notations in the general case are illustrated in Fig.~\ref{fig:BoxGeneral}. 

\begin{figure}[!ht]
\includegraphics[scale=0.7]{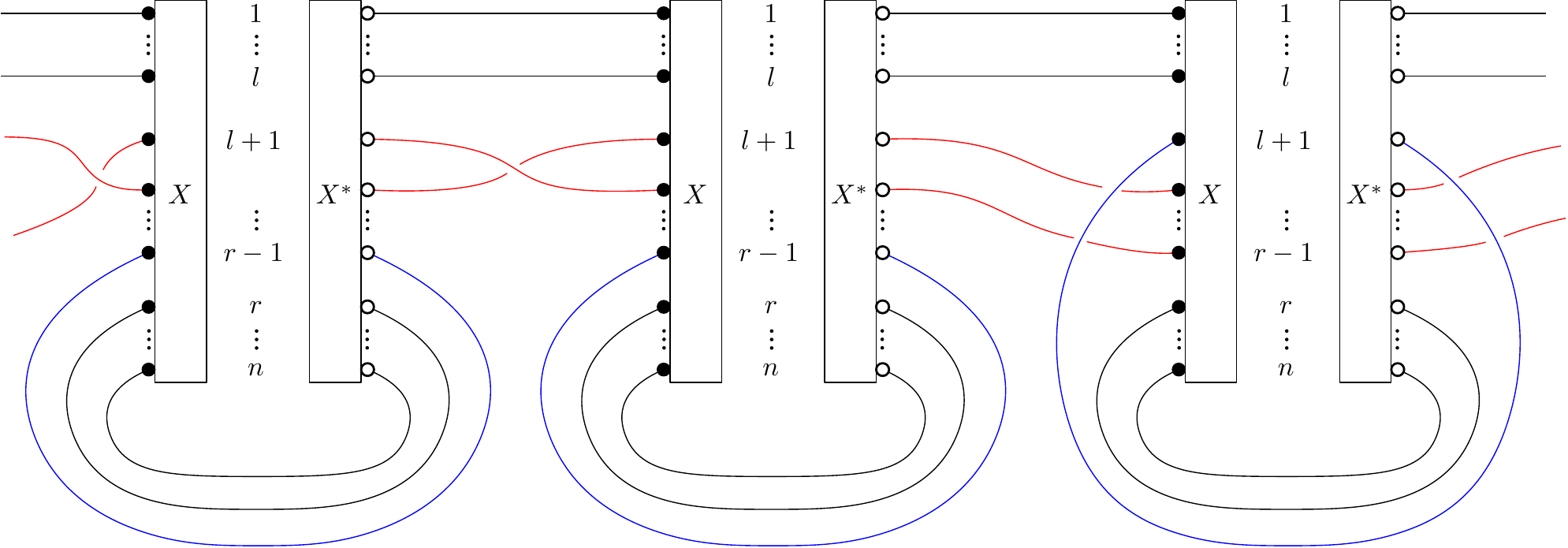}
\caption{A typical sample of an expectation \eqref{eq:ExpectationGenSigma}. The colors linked by red edges belong to the corresponding $J_a$s, while the blue edges link colors in the $I_a \setminus \llbracket 1, l\rrbracket $. }
\label{fig:BoxGeneral}
\end{figure}

Before moving on to the moment computation, let us point out that the data defining the moment $(\fsig,\bff)$ could be replaced by a single list of \emph{ordered} subsets of $\llbracket 1, n \rrbracket$.

\subsection{Exact expression for the moments}

\

\medskip
 
\begin{proposition}
\label{prop:Express1GenCase}
We suppose that  $N_{l+1} = \cdots = N_{r-1} = \NJ$\footnote{We put independent $N_i$'s for the colors $i\in\llbracket 1,n\rrbracket$ which are fixed points of all the initial wirings or which are always traced, and put a common $\NJ$ for the others. We stress however that more general cases might possibly be considered, if the colors can be separated into two sets if $\llbracket l+1,r-1\rrbracket = K_1\sqcup K_2$, such that the support of any orbit  is either included in $K_1$, or in $K_2$.  We may then choose a different $N_{J_1}$ and $N_{J_2}$ for colors in $K_1$ and $K_2$.}.  
Then, with the previous notations,
\be 
\bE \Tr_{\fsig} W_{\bff} (\{N_i\}) = \sum_{\alpha \in \cS_p}\, \prod_{i=1}^l N_i^{\#(\gamma\alpha)}\, \prod_{j=r}^n N_j^{\#\alpha}\, \NJ^{L(\bff, \fsig, \alpha )}, 
\ee 
in which $\gamma=(12 \cdots p)$, and $L(\bff, \fsig, \alpha )$ is a combinatorial function defined in \eqref{eq:def-L}.
\end{proposition}

\begin{proof} 

As in Section~\ref{sec:ExactMomentsABCD}, the moments will be expressed as a sum over Wick wirings $\alpha\in \cS_p$ -- or equivalently over maps in $\bM_p$ -- of some weight. Each loop in the box representation contributes to this weight with a factor  $N_i$. We want to describe these loops as {\it orbits} as was done in Def.~\ref{def:Orbits}, \textit{i.e.} in terms of the cycles of some permutations acting on the product of $r-l-1$ copies of $\llbracket 1, p\rrbracket$, of the form $[p]^i=\{(1,i), \ldots, (p,i)\}$ for each color $i$ in $\llbracket l+1, r-1\rrbracket$.

As detailed several times along this paper, to every permutation $\alpha\in \cS_p$ corresponds a combinatorial map $\cM$ with one black vertex and $\#\alpha$ white vertices, and whose $p$ labeled edges are disposed from 1 to $p$ counterclockwise around the black vertex and correspond to matrices $W_{I_a} = [X._{\widehat I_a}\bar X]$. The edge labeled $a$ therefore carries the set $I_a$ of $\cI$ colors, $l$ of which belong to every edge.

Let us take a closer look at the orbits in the map formulation.  As was previously the case in Section~\ref{sec:ExactMomentsABCD}, an orbit which has color $i$, when it arrives on a white vertex, leaves this vertex on the next edge also carrying the color $i$, counterclockwise. When an orbit of color $i$ arrives on a black vertex from the edge labeled $a$, it goes to the following edge around that vertex counterclockwise, labeled $a+1$, but changes color to $\sigma_a(i)$.  

 For the colors $1$ to $l$, these are actually the usual faces of the map, $F(\cM)$, as was the case for color $A$ in Section~\ref{sec:ExactMomentsABCD}. The colors $\llbracket 1,l \rrbracket$ therefore contribute with a factor $\prod_{i=1}^l N_i^{F(\cM)} = \prod_{i=1}^l N_i^{\#(\gamma\alpha)}$. {\it We may therefore as well forget these colors, and label the edges with the sets $J_a$ instead of $I_a$} (as was done for $j(a)$ in Sec.~\ref{sec:ExactMomentsABCD}).\\

The colors which belong to the sets $\widehat I_a$ will be taken care of further, and  we now focus on the remaining colors, which belong to at least one set $J_a$. In this general case, their behavior might differ from the particular case previously treated in Section~\ref{sec:ExactMomentsABCD}. Indeed, for colors in $\llbracket r+1, l-1 \rrbracket$, a given edge might appear several times on the same orbit (at most $k$ times). Therefore the orbits cannot in general be defined as the cycles of a permutation of the $p$ edges, or equivalently the faces of a combinatorial map without colors, which was the key point in Section~\ref{sec:ExactMomentsABCD}. In order to bypass this difficulty, we define a labeling for $pk$ copies of the edges, one per each couple $(a,i)$, where $a\in \llbracket 1, p\rrbracket$ labels the edge, and $i \in J_a$ is a color which is neither traced, nor a fixed point of all the $\sigma_b$. On these $pk$ elements, we define the following permutation
\be \label{eq:def-Gamma-f-sigma}
\Gamma_{\bff, \fsig} : (a, i) \mapsto (a+1, \sigma_a(i)).
\ee

We label each cycle of the permutation starting from the smallest color of the edge of smallest label in the cycle. For instance, if $J_1=\{B,C\}$, $J_2=\{B,D\}$, $J_3=\{C,D\}$,  $\sigma_1=\mathrm{id}$, $\sigma_2=\mathrm{id}$, and $\sigma_3$ is the transposition in $\cS_2$, then 
\be 
\label{eq:Ex1Gen}
\Gamma_{\bff, \fsig}=\Bigl( (1,B), (2,B), (3, C), (1, C), (2,D), (3,D) \Bigr),
\ee 
and if $J_1=\{B,C\}$, $J_2=\{B,D\}$, $J_3=\{B,C\}$, $J_4=\{C,D\}$, $\sigma_1=\mathrm{id}$, $\sigma_2=\mathrm{id}$, and  $\sigma_3$ and $\sigma_4$ are the transposition in $\cS_2$, then 
\be 
\label{eq:Ex2Gen}
\Gamma_{\bff, \fsig}=\Bigl( (1,B), (2,B), (3, B), (4, D)\Bigr)\Bigl((1,C), (2,D), (3,C) , (4,C) \Bigr).
\ee 

For colors in $\llbracket l+1, r-1 \rrbracket$, the behavior of the faces around white vertices is similar to Section~\ref{sec:ExactMomentsABCD}, and the duplication operation of Figure~\ref{fig:M23} generalizes locally:  performing the following duplication operation (illustrated for $\cI=3$) on every white vertex does not change locally the incident external orbit, and removes the color conditions on the white vertices.  
\begin{figure}[!ht]
\raisebox{
0.4cm}{\includegraphics[scale=0.9]{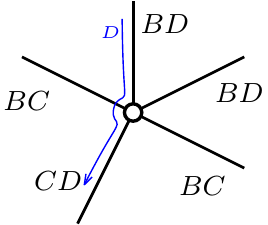}}\hspace{1cm}\raisebox{
1.4cm}{$\rightarrow$}\hspace{1cm}\includegraphics[scale=0.9]{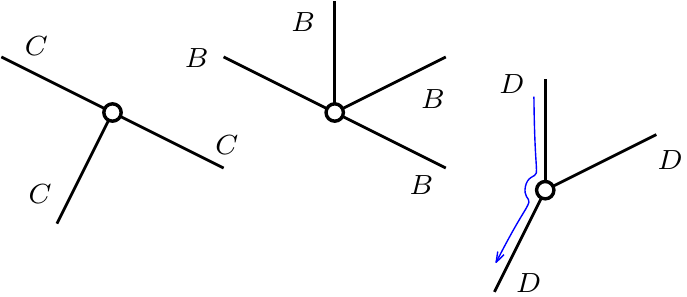}
\caption{\label{fig:M23-gen} Local duplication of white vertices. }
\end{figure}

This would precisely create one copy of an edge for each one of the $kp$ couples $(a,i)$. More precisely, if $i\in J_a$, denoting  $\alpha_i(a)$ the first edge following $a$ around the white endpoint of $a$ and containing color $i$, 
\be 
\alpha_i(a) = \alpha^q(a),\quad \text{where} \quad q=\min \{s\in\bN^\ast \mid i \in J_{\alpha^s(a)}\},
\ee 
($q$ depends on $\alpha$, $i$, and $a$) then the permutation defining the resulting white vertices is 
\be 
\alpha_{\bff} : (a, i) \mapsto (\alpha_i(a), i).
\ee 

We can now define a (non-necessarily connected) combinatorial map $\cM_{\bff, \fsig} = (\Gamma_{\bff, \fsig}, \alpha_{\bff})$ from the two permutations $\Gamma_{\bff, \fsig}$ and $\alpha_{\bff}$,  acting on  $\llbracket 1, p\rrbracket \times \llbracket 1, k\rrbracket $. See the examples in Fig.~\ref{fig:Ex1NewMap} and \ref{fig:ExsNewMap}. 
\begin{figure}[!ht]
\raisebox{
0.2cm}{\includegraphics[scale=0.7]{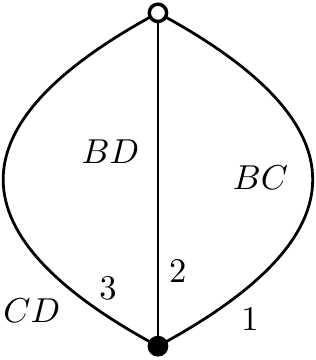}}\hspace{0.7cm}\raisebox{
1.4cm}{$\rightarrow$}\hspace{0.6cm}\includegraphics[scale=0.8]{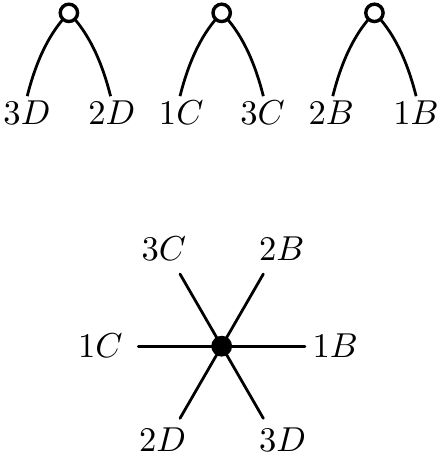}\hspace{0.7cm}\raisebox{
1.4cm}{$\rightarrow$}\hspace{0.6cm}\includegraphics[scale=0.7]{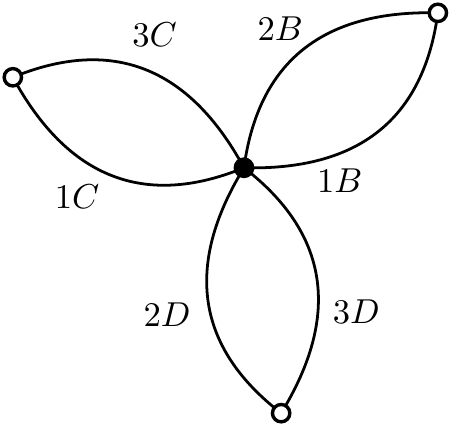}
\caption{\label{fig:Ex1NewMap}  The map $\cM_{\bff, \fsig}$ for an example in the case \eqref{eq:Ex1Gen}. The white vertex in the original map $\cM$ is tripled into 3 white vertices having half-edges of the same color attached. The order of the half-edges around the black vertex is given by the permutation $\Gamma_{\bff, \fsig}$.}
\end{figure}
\begin{figure}
\raisebox{
0.8cm}{\includegraphics[scale=0.7]{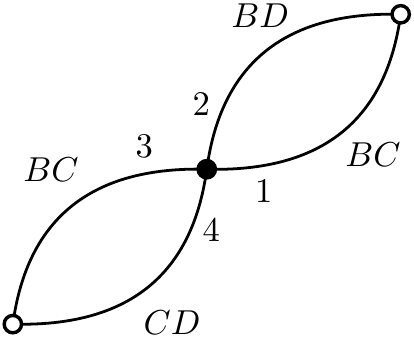}}\hspace{0.7cm}\raisebox{
2cm}{$\rightarrow$}\hspace{0.55cm}\includegraphics[scale=0.7]{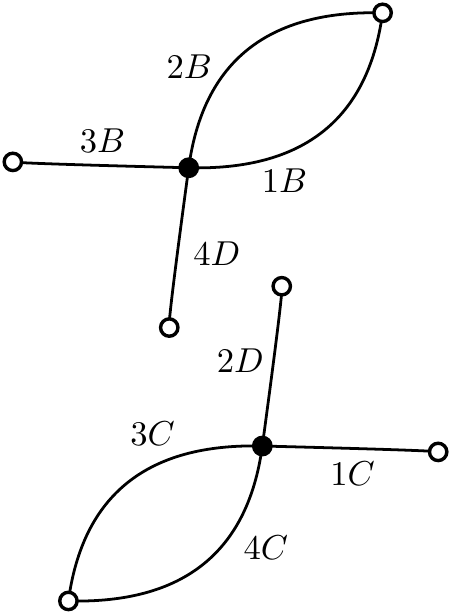}
\hspace{1.5cm}
\raisebox{
0.8cm}{\includegraphics[scale=0.7]{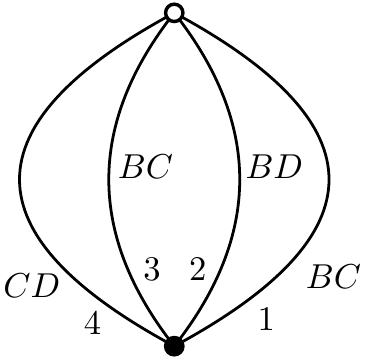}}\hspace{0.7cm}\raisebox{
2cm}{$\rightarrow$}\hspace{0.6cm}\includegraphics[scale=0.8]{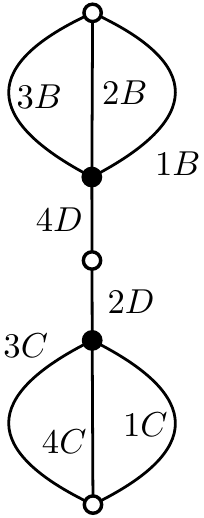}
\caption{\label{fig:ExsNewMap} The maps $\cM_{\bff, \fsig}$ for two examples in the case \eqref{eq:Ex2Gen}. }
\end{figure}

Among the orbits which act on colors in $\llbracket l+1, n \rrbracket$, one has to take into account the orbits which are entirely included in the white vertices of $\cM$. The number of such orbits contained in a given vertex is $l$, plus the number of colors in $\llbracket l+1, r-1\rrbracket$ which do not appear on any edge. The total number of the latter can be expressed as 
\be 
(r-l-1)\bigl(V(\cM) - 1\bigr) - \bigl( V(\cM_{\bff, \fsig}) - V_\bullet\bigr(\cM_{\bff, \fsig})),
\ee 
where we have denoted the number of black vertices of $\cM_{\bff, \fsig}$ as
$$
V_\bullet\bigr(\cM_{\bff, \fsig}) =V_\bullet(\bff, \fsig) = \#\Gamma_{\bff, \fsig}.
$$ 
This is an integer between 1 and $k$ {\it which does not depend on $\alpha$}.
These orbits therefore contribute with a factor $$ \prod_{i=r}^n N_i^{V(\cM) - 1} \NJ^{(r-l-1)(V(\cM) - 1) - ( V(\cM_{\bff, \fsig}) - V_\bullet(\bff, \fsig))},$$ where we recall that\footnote{Notice also  that if one wants to make explicit the analogy between $\cM$ and $\cM_{\bff, \fsig}$ in this expression, one can rewrite it as $\prod_{i=r}^n N_i^{V(\cM) - 1} \NJ^{(r-l-1)(V(\cM) - V_\bullet(\cM)) - ( V(\cM_{\bff, \fsig}) - V_\bullet(\bff, \fsig))}$ or   $\prod_{i=r}^n N_i^{V(\cM) - 1} \NJ^{(r-l-1)(V(\cM) - \#\gamma) - ( V(\cM_{\bff, \fsig}) - \#\Gamma_{\bff, \fsig})}$} $V(\cM) - 1 = \#\alpha$.\\

The remaining orbits run along at least one edge. Because the permutations  $\Gamma_{\bff,\fsig}$ and $\alpha_{\bff}$ encode precisely the way they behave locally around the black and the white vertices respectively, their total number is precisely given by the number of faces of the map $\cM_{\bff, \fsig}$, 
$$F(\cM_{\bff, \fsig}) = \#(\Gamma_{\bff,\fsig}\, \alpha_{\bff}),$$
hence contributing with a factor $\NJ^{F(\cM_{\bff, \fsig}) }$.
Putting all of this together, we have shown that 
\be 
\bE \Tr_{\fsig} W_{\bff} (\{N_i\}) = \sum_{\cM \in \bM_p}\, \prod_{i=1}^l N_i^{F(\cM)}\, \prod_{j=r}^n N_j^{V(\cM)-1}\, \NJ^{L(\bff, \fsig, \cM )}, 
\ee 
where we have defined 
\be\label{eq:def-L}
L(\bff, \fsig, \cM ) := F(\cM_{\bff, \fsig}) + (r-l-1)\bigl(V(\cM) - 1\bigr) - \bigl( V(\cM_{\bff, \fsig}) - V_\bullet\bigr(\cM_{\bff, \fsig})).
\ee 
Identifying $L(\bff, \fsig, \alpha ) = L(\bff, \fsig, (\gamma, \alpha) )$, this concludes the proof. \end{proof}

\
In the above proof we defined the companion map $\cM_{\bff,\fsig}$ to $\cM$. This companion map is a generalization of the previous companion map $\cM_f$ introduced in section \ref{sec:ExactMomentsABCD}. Since this is an important object, we recall its construction in the following definition
\begin{definition}
For all integer $p$, we associate to the triplet $(\cM,\bff, \fsig)$, with $\cM$ a combinatorial map in $\bM_p$, $\bff$ the list of colors of the edges of $\cM$ and $\fsig$ the list of permutations labeling the corners of the black vertex of $\cM$, a combinatorial map $\cM_{\bff, \fsig}$ called \emph{unfolded map} from the following data:
\begin{itemize}
    \item The set $E_{\bff}:=\bigl\{(a,i)\bigr\}_{a \in \llbracket 1,p \rrbracket, \, i\in J_a}$ is the set of edges.
    \item The permutations $\alpha_{\bff}$ and $\Gamma_{\bff, \fsig}$, where $\alpha_{\bff}$ is the permutation that defines the white vertices after local duplication, and $\Gamma_{\bff, \fsig}$ defines the black vertices of $\cM_{\bff, \fsig}$. Denoting
    \begin{equation}
         \alpha_i(a) = \alpha^q(a),\quad \text{where} \quad q=\min \{s\in\bN^\ast \mid i \in J_{\alpha^s(a)}\}, \nonumber
    \end{equation}
    the permutation $\alpha_{\bff}$ is defined as
    \begin{align}
     \alpha_{\bff}:& \ E_{\bff} \rightarrow E_{\bff} \nonumber \\
     & (a,i) \mapsto (\alpha_i(a),i), \nonumber
    \end{align}
    and $\Gamma_{\bff, \fsig}$ is defined as
    \begin{align}
        \Gamma_{\bff, \fsig}:& \ E_{\bff} \rightarrow E_{\bff} \nonumber\\
        &(a,i)\mapsto (a+1,\sigma_a(i)). \nonumber
    \end{align}
\end{itemize}
\end{definition}
If we compare with Definition~\ref{def:comb-maps} then we notice that we do not require that the group $\langle\alpha_{\bff}, \Gamma_{\bff,\fsig} \rangle$ acts transitively on $E_{\bff}$, this is because $\cM_{\bff, \fsig}$ can be disconnected. 
Remark also that each white vertex of $\cM$ is duplicated into $k=|J_a|$ white vertices of $\cM_{\bff, \fsig}$ if all the incident edges are labeled by the same set $J_a$; and into strictly more than $k$ white vertices if at most $k-1$ colors are common to all the edges. We therefore define the quantity
\be \label{eq:def-Delta}
\Delta_{\bff, \fsig}(\cM) := V(\cM_{\bff, \fsig}) - V_\bullet(\bff, \fsig)  - k(V(\cM) - 1),
\ee 
which vanishes if and only if all the edges incident to a given  white vertex of $\cM$ share the same color set $J_a$, and is positive otherwise.\\

Moreover, if the number of black vertices is $V_\bullet(\bff, \fsig)$, the number of connected components $\Kc$ of a given map $\cM_{\bff, \fsig}$ is an integer in $\{1, \ldots, V_\bullet(\bff, \fsig)\}$, so that we define 
\be 
\Sigma(\cM_{\bff, \fsig}) = V_\bullet(\bff, \fsig) - \Kc(\cM_{\bff, \fsig}),
\ee 
which is an integer between 0 and $V_\bullet(\bff, \fsig)-1$. 

\begin{theorem}
\label{thm:TheoremGen}
With the previous notations, the mixed moments of the marginals $\{W_{I_a}\}$ are expressed exactly using a sum over combinatorial maps
 \begin{equation}
 \label{eq:TheoremGenEqn}
 \bE \Tr_{\fsig} W_{\bff}  = \sum_{\cM \in \bM_p}\, \prod_{i=1}^l N_i^{2+p -V(\cM) - 2g(\cM)} \prod_{j=r}^n N_j^{V(\cM)-1} \NJ^{kp + V_\bullet(\bff, \fsig) + (r-l-2k-1)(V(\cM)-1)-\tilde L ( \bff, \fsig, \cM)}, 
 \end{equation}
 in which we have denoted
 \be \label{eq:def-tilde-L}
 \tilde L ( \bff, \fsig, \cM) = 2\bigl(g(\cM_{\bff, \fsig})+ \Delta_{\bff, \fsig}(\cM) + \Sigma(\cM_{\bff, \fsig})\bigr), 
 \ee
 where $g(\cM_{\bff, \fsig})\ge0$ is the genus of the map $\cM_{\bff, \fsig}$, and $\Delta_{\bff, \fsig}(\cM)\ge0$ and $\Sigma(\cM_{\bff, \fsig})\ge0$ have been defined above.
\end{theorem}

Before starting with the proof, notice that the quantity $\Delta_{\bff, \fsig}(\cM)$ seems to penalize the mixed cumulants in the large $N_J$ regime, if we trust the heuristic relating monochromatic white vertices (i.e.~white vertices incident to edges of the same color) to non-mixed cumulants. This is an indication for freeness at large $N_J$. 

\begin{proof} The only thing we need to prove is that $L(\bff, \fsig, \alpha ) = kp + V_\bullet(\bff, \fsig) + (r-l-2k-1)(V(\cM)-1)-\tilde L ( \bff, \fsig, \cM)$, which follows from the Euler characteristics of the map $\cM_{\bff, \fsig}$, 
\be 
2\Kc(\cM_{\bff, \fsig}) - 2g(\cM_{\bff, \fsig}) = F(\cM_{\bff, \fsig}) - kp + V(\cM_{\bff, \fsig}).
\ee 
We conclude using Proposition~\ref{prop:Express1GenCase} and the definitions of $\Delta_{\bff, \fsig}(\cM)$ and $\Sigma(\cM_{\bff, \fsig})$. \end{proof}

\begin{remark}
In the case where $k=0$, the result above degenerates, and we recover the classical Mar{\v{c}}enko-Pastur result from Proposition~\ref{prop:MP} (see also equation \eqref{eq:MP-limit-with-comb-maps} for the combinatorial map approach). Indeed, in this case $r=l+1$, and the factor with $\NJ$ is trivially equal to 1; moreover, the quantities $\fsig$ and $\bff$ are trivial and do not play any role. The result reads
$$\bE \Tr W_n  = \sum_{\cM \in \bM_p}\, \prod_{i=1}^l N_i^{2+p -V(\cM) - 2g(\cM)} \prod_{j=r}^n N_j^{V(\cM)-1}.$$
If we denote $N_{tot}:=\prod_{i=1}^l N_i$ and consider the scaling where $N'_{tot}:= \prod_{j=r}^n N_j \sim c N_{tot}$, we obtain
$$\bE \Tr W_n  = (1+o(1)) N_{tot}^{1+p}\sum_{\cM \in \bM_p} N_{tot}^{-2g(\cM)}c^{V(\cM)-1}.$$
Assuming moreover $N_{tot} \to \infty$, the limit selects planar maps, and we obtain \eqref{eq:MP-limit-with-comb-maps}:
$$\lim_{N_{tot} \to \infty}\bE \frac{1}{N_{tot}}\Tr \left(\frac{W_n}{N_{tot}}\right)^p = \sum_{\cM \in \bM_p^0}c^{V(\cM)-1}.$$
\end{remark}

\begin{remark}
The first non-trivial case corresponds to $k=1$, where just one of the $n$ legs of the tensors can ``move around''. In this case, it is clear that the permutations $\fsig$ do not play any role, and $\bff$ is just a function from $\llbracket 1, p\rrbracket$ to the set of tensor legs, selecting for each tensor the leg which ``moves around''.
We have $\bff:\llbracket 1, p\rrbracket \to \llbracket r \rrbracket]$, and, since $\fsig$ is trivial, $V_\bullet(\bff)=1$. Assuming, to keep things simple, that all the vector spaces have dimension $N$, the result of Theorem~\ref{thm:TheoremGen} reads
\begin{align*}
\bE \Tr W_\bff &= \sum_{\cM \in \bM_p} N^{(r-1)[2+p-V(\cM)-2g(\cM)]}N^{V(\cM)-1}N^{p + V_\bullet(\bff) + (r-2)(V(\cM)-1)-\tilde L ( \bff,\cM)}\\
&= N^{r(p+1)}\sum_{\cM \in \bM_p} N^{-2(r-1)g(\cM) - \tilde L(\bff,\cM)}.
\end{align*}
In order to find the dominating contributions, we have to identify the maps $\cM \in \bM_p$ which cancel both $g(\cM)$ and $\tilde L(\bff,\cM)$. The first condition corresponds to $\cM$ being planar, while the second one is equivalent to the cancellation of the three quantities $g(\cM_{\bff})$ ,$\Delta_{\bff}(\cM)$, $\Sigma(\cM_{\bff})$ from \eqref{eq:def-tilde-L}. We focus on the $\Delta$ quantity: it follows from \eqref{eq:def-Delta} that this quantity is zero iff all the edges incident to a white vertex in $\cM$ have the same color given by $f$; in other words, the partition introduced by the white vertices of $\cM$ on $\llbracket 1, p\rrbracket$ has to be smaller than the partition $\bff$ introduces on the same set. 
 We claim now that the conditions $g(\cM) = 0$ and $\Delta_{\bff}(\cM)=0$ imply that the other two quantities appearing with a negative sign in the exponent of $N$ cancel. Indeed, $\Sigma(\cM_{\bff})=0$, since $V_\bullet(\bff)=1$. Since $\Delta_{\bff}(\cM)=0$, it also follows that $\cM_{\bff}=\cM$, and thus $g(\cM_{\bff})=0$. To summarize, we reach the same conclusion as in Theorem~\ref{thm:2-marginals-large-BC}, see equation \eqref{eq:moments-marginals-AB-AC-balanced-MAPS}:
$$\lim_{N \to \infty} N^{-r(p+1)} \bE \Tr W_\bff = |\bM_p^{0,\bff}|,$$
where
\begin{align*}
\bM_p^{0,\bff} := \{ \cM \in \bM_p \, : \, &\cM \text{ is planar and $\bff$ assigns the same color to the edges}\\
& \text{incident to the white vertices of $\cM$}\}.
\end{align*}
\end{remark}

\

\subsection{The balanced asymptotical regime}
\label{sec:BalancedGeneral}

\

\medskip

As a first step, we focus on the limit where $N_1, \ldots, N_n$ all grow to infinity, while the ratio between $N_1\ldots N_l$ and $N_r \ldots N_n$ converges to a fixed constant at infinity.
More precisely, we can consider for instance $\forall i \in \llbracket 1,l \rrbracket, N_i=\NJ=N$ while $\forall j \in [\![r,n]\!], \ \textrm{lim}_{N\rightarrow \infty} N_j/N = c^{\frac 1 {n-r+1}}$ or pick $j \in [\![r,n]\!]$ such that $N_j \sim cN$ at infinity, while $N_i=N$ for $i\neq j$.

In that case, Theorem~\ref{thm:TheoremGen} writes 
 \begin{equation}
 \label{eq:MomentsGenLargeRegime}
 \bE \Tr_{\fsig} W_{\bff}  = N^{(k+l)(p-2) + l + n  + V_\bullet(\bff, \fsig)}\, \sum_{\cM \in \bM_p}\, h(N)^{V(\cM)-1}  N^{ - 2lg(\cM) -\tilde L ( \bff, \fsig, \cM) -(2k+2l - n)(V(\cM) - 2)},
  \end{equation}
  where $h(N)$ is an arbitrary function  such that $h(N) \rightarrow c$ when $N\rightarrow + \infty$. Notice that the factor $N^{V_\bullet(\bff, \fsig)}$ is out of the sum since, as already emphasized earlier, it only depends on $\bff,\fsig$. The terms $g(\cM)$, $\tilde L ( \bff, \fsig, \cM)$ and 
  $V(\cM) - 2$ are all non-negative, so we have three cases, depending on the sign of $2k+2l - n$. 
  
  \
  
  If $2k+2l - n<0$, the maps $\cM$ which survive in the large $N$ limit are those which satisfy
  \be
  g(\cM)=0\ (\text{if } l\neq0),\quad  \tilde L ( \bff, \fsig, \cM)=0,\quad \text{and}\quad V(\cM) = p + 1.
  \ee
  Since there is a unique solution to this system, given by the only tree in $\bM_p$, in that case, in the large $N$ limit,
    \be
   \lim_{N\rightarrow\infty} \frac 1 {N^{(n-k-l)p + l + V_\bullet(\bff, \fsig) }}\, \bE \Tr_{\fsig} W_{\bff}  =   c^{p}.
  \ee
  This case is therefore trivial, and we do not consider it. In the following, we first consider the case $2k+2l - n=0$ in Subsection~\ref{subsub:HalfColorsTraced}, and then the case $2k+2l - n>0$ in Subsection~\ref{subsub:LessHalfColorsTraced}.
  
  \smallskip

\subsubsection{\underline{Case where half the colors are traced}}
\label{subsub:HalfColorsTraced}

\

\medskip

In this section we study the case where $2k+2l - n$ vanishes, or equivalently  
\be 
\cI = l + k = \frac n 2,
\ee 
the number of colors $n$ is even, and for each edge, precisely half the colors are traced. The equation \eqref{eq:MomentsGenLargeRegime} rewrites as
 \begin{equation}
 \bE \Tr_{\fsig} W_{\bff}  = N^{(k+l)p + l + V_\bullet(\bff, \fsig) }\sum_{\cM \in \bM_p}\,   h(N)^{V(\cM)-1}  N^{ - 2lg(\cM) -\tilde L ( \bff, \fsig, \cM)},
 \end{equation}
where here $ h(N)$ is an arbitrary function whose limit at infinity is $ c$.\\ 
 
\

\noindent{\bf Large $N$ limit, case $l\neq 0$.}
In the large $N$ limit, the sum restricts to planar maps, which are such that edges incident to a given white vertex have the same set of colors $I_a$, and for which the map $\cM_{\bff, \fsig}$ is planar and has as many connected components as black vertices.  We define the set of ``free" maps 
\be
\label{eq:FreeMaps}
\bM_{\bff, \fsig}^{(0) \text{free}} := \bigl\{\cM\in \bM_p \mid g(\cM) = 0, g(\cM_{\bff, \fsig}) = 0, \Delta_{\bff, \fsig}(\cM)=0, \text{ and }\Sigma(\cM_{\bff, \fsig}) = 0\bigr\}.
\ee
The following result is just a consequence of Theorem~\ref{thm:TheoremGen}.
\begin{proposition}
\label{prop:LargeNinTheLarge1nRegime}
In the large $N_1,\ldots, N_n$ regime, if $ l + k = \frac n 2$ and with the previous notations,
 \begin{equation}
 \lim_{N\rightarrow\infty} \frac 1 {N^{\frac n 2 p + l + V_\bullet(\bff, \fsig) }}\, \bE \Tr_{\fsig} W_{\bff}  = \sum_{\cM \in \bM_{\bff, \fsig}^{(0) \text{free}}} c^{\, V(\cM) - 1}. 
 \end{equation}
\end{proposition}

Note that since $V_\bullet(\bff, \fsig)$ might be smaller than $k$, the moments corresponding to $\fsig$ and  $\bff$ might have a weaker scaling in $N$ than the usual $N^{\frac n 2 (p + 1)}$.

\

\smallskip

\noindent{\bf General case including $l=0$.} In the case where $l=0$, the map $\cM$ is a priori non necessarily planar in the large $N$ limit. In this paragraph, we show that maps with vanishing $\tilde L$ are in fact necessarily planar. As a consequence, the system defining the ``free" maps in \eqref{eq:FreeMaps} can in fact be reduced to
\be
\label{eq:FreeMaps2}
\bM_{\bff, \fsig}^{(0) \text{free}} = \bigl\{\cM\in \bM_p \mid  g(\cM_{\bff, \fsig}) = 0, \Delta_{\bff, \fsig}(\cM)=0, \text{ and }\Sigma(\cM_{\bff, \fsig}) = 0\bigr\},
\ee
thus extending the preceding result to the case $l=0$.
 Indeed, we have the following result 
\begin{proposition}
\label{prop:ShorterFreeMap}
A map $\cM\in\bM_p$ in a triplet $(\cM, \bff, \fsig)$ which satisfies 
\be 
\label{eq:CondLib}
g(\cM_{\bff, \fsig}) = 0 \quad \text{and} \quad \Delta_{\bff, \fsig}(\cM)=0,
\ee 
is necessarily planar: $g(\cM) = 0$. 
\end{proposition}

\proof Relying on the following Lemma~\ref{lemma:IneqFreeGen}, the proposition is proven easily by noticing that, if $g(\cM_{\bff, \fsig}) = 0$, the genus $g(\cM)$ is a non-negative integer which is smaller than one. \qed

\begin{lemma} 
\label{lemma:IneqFreeGen}
If $\Delta_{\bff, \fsig}(\cM)=0$, we have the following bound for  the number of faces  of the combinatorial map $\cM_{\bff, \fsig}$, 
\be 
\label{IneqFreeGen}
F(\cM_{\bff, \fsig})\le k F(\cM), \quad \text{ and } \quad  g(\cM)  \le \frac {g(\cM_{\bff, \fsig})}{k} + \frac{k-1 }{k}.
\ee 
\end{lemma} 

\begin{proof} Because $\Delta_{\bff, \fsig}(\cM)=0$, the orbits, though different from the faces, cover precisely the faces of $\cM$. The difference is that if a face starts on an edge $a\in \llbracket 1,p \rrbracket$ with a color $i \in J_a$, after going around $\cM$, it might come back for the first time on the same side of the edge $a$ for a color $j\neq i$. In that case, it would follow the face of $\cM$ again a certain number of times, until it reaches again the same side of the edge $a$, for the same color $i$. An orbit, which is projected to a face of $\cM_{\bff, \fsig}$, is therefore a multiple of a face of $\cM$. Each side of each edge $a$ is used exactly once per color in $J_a$. If all the orbits go exactly once around the map, then each face of $\cM$ corresponds to $k$ orbits (and thus faces of $\cM_{\bff, \fsig}$), as they never pass twice on the same side of an edge, and an edge has $k$ colors which are not traced. If not, there are less orbits, so that the first inequality is proven. 

To prove the second inequality, we consider the Euler characteristics of $\cM$, $2g(\cM) = 2 - V(\cM) + p - F(\cM)$. Since $\Delta_{\bff, \fsig}(\cM)=0$, we replace $V(\cM_{\bff, \fsig}) - V_\bullet(\bff, \fsig)= k(V(\cM) - 1)$, 
\begin{align} 
2kg(\cM) &= 2k - kV(\cM) + kp - kF(\cM) \\
& = k - \Bigl(V(\cM_{\bff, \fsig}) - V_\bullet(\bff, \fsig)\Bigr)+ kp - kF(\cM).
\end{align} 
Now using the Euler characteristics of $\cM_{\bff, \fsig}$, $  kp - V(\cM_{\bff, \fsig})  =   2g(\cM_{\bff, \fsig}) + F(\cM_{\bff, \fsig}) - 2K(\cM_{\bff, \fsig})$, we find that 
\begin{equation} 
2kg(\cM) = k + 2g(\cM_{\bff, \fsig}) + F(\cM_{\bff, \fsig}) - 2K(\cM_{\bff, \fsig}) + V_\bullet(\bff, \fsig)- kF(\cM).
\end{equation}  
Using the first inequality of \eqref{IneqFreeGen}, we are left with 
\begin{equation} 
2kg(\cM) \le  k + 2g(\cM_{\bff, \fsig})  - 2K(\cM_{\bff, \fsig}) + V_\bullet(\bff, \fsig).
\end{equation}  
We conclude using that $K(\cM_{\bff, \fsig}) \ge 1$ and $V_\bullet(\bff, \fsig) \le k$. 
\end{proof}

\begin{remark} A consequence of Proposition~\ref{prop:ShorterFreeMap} is that keeping the  $N_i$ finite for $i\in\llbracket 1,l\rrbracket$ does not lead to more interesting behavior. Indeed, a map with $\tilde L ( \bff, \fsig, \cM)=0$ also has vanishing genus, so that the large $N$ limit also selects all the maps  in $\bM_{\bff, \fsig}^{(0) \text{free}}$, and we recover the results of the present case.

\end{remark}

\smallskip

\
 
\noindent{\bf Application to the asymptotic freeness of permuted marginals of random Gaussian tensors.}
We now consider an application of Theorem~\ref{thm:TheoremGen} to a new situation, where we consider the set of all possible marginals of the tensor $X\otimes X^*$, with all possible permutations acting on half of the tensor legs. More precisely, let $n$ be even, and put, for an $(n/2)$-subset $I$ of $\llbracket 1,n\rrbracket$, $W_I = [\mathrm{id}_{I} \otimes \mathrm{Tr}_{\llbracket 1,n\rrbracket \setminus I}](XX^*)$, where $X \in (\mathbb C^N)^{\otimes n}$ is a Gaussian tensor; note that we assume here, for the sake of simplicity, that all the Hilbert space dimensions are $N$. The general case of different Hilbert space dimensions, say $N_i \sim c_i N$ for some constants $c_i >0$ and $N \to \infty$, can be easily obtained.

Let us introduce, for every permutation $\pi \in \mathcal S_{n/2}$,
\be 
\label{eq:WPi}
W_{I,\pi} = P_\pi W_I P_\pi^{-1} \in \mathcal M_{N^{n/2}}(\mathbb C)
\ee
the matrix obtained by permuting the tensor legs of $W_I$ according to the permutation $\pi$. For example, in the case $n=4$, there are 12 possible matrices $W_{I,\pi}$. We present in Figure~\ref{fig:W-1-2-vs-W-12} two examples: $W_{\{1,2\},(1)(2)} = W_{\{1,2\}}$ and $W_{\{1,3\},(12)} = F W_{\{1,3\}} F$, where $F \in \mathcal U_{d^2}$ is the \emph{flip operator}, acting on simple tensors by $F(x \otimes y) = y \otimes x$. 
\begin{figure}[!ht]
\includegraphics[scale=1]{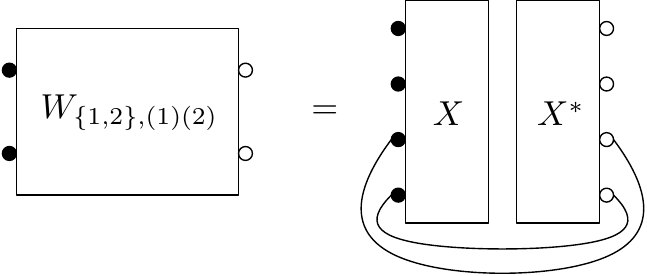} \qquad \includegraphics[scale=1]{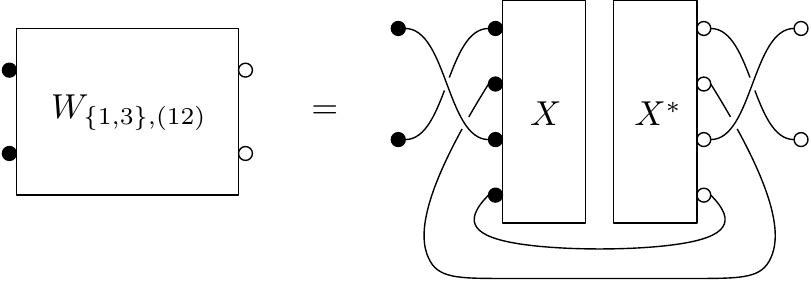}
\caption{Diagrams for the permuted marginals $W_{\{1,2\},(1)(2)}$ and $W_{\{1,3\},(12)}$.}
\label{fig:W-1-2-vs-W-12}
\end{figure}

This is therefore a particular case of the large $N_1,\ldots, N_n$ regime, for specific permutations $\sigma_a\in\fsig$, which factorize as 
\be
\label{eq:FactorizeSigma}
\forall a\in\llbracket 1,p\rrbracket, \qquad \sigma_a = \pi_{a+1}^{-1}\pi_a.
\ee
Note that the permutations $\sigma_a$ might share some fixed points (recall that we denote by $l$ the number of these common fixed points, and we also put $k=n/2-l$). Proposition~\ref{prop:LargeNinTheLarge1nRegime} applies, and the maps that contribute to the large $N$ limit are a subset of $\bM_{\bff, \fsig}^{(0) \text{free}}$: in particular, they are planar, and the edges that are attached to the same white vertices all have the same set of colors $J_a$.  Moreover, the hypothesis \eqref{eq:WPi} has stronger  consequences. Indeed, we have the following lemmas.

\begin{lemma} 
\label{lemma:BlackVerticesForWPi}
If the permutations in $\fsig$ factorize as in \eqref{eq:FactorizeSigma}, for any $\cM\in\bM_p$, the map $\cM_{\bff, \fsig}$ has precisely $k$ black vertices. As a consequence, if $\Sigma(\cM_{\bff, \fsig})=0$, the maps $\cM_{\bff, \fsig}$ have $k$ connected components.
\end{lemma}
\begin{proof}
In the combinatorial maps setting, the picture is the following: when arriving to an edge,
the color of an orbit changes by multiplying with $\pi_a^{-1}$, and when leaving the edge, one multiplies with the inverse  permutation $\pi_a$. Formally, since the black vertices of $\cM_{\bff, \fsig}$ correspond to the cycles of the permutation $\Gamma_{\bff, \fsig}$ from \eqref{eq:def-Gamma-f-sigma}, we show that $\Gamma_{\bff, \fsig}$ has precisely $k$ cycles. Indeed, we have
$$(a,i) \xrightarrow{\Gamma_{\bff, \fsig}} (a+1, \pi_{a+1}^{-1}\pi_a(i)) \xrightarrow{\Gamma_{\bff, \fsig}} (a+2, \underbrace{\pi_{a+2}^{-1}\pi_{a+1}\pi_{a+1}^{-1}\pi_a(i)}_{\pi_{a+2}^{-1}\pi_a(i)})\xrightarrow{\Gamma_{\bff, \fsig}}  \cdots \xrightarrow{\Gamma_{\bff, \fsig}} (\underbrace{a+p}_a, \underbrace{\pi_{a+p}^{-1}\pi_a(i)}_i),$$
for all $i \in J_a$, proving the claim (the addition on the first coordinate is done modulo $p$). 
\end{proof}

\begin{lemma}\label{lem:free-implies-pi-constant}
If  $\cM$ is such that $\Delta_{\bff, \fsig}(\cM)=\Sigma(\cM_{\bff, \fsig})=0$, 
and the permutations in $\fsig$ factorize as in  \eqref{eq:FactorizeSigma}, the edges incident to a given white vertex are all labeled by the same permutation $\pi$. 
\end{lemma}
\begin{proof}
Consider two edges $a$ and $b$ incident to the same white vertex in $\cM$. As $\Delta_{\bff, \fsig}(\cM)=0$, they have the same set of colors $J_a=J_b$. Take a color $i\in J_a$. In the unfolded map $\cM_{\bff, \fsig}$, the white vertex corresponding to the color $i \in J_a$ has attached to it, among others, the edges $(a,i)$ and $(b,i)$. The edge $(a,i)$ has its other end attached to the black vertex corresponding to the cycle containing $(a,i)$ in the permutation $\Gamma_{\bff, \fsig}$. In order for the map $\cM_{\bff, \fsig}$ to have exactly $k$ connected components (see Lemma~\ref{lemma:BlackVerticesForWPi}), the edge $(b,i)$ must be connected to the same black vertex (otherwise, the aforementioned white vertex would be connected to two different black vertices, decreasing the number of connected components). However, from the proof of Lemma~\ref{lemma:BlackVerticesForWPi}, we see that the black vertex has attached to it the following edges: $\{(c,\pi_c^{-1}\pi_a(i))\}_{c \in \llbracket 1,p\rrbracket}$. Hence, we must have $\pi_b^{-1}\pi_a(i) = i$ for all $i$, which is the claim.
\end{proof}

\begin{lemma}\label{lem:rCopiesMap}
If  $\cM$ is such that $\Delta_{\bff, \fsig}(\cM)=\Sigma(\cM_{\bff, \fsig})=0$,  and the permutations in $\fsig$ factorize as in  \eqref{eq:FactorizeSigma}, the unfolded map $\cM_{\bff, \fsig}$ consists of $k$ copies of the original map $\cM$ (when discarding the edge coloring). 
\end{lemma}
\begin{proof}
 From the proof of Lemma~\ref{lemma:BlackVerticesForWPi}, $\cM_{\bff, \fsig}$ has $k$ connected components, and each one of them has one black vertex with $p$ incident edges, ordered from $1$ to $p$. Discarding the edge coloring,
 the permutation defining the black vertex of each one of the connected components is therefore $\gamma$. These $p$ edges are attached to white vertices in $\cM_{\bff, \fsig}$.
Now since all the edges $a$ incident to some white vertex $v$ of $\cM$ have the same sets $J_a$, each white vertex $v$ of $\cM$ corresponds to $k=|J_a|$ white vertices $\{v^i\}_{i \in J_a}$ in the unfolded map $\cM_{\bff, \fsig}$, with the same edges $(a, \alpha(a), \alpha^2(a), \cdots)$ attached to it, in the same order. Therefore, discarding the edge coloring,
the permutation defining the white vertices of a given connected component of $\cM_{\bff, \fsig}$ is just $\alpha$. This proves that each connected component of  $\cM_{\bff, \fsig}$ is isomorphic to $\cM$.
\end{proof}

We arrive now at one of the main results of this paper, the asymptotic freeness of \emph{all} the (permuted) balanced marginals of a random multipartite quantum state.

\begin{theorem}\label{thm:asymptotic-freeness-balanced-general}
In the asymptotic setting described above, the family of random matrices 
$$(W_{f,\pi})_{f \in \binom{\llbracket 1,n \rrbracket}{n/2},\, \pi \in \mathcal S_{n/2}}$$ converges, as $N \to \infty$, to a family of $\binom{n}{n/2}(n/2)!$ free random variables, each having a Mar{\v{c}}enko-Pastur distribution of parameter $1$. At the level of moments, this reads
\begin{equation}\label{eq:moment-formula-Wishart-permuted}
\lim_{N \to \infty} \bE N^{-n(p+1)/2} \Tr W_{\boldsymbol{f},\boldsymbol{\pi}} = \sum_{\cM \in \bM^{bal}_{\boldsymbol{f},\boldsymbol{\pi}}} 1,
\end{equation}
where $\boldsymbol{f}$ is a word of $p$ $(n/2)$-subsets of $\llbracket 1,n \rrbracket$, $\boldsymbol{\pi}$ is a word of $p$ permutations and
$$\bM^{bal}_{\boldsymbol{f},\boldsymbol{\pi}} := \{ \cM \in \bM_p \, : \, \cM \text{ is planar and $\boldsymbol{f},\boldsymbol{\pi}$ are constant on the white vertices of $\cM$}\}.$$
\end{theorem}
\begin{remark}
The above result can be easily generalized to the case where the dimensions of the Hilbert spaces are not identical: one has to replace the ``1'' in the right hand side of \eqref{eq:moment-formula-Wishart-permuted} with the appropriate product of constants $c_i$, where the dimensions of the Hilbert spaces scale as $N_i \sim c_i N$, with $N \to \infty$.
\end{remark}
\begin{proof}
The statement about freeness follows easily from the moment formula, using the correspondence between planar maps and non-crossing partitions and the fact that a map $\cM$ belongs to $\bM^{bal}_{\boldsymbol{f},\boldsymbol{\pi}}$ iff its corresponding non-crossing partition $\alpha$ satisfies $\alpha \leq (\ker \boldsymbol{f} \wedge \ker{\boldsymbol{\pi}})$ (see also the end of the proof of Theorem~\ref{thm:2-marginals-large-BC}). 

We now show the moment formula \eqref{eq:moment-formula-Wishart-permuted}. Note that we are actually in the setting of Theorem~\ref{thm:TheoremGen} and of Proposition~\ref{prop:LargeNinTheLarge1nRegime}, with the permutations $\fsig$ being given by $\sigma_a = \pi_{a+1}^{-1}\pi_a$, $\forall a \in \llbracket 1, p\rrbracket$.  Applying now Proposition~\ref{prop:LargeNinTheLarge1nRegime} (remember that all Hilbert space dimensions are $N$ in our current setting), we get
$$\lim_{N \to \infty} \bE N^{-n(p +1)/2  } \Tr W_{\boldsymbol{f},\boldsymbol{\pi}} = \sum_{\cM \in \bM_{\bff, \fsig}^{(0) \text{free}}} 1,$$
where we recall that (see \eqref{eq:FreeMaps2}) $\bM_{\bff, \fsig}^{(0) \text{free}} = \{\cM \, : \, \tilde L ( \bff, \fsig, \cM) = 0\}$, the functional $\tilde L$ being defined in \eqref{eq:def-tilde-L}. From Lemma~\ref{lemma:BlackVerticesForWPi}, we get that for all maps $\cM$, $V_\bullet(\boldsymbol{f},\boldsymbol{\sigma}) = k$. It is now enough to show that $\bM_{\bff, \fsig}^{(0) \text{free}} = \bM^{bal}_{\boldsymbol{f},\boldsymbol{\pi}}$. Let us start with the inclusion $\bM_{\bff, \fsig}^{(0) \text{free}} \subseteq \bM^{bal}_{\boldsymbol{f},\boldsymbol{\pi}}$. First, from $\Delta_{\bff, \fsig}(\cM) = 0$, we get that $\boldsymbol f$ is constant on the white vertices of $\cM$, see \eqref{eq:def-Delta}. The fact that $\boldsymbol \pi$ is also constant on the white vertices of $\cM$ follows form Lemma~\ref{lem:free-implies-pi-constant}, while the planarity of $\cM$ follows from Proposition~\ref{prop:ShorterFreeMap}. 

For the other inclusion, consider a map $\cM \in \bM^{bal}_{\boldsymbol{f},\boldsymbol{\pi}}$. Since $\boldsymbol f$ is constant on the white vertices of $\cM$, we have $\Delta_{\bff, \fsig}(\cM) = 0$. Also, recall from the proof of Lemma~\ref{lem:free-implies-pi-constant} that, for any edge $a$ and color $i \in J_a$, the black vertex to which $(a,i)$ is connected in $\cM_{\bff, \fsig}$ has the following incident edges: $\{(c,\pi_c^{-1}\pi_a(i))\}_{c \in \llbracket 1, p\rrbracket}$. In particular, using the fact that $\boldsymbol{\pi}$ is constant on the white vertices of $\cM$, all edges $(b,i)$ incident to the same white vertex as $(a,i)$ in $\cM_{\bff, \fsig}$ are connected to the same black vertex, and thus $\cM_{\bff, \fsig}$ has as many connected components as black vertices (\textit{i.e.}~$k$), so $\Sigma(\cM_{\bff, \fsig}) = 0$. Finally, from Lemma~\ref{lem:rCopiesMap}, the map $\cM_{\bff, \fsig}$ consists of $k$ copies of $\cM$, so that if $\cM$ is planar, $\cM_{\bff, \fsig}$ is too. This concludes the proof. 
\end{proof}
\begin{remark}
The result above has a simple interpretation in the case of two marginals $W_{f,\mathrm{id}}$ and $W_{\widehat f, \mathrm{id}}$, for complementary sets $f, \widehat f$: $f \sqcup \widehat f = \llbracket 1,n \rrbracket$. Indeed, for all $N$, these two marginals have exactly the same eigenvalues, and their eigenvector unitary operators are independent (this follows from the fact that they act on non-overlapping tensor factors). Free independence is in this case a consequence of Voiculescu's classical result \cite{voiculescu1990circular}. The discussion above generalizes easily to an arbitrary number of \emph{non-overlapping} marginals.
\end{remark}
  
  \smallskip
  
\subsubsection{\underline{Case where less than half the colors are traced}}
\label{subsub:LessHalfColorsTraced}

\medskip

\begin{proposition}
\label{prop:LargeNinTheLarge1nRegimeLessHalf}
In the large $N$ regime, if $ l + k > \frac n 2$ and with the previous notations,
 \begin{equation}
 \lim_{N\rightarrow\infty} \frac 1 {N^{n+(k+l)(p-2) + l   + V_\bullet(\bff, \fsig)-\mu}}\, \bE \Tr_{\fsig} W_{\bff}  = \sum_{\cM \in \bM_{\bff, \fsig}^{>}}\,  c^{\, V(\cM) - 1}, 
  \end{equation}
where 
	$$\mu=\min\{2lg(\cM) + \tilde L ( \bff, \fsig, \cM) + (2k+2l - n)(V(\cM) - 2) \, : \, \cM \in \bM_p\} \geq 0$$  
	and
	$$ \bM_{\bff, \fsig}^{>} := \{ \cM \in \bM_p \, : \, 2lg(\cM) + \tilde L ( \bff, \fsig, \cM) + (2k+2l - n)(V(\cM) - 2) = \mu \}.$$
\end{proposition}
\begin{proof} It follows from \eqref{eq:MomentsGenLargeRegime}, but let us detail the computations starting from Theorem~\ref{thm:TheoremGen}.
	Under the usual balanced scaling, the exponent of $N$ from Theorem~\ref{thm:TheoremGen} can be bounded as follows
		\begin{align*}
		l(2+p-&V(\cM)-2g(\cM)) + (n-r+1)(V(\cM)-1)+kp+V_\bullet(\boldsymbol{f},\boldsymbol{\sigma})+\\
		&\qquad\qquad\qquad\qquad+(r-l-2k-1)(V(\cM)-1) - \tilde L ( \bff, \fsig, \cM) \\
		&= (n-2l-2k)(V(\cM)-1) + (k+l)p+l-2lg(\cM)+ V_\bullet(\boldsymbol{f},\boldsymbol{\sigma})-\tilde L ( \bff, \fsig, \cM)\\
		&=  n+(k+l)(p-2)+l+ V_\bullet(\boldsymbol{f},\boldsymbol{\sigma}) -\\
		&\qquad\qquad\qquad\qquad- [2lg(\cM) + \tilde L ( \bff, \fsig, \cM) + (2k+2l - n)(V(\cM) - 2)]\\
		&\leq n+(k+l)(p-2)+l+ V_\bullet(\boldsymbol{f},\boldsymbol{\sigma})-\mu.
		\end{align*}
		The fact that $\mu \geq 0$ follows from the hypothesis $n-2l-2k<0$, and from the bounds $V(\cM)\geq 2$, $g(\cM)\geq 0$, $ \tilde L ( \bff, \fsig, \cM)\geq 0$.
\end{proof}

If all the edges have the same set of colors, the minimum $\mu$ is achieved by the unique planar map in $\bM_p$ having two vertices (one black and one white). However we stress that in the general case, this map does not have a vanishing $ \tilde L ( \bff, \fsig, \cM)\geq 0$, and therefore the minimization problem seems highly non-trivial.

Indeed, depending on the values taken by the coefficient of $V(\cM)-2$, two maps with different $g$, $\tilde L$, and $V$ might have the same minimal value of $2lg + \tilde L  + (2k+2l - n)(V - 2)$. We illustrate this with a very simple example in the case $l=1$ (color $A$), $k=2$ (colors $B,C$), and $n$ is kept as a parameter. We take $\sigma_a=\pi_{a+1}^{-1}\pi_a$, $\pi_1=(12)$ (the transposition) and $\pi_2=(1)(2)$ (the identity), so that both $\sigma_1$ and $\sigma_2$ are transpositions. In that case, 
$$
	\includegraphics[width=0.55\textwidth]{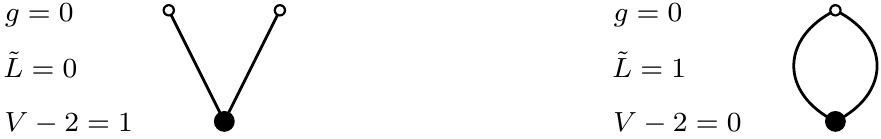}
$$
\begin{itemize}
    \item if $n=3$, $2k+2l - n = 3$, and the only map in $\bM_{\bff, \fsig}^{>}$ is that with two vertices,
    \item if $n=5,6$, $2k+2l - n = 1,0$ and the only map in $\bM_{\bff, \fsig}^{>}$ is the tree,
    \item if $n=4$, $2k+2l - n = 2$ and the two maps belong to $\bM_{\bff, \fsig}^{>}$.
    \end{itemize}

\subsection{The unbalanced asymptotical regime}

We consider in this section the asymptotical regime where the dimensions of the Hilbert spaces that ``move around'' are fixed; in this respect, the setting here generalized the $4$-partite situation considered in Section \ref{sec:2-marginals-fixed-BC}. In this case,  we fix $N_J = m$, and set $\forall i \in \llbracket 1,l \rrbracket, N_i=N$ and $\forall j \in \llbracket r,n\rrbracket, N_j \sim c^{\frac 1 {n-r+1}}N$ as $N\rightarrow \infty$. Theorem~\ref{thm:TheoremGen} writes
 \begin{align}
\bE \Tr_{\fsig} W_{\bff}  &= N^{pl + n-r+1}
 \sum_{\cM \in \bM_p} h(N)^{V(\cM)-1}  N^{ - 2lg(\cM) -  (l+r-n-1)(V(\cM)-2)} \\
  \nonumber&\hspace{4cm}\times m^{kp + V_\bullet(\bff, \fsig) + (r-l-2k-1)(V(\cM)-1) -\tilde L ( \bff, \fsig, \cM)},
 \end{align}
where $h$ is an arbitrary function whose limit at infinity is $c$. Again  we have different behaviors depending on the factor $l+r-n-1 = 2l+k-n$ in front of $V(\cM)-2$:
\begin{itemize}
    \item If $l+r-n-1<0$, the large $N$ limit is trivial, the only surviving map being the one with $p+1$ vertices (1 black and $p$ white). This situation was detailed in the beginning of Sec.~\ref{sec:BalancedGeneral}.
    \item If $l+r-n-1>0$, then $l >0$, and the large $N$ limit selects the only planar map with two vertices. The limit is also trivial. 
\end{itemize}

\

Besides the two trivial cases above, the only remaining situation is $l+r-n-1=0$ (or, equivalently $2l+k=n$) and we therefore assume it in the following.  In this case, Theorem~\ref{thm:TheoremGen} writes
 \begin{equation}
 \bE \Tr_{\fsig} W_{\bff}  = N^{l(p+1)}\sum_{\cM \in \bM_p}\, h(N)^{V(\cM)-1}   N^{ - 2lg(\cM)} m^{kp + V_\bullet(\bff, \fsig) + (r-l-2k-1)(V(\cM)-1) -\tilde L ( \bff, \fsig, \cM)}, 
 \end{equation}
where  $(r-l-2k-1)(V(\cM)-1) -\tilde L ( \bff, \fsig, \cM)$ is a generalization of $``\alt"$  which we defined in Def.~\ref{eq: Definition-Alt}, Section~\ref{sec:ABCD} for $n=4$. We recall that $\bM_p^{(0)}$ is the subset of planar elements of $\bM_p$. 
\begin{proposition}
\label{prop:lastreg}
In this regime, with the previous notations, 
 \begin{equation}
 \lim_{N\rightarrow\infty} \frac 1 {N^{l(p+1)}}\, \bE \Tr_{\fsig} W_{\bff}  = \sum_{\cM \in \bM_p^{(0)}}\, c^{\, V(\cM) - 1} m^{kp + V_\bullet(\bff, \fsig) + (r-l-2k-1)(V(\cM)-1) -\tilde L ( \bff, \fsig, \cM)}. 
 \end{equation}
\end{proposition}

Note that if for $a\in\llbracket 1, p \rrbracket$  we assume $V_\bullet(\bff, \fsig)=k$, and we define $\tilde W_{f(a)} = \frac 1{N^l m^k} W_{f(a)}$, as well as $\tilde c = cm^{r-l-2k-1}$, this rewrites as
 \begin{equation}
 \label{eq:lastreg}
 \lim_{N\rightarrow\infty} \frac 1 {N^lm^k}\, \bE \Tr_{\fsig} \tilde W_{\bff}  = \sum_{\cM \in \bM_p^{(0)}}\, \tilde c^{V(\cM) - 1} m^{ -\tilde L ( \bff, \fsig, \cM)}. 
 \end{equation}

\smallskip

\ 

\noindent {\bf An example for $n=5$.}
We now focus on the quantity $\tilde L ( \bff, \fsig, \cM)$, in the following interesting case of ``ABCDE": $n=5$, $l=1$ (color $A$), $r=5$ (color $E$ 
is always traced), and $k=2$ (the edges carry two colors among $\{B,C,D\}$); the colors $A$ and $E$ correspond to spaces of dimension growing to infinity, while the colors $B,C,D$ correspond to $\bC^m$. As $l+r - n - 1=0$, Prop.~\ref{prop:lastreg} and Eq.~\eqref{eq:lastreg} apply. The study of the properties of $\tilde L$ in this case shows that the corresponding cumulant functions do not factorize over the cycles of non-crossing partitions, but depend more subtly on these non-crossing partitions. In what follows, we study the special case of a map with two vertices (one black, one white), and we show, by the means of an example, that in the general case the factorization property does not hold.

To start, we recall the expression of $\tilde L ( \bff, \fsig, \cM)$
\begin{equation}
    \tilde L ( \bff, \fsig, \cM)=2(g(\cM_{\bff,\fsig})+\Delta_{\bff,\fsig}(\cM)+\Sigma(\cM_{\bff,\fsig})). 
\end{equation}
Any permutation $\sigma\in \cS_2$\footnote{In the following we omit color $A$ from the support and image of the $\sigma_a$'s as it is a common fixed point.  } defines a partial permutation on $\llbracket l+1, r-1\rrbracket = \llbracket 2, 4\rrbracket$, which we complete canonically to have a permutation $\Phi(\sigma)$ of $\llbracket 2, 4\rrbracket$. Indeed, there is only one missing color in the support and image of each permutation $\sigma_a$, so we just pair them (\textit{e.g.} if $\sigma$ is $B\rightarrow B,  C\rightarrow D$, then $\Phi(\sigma)$  is $B\rightarrow B,  C\rightarrow D, D\rightarrow C$, and if $\sigma$ is $B\rightarrow C,  C\rightarrow D$, then $\Phi(\sigma)$  is $B\rightarrow C,  C\rightarrow D, D\rightarrow B$). We recall that $\lvert\tilde\sigma\rvert$ is the length of $\tilde \sigma$, i.e.~its number of inversions.

\begin{proposition}\label{prop:L-tilde-V=2} In the $n=5$ example
described above, if $g(\cM)=0$, $V_\bullet(\bff, \fsig)=k=2$, and $V(\cM)=2$, then 
\be 
\label{eq:Gen-Alt-n6}
\tilde L ( \bff, \fsig, \cM) = \sum_{a=1}^p\, \lvert\Phi(\sigma_a)\rvert
\ee
\end{proposition}

\begin{proof}To begin, note that the white vertex of $\cM$ can be duplicated into $2$ or $3$ white vertices in the unfolded map $\cM_{\bff, \fsig}$, depending on whether the edges share the same set of colors or not; hence, $\Delta_{\bff, \fsig}$ is either $0$ or $1$. We first assume that $\Delta_{\bff, \fsig}=0$. This means that e.g.~$B$ is not in any of the color sets $J_a$, the $\sigma_a$ are either the identity, either the transposition, on $\{C,D\}$, and that $\Phi(\sigma_a)(B) = B$ for all $a$. We can assume that there are only transpositions in the corners, because when one has the identity permutation, one can collapse the two neighboring edges into a single one without changing the genus of $\cM$ or $\cM_{\bff, \fsig}$, nor $\Sigma(\cM_{\bff, \fsig})$. Here one can make a simple recurrence, which is quite similar to the one appearing in the proof of Proposition~\ref{prop:L-and-alt} in terms of maps: pick any $a$, and replace $\sigma_a$ and $\sigma_{a+1}$ by two identities. The number of inversions decreases by two. In $\cM_{\bff, \fsig}$, this amounts to ``flipping" the two edges, as shown in the figure below (the half-edges incident to the white vertices remain the same, but the endpoints incident to the black vertices are exchanged\footnote{Indeed changing $\sigma_a$ and $\sigma_{a+1}$ changes the cycles of $\Gamma_{\bff, \fsig}$ while keeping $\hat \alpha$ invariant}). The flip clearly decreases the genus or the number of connected components. To see this, note that we can decompose the flip in edge-deletions and  creations. The edges have (the same) two faces incident each. Delete one of them, and the other gets a single face visiting both sides. Delete the second edge: if it is a bridge the number of connected components increases (and $\Sigma_{\bff, \fsig}$ goes from $1$ to $0$) and the genus is unchanged, so $\tilde L$ decreases by two. If not, $\Sigma_{\bff, \fsig}$ is unchanged, and the genus decreases by one, so $\tilde L$ decreases by two, and we get the formula, after creating two new edges to complete the flip (creating the two new edges does not change the number of connected components nor the genus).
\begin{figure}[!ht]
\includegraphics[scale=1]{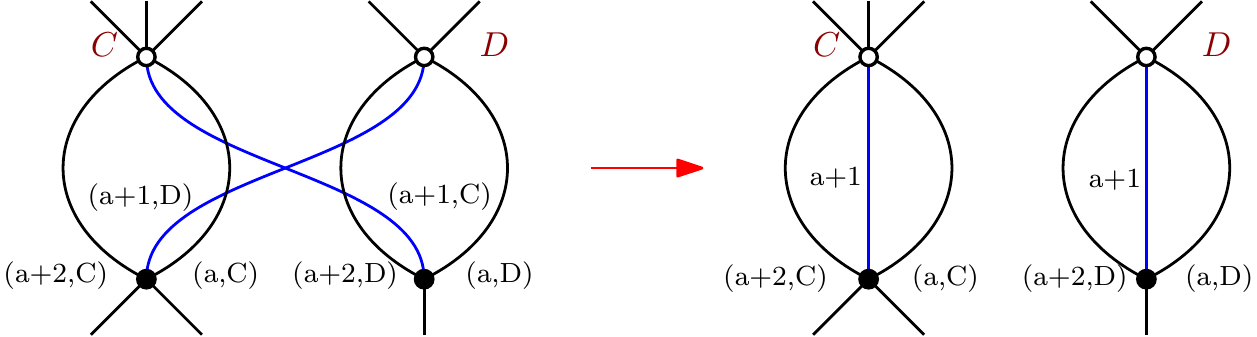} 
\end{figure}

We now assume that $\Delta_{\bff, \fsig}=1$. The strategy is to pick some $a$ such that $\sigma_{a+1}$ contains $B$ in its image (we can always find one, since $\Delta_{\bff, \fsig}\neq 0$), and replace $\sigma_a$ and $\sigma_{a+1}$ so that $B$ is a fixed point of $\sigma_{a+1}$. Then do the same for  $\sigma_{a-1}$ and $\sigma_{a}$, and so on.  At each step, one should be careful that the variation in the number of inversions is the same as $\tilde L$. In the end, $B$ is a common fixed point of all the $\sigma_a$, so we are in the $\Delta_{\bff, \fsig}=0$  case treated above, and we can conclude. There are two cases, depending whether $B$ is in the pre-image of $\sigma_{a+1}$ or not: 
\begin{enumerate}
 \item If $B$ is in the pre-image of $\sigma_{a+1}$ (and is not a fixed point of $\sigma_{a+1}$, in which case we go to the next step right away), we exchange it with the other color 
in the pre-image of $\sigma_{a+1}$, both in  $\sigma_a$ and $\sigma_{a+1}$:
\be 
\begin{tabular}{rcccl} &$\sigma_a$ & &$\sigma_{a+1}$& \\ $c_a$&$\mapsto$ & $c_{a+1}$ &$\mapsto$ & $B$ \\$c'_a$&$\mapsto$ & $B$ &$\mapsto$ & $c'_{a+2}$\end{tabular}
\raisebox{-1.5ex}{$\quad\leadsto \quad$}
\begin{tabular}{rcccl} & & && \\ $c_a$&$\mapsto$ & $B$ &$\mapsto$ & $B$ \\$c'_a$&$\mapsto$ & $c_{a+1}$ &$\mapsto$ & $c'_{a+2}$\end{tabular}.
\ee 
There are two sub-cases: either the number of inversions decreases by two, in which case we verify that $\tilde L$ decreases by two (6 possibilities), or the number of inversions is constant,  in which case we  verify that $\tilde L$ is also constant (4 possibilities).\\

 \item If $B$ is not in the pre-image of $\sigma_{a+1}$, we just exchange the pre-image of $B$ with $B$, and leave the other color in the pre-image of $\sigma_{a+1}$ untouched. Again, either the number of inversions decreases by two, in which case we verify that $\tilde L$ decreases by two (12 possibilities), or the number of inversions is constant,  in which case we  verify that $\tilde L$ is also constant (8 possibilities). 
\end{enumerate}
 This concludes the proof
\end{proof}

\begin{figure}[!ht]
    \centering
    \includegraphics[scale=0.75]{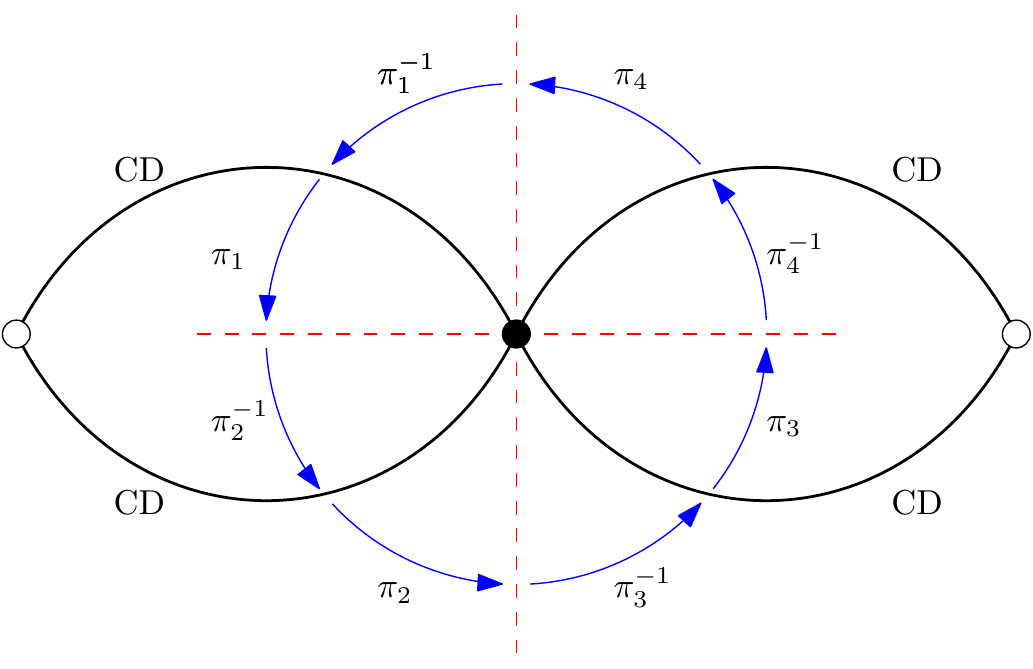}
    \caption{For $\pi_1=(1)(2)$, $\pi_2=(12)$, $\pi_3=(1)(2)$, $\pi_4=(12)$, one has $\tilde L(\bff,\fsig, \cM)=2$, while the value of $\tilde L$ for both submaps with one white vertex is $2$, so the sum over the one white vertex submaps of $\tilde L$ does not match $\tilde L(\bff,\fsig, \cM)$. However, if $\pi_1=(1)(2)$, $\pi_2=(1)(2)$, $\pi_3=(1)(2)$, $\pi_4=(1)(2)$, then $\tilde L(\bff,\fsig, \cM)=0$ and the one white vertex submaps also have $\tilde L=0$. This behavior indicates that there is no factorization over the cycles of $\alpha$.}
    \label{fig:non_factorization_example}
\end{figure}
When trying to generalize the formula of the result above to $V(\cM)>2$, we face the following difficulty: around a white vertex, there is at least one corner for which we do not have a $\sigma$ (the white corner which does not face a black corner). On the example of Fig.~\ref{fig:non_factorization_example}, it is easily seen that the value of $\tilde L$ for the full map is not given as the sum of the value $\tilde L$ for the two submaps with one white vertex. This indicates that there is no factorization of $\tilde L$ over the cycles of $\alpha$. It is also possible to display a formula for the case of maps whose edges are only of color $CD$. This formula shows that there is no factorization over the cycles of $\alpha$. Thus it is not possible to write the free cumulants in the context of (scalar) free probability. It seems that the right framework needed to tackle this unbalanced scaling is the one of free probability with amalgamation over an algebra $\cB$ that is contained in  $\un_A\otimes \mathcal{M}_m(\bC)\otimes \mathcal{M}_m(\bC)\supset \cB$. In order to limit the length of this paper, we postpone the exploration and the exposition of such results to future work.

\bibliographystyle{alpha}
\bibliography{tensor-QIT}

\begin{thebibliography}{PGVWC07}

\bibitem[AGZ10]{anderson2010introduction}
Greg~W Anderson, Alice Guionnet, and Ofer Zeitouni.
\newblock {\em An introduction to random matrices}.
\newblock Cambridge University Press, 2010.

\bibitem[AHH12]{ambainis2012random}
Andris Ambainis, Aram~W Harrow, and Matthew~B Hastings.
\newblock Random tensor theory: extending random matrix theory to mixtures of
  random product states.
\newblock {\em Communications in Mathematical Physics}, 310(1):25--74, 2012.

\bibitem[AS17]{aubrun2017alice}
Guillaume Aubrun and Stanis{\l}aw~J Szarek.
\newblock {\em Alice and Bob Meet Banach: The Interface of Asymptotic Geometric
  Analysis and Quantum Information Theory}, volume 223.
\newblock American Mathematical Soc., 2017.

\bibitem[BBCC11]{banica2011free}
Teodor Banica, ST~Belinschi, Mireille Capitaine, and Benoit Collins.
\newblock Free bessel laws.
\newblock {\em Canadian Journal of Mathematics}, 63:3--37, 2011.

\bibitem[BG17]{Gaetan-Elba}
G.~{Borot} and E.~{Garcia-Failde}.
\newblock {Simple maps, Hurwitz numbers, and Topological Recursion}.
\newblock {\em ArXiv e-prints}, October 2017.

\bibitem[Bia97]{biane1997some}
Philippe Biane.
\newblock Some properties of crossings and partitions.
\newblock {\em Discrete Mathematics}, 175(1):41--53, 1997.

\bibitem[BS10]{bai2010spectral}
Zhidong Bai and Jack~W Silverstein.
\newblock {\em Spectral analysis of large dimensional random matrices},
  volume~20.
\newblock Springer, 2010.

\bibitem[BSY88]{bai1988note}
Zhidong~D Bai, Jack~W Silverstein, and Yong~Q Yin.
\newblock A note on the largest eigenvalue of a large dimensional sample
  covariance matrix.
\newblock {\em Journal of Multivariate Analysis}, 26(2):166--168, 1988.

\bibitem[CC15]{Carrell15}
Sean~R. Carrell and Guillaume Chapuy.
\newblock Simple recurrence formulas to count maps on orientable surfaces.
\newblock {\em Journal of Combinatorial Theory, Series A}, 133:58 -- 75, 2015.

\bibitem[CDKW14]{christandl2014eigenvalue}
Matthias Christandl, Brent Doran, Stavros Kousidis, and Michael Walter.
\newblock Eigenvalue distributions of reduced density matrices.
\newblock {\em Communications in mathematical physics}, 332(1):1--52, 2014.

\bibitem[CFF13]{Chapuy13}
Guillaume Chapuy, Valentin F\'eray, and \'Eric Fusy.
\newblock A simple model of trees for unicellular maps.
\newblock {\em Journal of Combinatorial Theory, Series A}, 120(8):2064 -- 2092,
  2013.

\bibitem[Cha10]{Chapuy2010}
Guillaume Chapuy.
\newblock The structure of unicellular maps, and a connection between maps of
  positive genus and planar labelled trees.
\newblock {\em Probability Theory and Related Fields}, 147(3):415--447, Jul
  2010.

\bibitem[Cha11]{Chapuy11}
Guillaume Chapuy.
\newblock A new combinatorial identity for unicellular maps, via a direct
  bijective approach.
\newblock {\em Advances in Applied Mathematics}, 47(4):874 -- 893, 2011.

\bibitem[CMSS07]{collins2007second}
Beno{\i}t Collins, James~A Mingo, Piotr Sniady, and Roland Speicher.
\newblock Second order freeness and fluctuations of random matrices. {III}.
  higher order freeness and free cumulants.
\newblock {\em Doc. Math}, 12:1--70, 2007.

\bibitem[CN11]{collins2011gaussianization}
Beno{\^\i}t Collins and Ion Nechita.
\newblock Gaussianization and eigenvalue statistics for random quantum channels
  ({III}).
\newblock {\em The Annals of Applied Probability}, pages 1136--1179, 2011.

\bibitem[CN16]{collins2016random}
Benoit Collins and Ion Nechita.
\newblock Random matrix techniques in quantum information theory.
\newblock {\em Journal of Mathematical Physics}, 57(1), 2016.

\bibitem[CN{\.Z}10]{collins2010randoma}
Beno{\^\i}t Collins, Ion Nechita, and Karol {\.Z}yczkowski.
\newblock Random graph states, maximal flow and fuss--catalan distributions.
\newblock {\em Journal of Physics A: Mathematical and Theoretical},
  43(27):275303, 2010.

\bibitem[CN{\.Z}13]{collins2013area}
Beno{\^\i}t Collins, Ion Nechita, and Karol {\.Z}yczkowski.
\newblock Area law for random graph states.
\newblock {\em Journal of Physics A: Mathematical and Theoretical},
  46(30):305302, 2013.

\bibitem[DC14]{dupic2014spectral}
Thomas Dupic and Isaac~P{\'e}rez Castillo.
\newblock Spectral density of products of wishart dilute random matrices. part
  i: the dense case.
\newblock {\em arXiv preprint arXiv:1401.7802}, 2014.

\bibitem[Fol13]{folland2013real}
Gerald~B Folland.
\newblock {\em Real analysis: modern techniques and their applications}.
\newblock John Wiley \&amp; Sons, 2013.

\bibitem[GS98]{Goupil98}
A.~Goupil and Gilles Schaeffer.
\newblock Factoring n-cycles and counting maps of given genus.
\newblock {\em Eur. J. Comb.}, 19(7):819--834, October 1998.

\bibitem[Hal98]{hall1998random}
Michael~JW Hall.
\newblock Random quantum correlations and density operator distributions.
\newblock {\em Physics Letters A}, 242(3):123--129, 1998.

\bibitem[HP00]{hiai2000semicircle}
Fumio Hiai and D\'enes Petz.
\newblock {\em The semicircle law, free random variables, and entropy},
  volume~77 of {\em Mathematical Surveys and Monographs}.
\newblock American Mathematical Society, 2000.

\bibitem[HPT64]{harary1964number}
Frank Harary, Geert Prins, and WT~Tutte.
\newblock The number of plane trees.
\newblock In {\em Indagationes Mathematicae (Proceedings)}, volume~67, pages
  319--329. Elsevier, 1964.

\bibitem[HT03]{haagerup2003random}
Uffe Haagerup and Steen Thorbj{\o}rnsen.
\newblock Random matrices with complex gaussian entries.
\newblock {\em Expositiones Mathematicae}, 21(4):293--337, 2003.

\bibitem[Iss18]{isserlis1918formula}
Leon Isserlis.
\newblock On a formula for the product-moment coefficient of any order of a
  normal frequency distribution in any number of variables.
\newblock {\em Biometrika}, 12(1/2):134--139, 1918.

\bibitem[Kre72]{kreweras1972partitions}
Germain Kreweras.
\newblock Sur les partitions non crois{\'e}es d'un cycle.
\newblock {\em Discrete Mathematics}, 1(4):333--350, 1972.

\bibitem[Luk70]{Lukacs_book}
Eugene Lukacs.
\newblock {\em Characteristic Functions}.
\newblock Griffin, 1970.

\bibitem[Meh04]{mehta2004random}
Madan~Lal Mehta.
\newblock {\em Random matrices}, volume 142.
\newblock Academic press, 2004.

\bibitem[MP67]{marcenko1967distribution}
Vladimir~A Mar{\v{c}}enko and Leonid~Andreevich Pastur.
\newblock Distribution of eigenvalues for some sets of random matrices.
\newblock {\em Sbornik: Mathematics}, 1(4):457--483, 1967.

\bibitem[MT01]{GraphsOnSurfaces}
Bojan Mohar and Carsten Thomassen.
\newblock {\em Graphs on Surfaces}.
\newblock Johns Hopkins University Press, 2001.

\bibitem[NC10]{nielsen2010quantum}
Michael~A Nielsen and Isaac~L Chuang.
\newblock {\em Quantum computation and quantum information}.
\newblock Cambridge university press, 2010.

\bibitem[Nec07]{nechita2007asymptotics}
Ion Nechita.
\newblock Asymptotics of random density matrices.
\newblock {\em Annales Henri Poincar{\'e}}, 8(8):1521--1538, 2007.

\bibitem[Nov14]{novak2014three}
Jonathan Novak.
\newblock Three lectures on free probability.
\newblock {\em Random matrix theory, interacting particle systems, and
  integrable systems}, 65:309--383, 2014.

\bibitem[NS06]{nica2006lectures}
Alexandru Nica and Roland Speicher.
\newblock {\em Lectures on the combinatorics of free probability}, volume~13.
\newblock Cambridge University Press, 2006.

\bibitem[OS{\.Z}10]{osipov2010random}
V~Osipov, Hans-Juergen Sommers, and K~{\.Z}yczkowski.
\newblock Random bures mixed states and the distribution of their purity.
\newblock {\em Journal of Physics A: Mathematical and Theoretical},
  43(5):055302, 2010.

\bibitem[Pag93]{page1993average}
Don~N Page.
\newblock Average entropy of a subsystem.
\newblock {\em Physical review letters}, 71(9):1291, 1993.

\bibitem[PGVWC07]{perez-garcia2006matrix}
David Perez-Garcia, Frank Verstraete, Michael~M Wolf, and J~Ignacio Cirac.
\newblock Matrix product state representations.
\newblock {\em Quantum Inf. Comput.}, 7(5):401, 2007.

\bibitem[P{\.Z}11]{penson2011product}
Karol~A Penson and Karol {\.Z}yczkowski.
\newblock Product of ginibre matrices: Fuss-catalan and raney distributions.
\newblock {\em Physical Review E}, 83(6):061118, 2011.

\bibitem[Slo16]{oeis}
N.~J.~A. Sloane.
\newblock The on-line encyclopedia of integer sequences, published
  electronically at \href{https://oeis.org}{https://oeis.org}, 2016.

\bibitem[SZ03]{sommers2003bures}
Hans-J{\"u}rgen Sommers and Karol Zyczkowski.
\newblock Bures volume of the set of mixed quantum states.
\newblock {\em Journal of Physics A: Mathematical and General}, 36(39):10083,
  2003.

\bibitem[S{\.Z}04]{sommers2004statistical}
Hans-J{\"u}rgen Sommers and Karol {\.Z}yczkowski.
\newblock Statistical properties of random density matrices.
\newblock {\em Journal of Physics A: Mathematical and General}, 37(35):8457,
  2004.

\bibitem[Tut75]{TutteTrinity}
W.~T. Tutte.
\newblock Duality and trinity.
\newblock {\em Colloquium Mathematical Society Janos Bolyai}, 10:1459?1472,
  1975.

\bibitem[Voi90]{voiculescu1990circular}
DV~Voiculescu.
\newblock Circular and semicircular systems and free product.
\newblock {\em Progr. Math.}, 92:45--60, 1990.

\bibitem[Wig55]{wigner1955characteristic}
Eugene~P Wigner.
\newblock Characteristic vectors of bordered matrices with infinite dimensions.
\newblock {\em Annals of Mathematics}, pages 548--564, 1955.

\bibitem[Wis28]{wishart1928generalised}
John Wishart.
\newblock The generalised product moment distribution in samples from a normal
  multivariate population.
\newblock {\em Biometrika}, pages 32--52, 1928.

\bibitem[WL72]{walsh1972counting}
Timothy~RS Walsh and Alfred~B Lehman.
\newblock Counting rooted maps by genus. i.
\newblock {\em Journal of Combinatorial Theory, Series B}, 13(3):192--218,
  1972.

\bibitem[Zag86]{Zagier1986}
Harer~J. Zagier, D.
\newblock The euler characteristic of the moduli space of curves.
\newblock {\em Inventiones mathematicae}, 85:457--486, 1986.

\bibitem[{\.Z}S01]{zyczkowski2001induced}
Karol {\.Z}yczkowski and Hans-J{\"u}rgen Sommers.
\newblock Induced measures in the space of mixed quantum states.
\newblock {\em Journal of Physics A: Mathematical and General}, 34(35):7111,
  2001.

\bibitem[{\.Z}S03]{zyczkowski2003hilbert}
Karol {\.Z}yczkowski and Hans-J{\"u}rgen Sommers.
\newblock Hilbert-schmidt volume of the set of mixed quantum states.
\newblock {\em Journal of Physics A: Mathematical and General}, 36(39):10115,
  2003.

\end{thebibliography}

\end{document}